\documentclass[letterpaper,11pt,reqno,dvipsnames]{amsart}
\usepackage[portrait,margin=3cm]{geometry}
\usepackage{hyperref}
\usepackage[foot]{amsaddr}
\usepackage{amssymb,amsthm}
\usepackage[textsize=small]{todonotes}
\usepackage{graphicx}
\usepackage[numeric,initials,nobysame]{amsrefs}
\usepackage{upref,setspace}
\usepackage{xcolor,colortbl}
\usepackage{enumerate}
\usepackage{bbm}
\usepackage{float}
\usepackage[mathscr]{eucal}
\usepackage{graphicx}
\usepackage{sidecap}
\usepackage{latexsym}
\usepackage{psfrag}
\usepackage{epsfig}
\usepackage{enumerate}
\usepackage{amsmath}
\usepackage{amsthm}
\usepackage{amssymb}
\usepackage{pdfsync}
\usepackage{fourier}
\usepackage{multirow}
\usepackage{comment}
\usepackage{amsfonts}
\usepackage{color}
\usepackage[textsize=small]{todonotes}
\usepackage{hyperref}

\newif\iffinal
\finalfalse	
\RequirePackage[utf8]{inputenc}
\numberwithin{equation}{section}
\newtheorem{theorem}{\bf Theorem}[section]
\newtheorem{definition}[theorem]{\bf Definition}

\newtheorem{lemma}[theorem]{\bf Lemma}

\newtheorem{condition}[theorem]{\bf Condition}
\newtheorem{proposition}[theorem]{\bf Proposition}

\theoremstyle{remark}
\newtheorem{example}[theorem]{\bf Example}
\newtheorem{remark}[theorem]{\bf Remark}
\newcommand{\NJ}{\mathbb{N}_J}
\newcommand{\NI}{\mathbb{N}_I}
\newcommand{\NN}{\mathbb{N}}
\newcommand{\RR}{\mathbb{R}}
\newcommand{\one}{\mathbf{1}}
\newcommand{\cle}{\mathcal{E}}
\newcommand{\clf}{\mathcal{F}}
\newcommand{\clg}{\mathcal{G}}
\newcommand{\clp}{\mathcal{P}}
\newcommand{\clc}{\mathcal{C}}
\newcommand{\clq}{\mathcal{Q}}
\newcommand{\cls}{\mathcal{S}}
\newcommand{\clo}{\mathcal{O}}
\newcommand{\clm}{\mathcal{M}}
\newcommand{\ch}{\check}
\newcommand{\Om}{\Omega}
\newcommand{\om}{\omega}
\newcommand{\bw}{\mathbf{w}}
\newcommand{\bz}{\mathbf{z}}
\newcommand{\bm}{\mathbf{m}}
\newcommand{\eps}{\epsilon}
\begin{document}

\title[Control Policies for HGI Performance]{Control Policies Approaching HGI Performance in Heavy Traffic for Resource Sharing Networks}

\date{}

\author[Budhiraja]{Amarjit Budhiraja$^1$}
\author[Johnson]{Dane Johnson$^1$}
\address{$^1$Department of Statistics and Operations Research, 304 Hanes Hall, University of North Carolina, Chapel Hill, NC 27599}
\email{budhiraj@email.unc.edu, danedane@email.unc.edu}
\maketitle
\begin{abstract}
We consider resource sharing networks of the form introduced in the work of Massouli\'{e} and Roberts(2000) as models for Internet flows. The goal is to study the open problem, formulated in Harrison et al. (2014),
of constructing simple form rate allocation policies for broad families of resource sharing networks with associated costs converging to the Hierarchical Greedy Ideal performance in the heavy traffic limit.
We consider two types of cost criteria,  an infinite horizon discounted cost,   and  a long time average cost per unit time.  We introduce a sequence of rate allocation control policies that are determined in terms of certain thresholds for the scaled queue length processes and prove 
that, under conditions, both type of  costs  associated with these policies converge in the heavy traffic limit to the corresponding HGI performance. The conditions needed for these results are satisfied  by all the examples considered in Harrison et al. (2014).
\end{abstract}

	\section{Introduction}
	
In \cite{harmandhayan}	the authors have formulated an interesting and challenging open problem for resource sharing networks that were introduced in the work of Massouli\'{e} and Roberts \cite{masrob} as models for Internet flows. A typical network of interest consists of $I$ resources (labeled $1, \ldots , I$) with associated capacities $C_i$, $i = 1, \ldots , I$.
Jobs of type $1, \ldots , J$ arrive according to independent Poisson processes with rates depending on the job-type and the job-sizes of different job-type are exponentially distributed with parameters once more depending on the type. Usual assumptions on mutual independence are made. The processing of a job is accomplished by allocating a {\em flow rate}
to it over time and a job departs from the system when the integrated flow rate equals the size of the job. A typical job-type requires simultaneous processing by several resources in the network. This relationship between job-types and resources is described through a $I\times J$ incidence matrix $K$ for which $K_{ij}= 1$ if $j$-th job-type requires processing by resource $i$ and $K_{ij}=0$  otherwise. Denoting by $x = (x_1, \ldots , x_J)'$ the vector of flow rates allocated to various job-types at any given time instant, $x$ must satisfy the capcity constraint $Kx \le C$, where $C= (C_1, \ldots , C_I)'$.

One of the basic problems for such networks is to construct ``good'' dynamic control policies that allocate resource capacities to jobs in the system. A ``good'' performance is usually quantified in terms of an appropriate cost function. One can formulate an optimal stochastic control problem using such a cost function, however in general such control problems are intractable and therefore one considers an asymptotic formulation under a suitable scaling.
The paper \cite{harmandhayan} formulates a Brownian control problem (BCP) that formally approximates the system manager's control under heavy traffic conditions. Since finding optimal solutions of such general Brownian control problems and constructing asymptotically optimal control policies for the network based on such solutions is a notoriously hard problem, the paper
\cite{harmandhayan} proposes a different approach in which the goal is not to seek an asymptotically optimal solution for the network but rather control policies that achieve the so called {\em Hierarchical Greedy Ideal} (HGI) performance in the heavy traffic limit. Formally speaking, HGI performance is the cost associated with a control in the BCP (which is in general sub-optimal), under which (I) no resource's capacity is underutilized when there is work for that resource in the system,  
and (II) the total number of jobs of each type at any given instant is the minimum consistent with the vector of workloads for the various resources. Desirability of such control policies has been argued in great detail in \cite{harmandhayan}
through simulation and numerical examples and will not be revisited here.

The main open problem formulated in \cite{harmandhayan} is to construct simple form rate allocation policies for broad families of resource sharing networks with associated costs converging to the HGI performance determined from the corresponding BCP. The goal of this work is to make progress on this open problem. We consider two types of cost criteria, the first is an infinite horizon discounted cost (see \eqref{eq: costdisc})  and the second is a long time average cost per unit time (see \eqref{eq: costerg}). In particular the second cost criterion is analogous to the cost function considered in \cite{harmandhayan}. We introduce a sequence of rate allocation control policies that are determined in terms of certain thresholds for the scaled queue length processes and prove in Theorems \ref{thm:thm6.5} and \ref{thm:thm6.5disc}
that, under conditions, the costs \eqref{eq: costerg} and \eqref{eq: costdisc} associated with these policies converge in the heavy traffic limit to the corresponding HGI performance. 

We now comment on the conditions that are used in establishing the above results. The first main condition (Condition \ref{cond:loctrafcond}) we need is the existence of {\em local traffic} on each resource, namely for each resource $i$ there is a unique job type that only uses resource $i$. This basic condition, first introduced in \cite{kankelleewil}, is also a key assumption in \cite{harmandhayan} and is needed in order to ensure that the state space of the {\em workload process} is all of the positive orthant (see Section \ref{sec:hgi} for  a discussion of this point).  Our second condition (Condition \ref{cond:HT1})
is a standard heavy traffic condition and a stability condition for diffusion scaled workload processes. The stability condition will be key in Section \ref{sec:unifmom} when establishing moment bounds that are uniform in time and scaling parameter.
We now describe the final main condition used in this work. In Section \ref{sec:polandmaires}  we will see that the collection of all job-types can be decomposed into the so called {\em primary} jobs and {\em secondary } jobs. Primary jobs are those with `high' holding cost and intuitively are the ones we want to process first. It will also be seen in Section \ref{sec:polandmaires}  that the collection $\cls^1$ of all job-types that only require processing from a single resource is contained in the collection $\cls^s$ of all secondary jobs. Our third main condition, formulated as Condition \ref{cond:viableRankExists},
says that there is a {\em ranking} of all job-types in $\cls^m \doteq \cls^s\setminus \cls^1$. A precise notion of a ranking is given in Definition \ref{def:viableRank}, but roughly speaking, the job-types with larger rank value will get higher `attention' in a certain sense under our proposed policy. We note that the ranking is given through a deterministic map that only depends on system parameters and not on the state of the system.
The condition is somewhat nontransparent and notationally  cumbersome  and  so we provide two sufficient conditions in Theorems \ref{thm:SuffSubsetJobs} and \ref{thm:SuffRemoveJobCond} for Condition \ref{cond:viableRankExists} to hold. We also discuss in Remarks \ref{rem:firsthm} and \ref{rem:secthm} some examples where one of these sufficient conditions holds. In particular, all the examples in \cite{harmandhayan} (2LLN, 3LLN, C3LN, and the negative example of Section 13 therein) satisfy Condition \ref{cond:viableRankExists}. Furthermore, there are many other networks not covered by Theorems \ref{thm:SuffSubsetJobs} and \ref{thm:SuffRemoveJobCond} where Condition \ref{cond:viableRankExists} is satisfied and in Example \ref{exam:out} we provide one such example. Finally, it is not hard to construct examples where Condition \ref{cond:viableRankExists} fails and in Example \ref{exam:outout} we give  such an example. Construction of simple form rate allocation policies that achieve HGI performance in the heavy traffic limit for general families of  models as in Example \ref{exam:outout}
remains a challenging open problem. We expect that suitable notions of state dependent ranking maps will be needed in order to use the ideas developed in the current work for treating such models, however the proofs and constructions are expected to be substantially more involved.

Our rate allocation policy is introduced in Definition \ref{def:workAllocScheme}. Implementation of this policy requires first determining the collection of secondary jobs. This step, using the definition in \eqref{eq:primjobdef}, can be completed easily by solving a finite collection of linear programming problems. The next step is to determine a {\em viable ranking} (if it exists) of all jobs in $\cls^m$. In general when $\cls^m$ is very large, determining this ranking may be a numerically hard problem, however as discussed in Section \ref{sec:simsuffcond}, for many examples this ranking can be given explicitly in a simple manner. Once a ranking is determined, the policy in Definition \ref{def:workAllocScheme} is explicit given in terms of arbitrary positive constants $c_1, c_2$ with $c_1<c_2$ and $\alpha \in (0, 1/2)$. 
	Roughly speaking, our approach is applicable to systems where job-types have a certain ordering of ``urgency'' in the sense that, regardless of the particular workload, we want as much of it as possible to come from the least urgent job types.    
	A second concern that needs to be addressed is that a resource should work at `near' full capacity when there is `non negligible' amount of work for it.
A detailed discussion of how the  proposed policy achieves these goals is given in Remark \ref{rem:poldisc} where we also comment on connection between this policy and the UFO policies proposed in \cite{harmandhayan}.
	
We now comment on the proofs of our main results, Theorems \ref{thm:thm6.5} and \ref{thm:thm6.5disc}. 
Both results rely on large deviation probability estimates and stopping time constructions of the form introduced first in the works of Bell and Williams \cite{belwil1, belwil2} (see also \cite{budgho1} and \cite{atakum}).  A key result is Theorem \ref{thm:discCostInefBnd} which relates the cost under our policy with the workload cost function $\clc$ in \eqref{eq:eq942}. This estimate is crucial in achieving property (II) of the HGI asymptotically. Asymptotic achievement of property (I) of HGI is a consequence of Theorem \ref{thm:finTimeConvToRBM}, the estimate in \eqref{eq:eqworkldfi} and continuity properties of the Skorohod map. 
Proof of Theorem \ref{thm:thm6.5} requires additional moment estimates that are uniform in time and the scaling parameter (see Section \ref{sec:unifmom}). A key such estimate is given in Theorem \ref{thm:WmomBnd}, 
the proof of which relies on the construction of a suitable Lyapunov function (see Proposition \ref{thm:timeDecayV}). Once uniform moment bounds are available, one can argue tightness of certain path occupation measures (see Theorem \ref{thm:occMeasTight}) and characterize their limit points in a suitable manner (see Theorem \ref{thm:limitMeasProp}). Desired cost convergence then follows readily by appealing to continuous mapping theorem and uniform integrability estimates.

The paper is organized as follows. In Section \ref{sec:backg} we introduce the state dynamics, cost functions of interest, and two of our main conditions. Section \ref{sec:hgi} gives the precise definition of Hierarchical Greedy Ideal Performance in terms of certain costs associated with $I$ dimensional reflected Brownian motions.  In Section \ref{sec:polandmaires} we introduce our final key condition (Condition \ref{cond:viableRankExists}), present our dynamic rate allocation policy, and give our two main convergence results: Theorems \ref{thm:thm6.5} and \ref{thm:thm6.5disc}.  Section \ref{sec:example} discusses Condition \ref{cond:viableRankExists} and presents some sufficient conditions for it to be satisfied. This section also gives an example where the condition fails to hold.
Sections \ref{sec:secworkcost} - \ref{sec:pathoccmzr} form the technical heart of this work. Section \ref{sec:secworkcost}
proves some useful properties of the workload cost function $\clc(\cdot)$ introduced in \eqref{eq:eq942} and Section \ref{sec:ratallpol} studies some important structural properties of our proposed rate allocation policy. Section \ref{sec:ldest} is  technically the most demanding part of this work. It provides some key estimates on costs under our scheme in terms of the workload cost function and establishes certain moment estimates that are  uniform in time and the scaling parameter.  In Section \ref{sec:pathoccmzr} we introduce certain path occupation measures,  prove their tightness,  and characterize the limit points. Finally Section \ref{sec:pfsmainthms} completes the proof our two main results. An appendix contains some standard large deviation estimates for Poisson processes.

	The following notation will be used.
	 For a Polish space $\mathbb{S}$, denote
	the corresponding Borel $\sigma$-field by $\mathcal{B}(\mathbb{S})$. 
	Denote by $\mathcal{P}(\mathbb{S})$ (resp.\ $\mathcal{M}(\mathbb{S})$) the
	space of probability measures (resp. finite measures) on $\mathbb{S}$,
	equipped with the topology of weak convergence. 
	For $f
	\colon \mathbb{S} \to \mathbb{R}$, let $\|f\|_\infty \doteq \sup_{x \in 
	\mathbb{S}} |f(x)|$. 
	For a Polish space $\mathbb{S}$ and $T>0$, denote by $C([0,T]:\mathbb{S})$
	(resp.\ $D([0,T]:\mathbb{S})$) the space of continuous functions
	(resp.\ right continuous functions with left limits) from $[0,T]$ to $%
	\mathbb{S}$, endowed with the uniform topology (resp.\ Skorokhod topology). 
	We say a collection $\{ X^n \}$ of $\mathbb{S}$-valued random variables is
	tight if the distributions of $X^n$ are tight in $\mathcal{P}(\mathbb{S})$.
	Equalities and inequalities involving vectors are
	interpreted component-wise.

\section{General Background}

\label{sec:backg}

Assume there are $J$ types of jobs and $I$ resources for processing them. \
The network is described through the  $I\times J$ matrix $K$ that has
entries $K_{ij}=1$ if resource $i$ works on job type $j$, and $K_{i,j}=0$
otherwise. We will assume (for simplicity) that no two columns
of $K$ are identical, namely, given a subset of resources, there is at most one job-type that has this subset as the associated set of resources.
 Given $m \in \mathbb{N}$, we let $\mathbb{N}_m
\doteq \{1, 2, \ldots m\}$. In particular, $\mathbb{N}_I=\{1, \ldots I\}$
and $\mathbb{N}_J=\{1, \ldots J\}$.

Denote by $N_j$ the set of resources that work on type $j$ jobs, i.e. 
\begin{equation*}
N_j \doteq \{i \in \mathbb{N}_I: K_{i,j}=1\}.
\end{equation*}
Let $\cls^1$ be the collection of all job types that use only one resource.
I.e. 
\begin{equation*}
\cls^1 \doteq \{j \in \mathbb{N}_J: {\mathbf{1}}^T Ke_j = \sum_{i=1}^I K_{ij} =
1\},
\end{equation*}
where $e_j$ is the unit vector in $\mathbb{R}^J$ with $1$ in the $j$-th
coordinate and $\mathbf{1}$ is the $I$-dimensional vector of ones.
Throughout we assume that for every resource there is a unique job type that
only uses that resource, namely the following condition is satisfied.
\begin{condition}
	\label{cond:loctrafcond}
$\bigcup_{j \in \cls^1} N_j = \mathbb{N}_I	$
\end{condition}
We denote the unique job-type that uses only resource $i$
as $\check j(i)$. Similarly for $j \in \cls^1$, we denote by $\hat i(j)$ the
unique resource that processes this job-type.

The capacity for resource $i$ is given by $C_{i}$. Let $\{\eta
_{j}^{r}(k)\}_{k=1}^{\infty }$ be the i.i.d. inter-arrival times for the $j$%
-th job type and let $\{\Delta _{j}^{r}(k)\}_{k=1}^{\infty }$ be the
associated i.i.d. amounts of work for the $j$-th job type. If at a given
instant work of type $j$ is processed at rate $x_{j}$ then the capacity
constraint requires that  $C\geq Kx$.  We assume the $\{\eta _{j}^{r}(k)\}_{k=1}^{\infty }$ are
exponentially distributed with rates $\lambda _{j}^{r}$ and the $\{\Delta
_{j}^{r}(k)\}_{k=1}^{\infty }$ are exponentially distributed with rates $\mu
_{j}^{r}$. \ Define Poisson processes 
\begin{equation*}
A_{j}^{r}(t)=\max \left\{ k:\sum_{i=1}^{k}\eta _{j}^{r}(i)\leq t\right\},\; 
S_{j}^{r}(t)=\max \left\{ k:\sum_{i=1}^{k}\Delta _{j}^{r}(i)\leq t\right\} 
\text{.}
\end{equation*}%
\ Let $\varrho_{j}^{r}=\frac{\lambda _{j}^{r}}{\mu _{j}^{r}}$ and $\varrho^r
\doteq (\varrho_j^r)_{j=1}^J$. The following will be our main heavy traffic
condition. The requirement $v^*>0$ will ensure the stability of the reflected Brownian motion in \eqref{eq:eqrbm}
and will be a key ingredient for uniform moment estimates in Section \ref{sec:unifmom}. 

\begin{condition}
\label{cond:HT1}  $C > K\varrho^{r}$ for all $r$. For some $\lambda_j, \mu_j \in (0,\infty)$,
$\lim_{r\rightarrow \infty } \lambda_{j}^{r} = \lambda_j$, 
$\lim_{r\rightarrow \infty } \mu_{j}^{r} = \mu_j$, for all $j \in \NJ$. With $\varrho_j = \frac{\lambda_j}{\mu_j}$
and $\varrho = (\varrho_j)_{j\in \NJ}$,
 $ C= K\varrho$, 
$
\lim_{r\rightarrow \infty}r(\varrho - \varrho^{r})= \beta^*$, $v^* \doteq K\beta^*>0$.
\ 
\end{condition}

Consider a $J$-dimensional absolutely continuous, nonnegative, non-decreasing stochastic process $%
\{B^r(t)\}$ where $B^r_j(t)$ represents the amount of type $j$ work
processed by time $t$ under a given policy.
%
%
Note that such a process must satisfy the resource
constraint: 
\begin{equation}
K\dot{B}^r_j(t) \le C, \mbox{ for all } t \ge 0.\label{eq:resconst}
\end{equation}

Define the $I$ dimensional capacity-utilization process $T^r = KB^r$. Then $%
T^r_i(t)$ represents the amount of work processed by the $i$-th resource by
time $t$. Letting $I^r(t)= tC -T^r(t)$, $I^r_i(t)$ represents the unused
capacity of resource $i$ by time $t$. %
Let $\{Q^{r}(t)\}$ be the $J$-dimensional process, where $Q^{r}_j(t)$
represents the number of jobs in the queue for type $j$ jobs. Then 
\begin{equation}
Q^{r}(t)=q^r+A^{r}(t)-S^{r}\left(B^r(t)\right),\label{eq:queleneqn}
\end{equation}
where $q^r$ denotes the initial queue-length vector.
For $B^r$ to be a valid rate allocation policy, $Q^r$ defined by \eqref{eq:queleneqn} must satisfy
\begin{equation}
	\label{eq:qrnonneg}
	Q^r(t) \ge 0 \mbox{ for all } t\ge 0.
\end{equation}
Any absolutely continuous, nonnegative, non-decreasing stochastic process $%
\{B^r(t)\}$ satisfying \eqref{eq:resconst}, \eqref{eq:qrnonneg} and appropriate non-anticipativity conditions
will be referred to as a {\bf resource allocation policy} or simply a {\bf control policy}.
Non-anticipativity conditions on $\{B^r\}$ are formulated using multi-parameter filtrations as in \cite{budgho2} (see Definition 2.6 (iv) therein). We omit details here, however we will note that from Theorem 5.4 of
\cite{budgho2} it follows that the control policy constructed in Section \ref{sec:secallostra} is non-anticipative in the sense of \cite{budgho2}.

  Let $W^{r}(t)$ be the $I$%
-dimensional workload process given by $W^{r}(t)=KM^{r}Q^{r}(t)$ where $M^{r}
$ is the diagonal matrix with entries $1/\mu _{j}^{r}$.

Define the fluid-scaled quantities by%
\begin{align}\label{eq:eq935}
	\begin{split}
&\bar{T}^{r}(t)=T^{r}(r^{2}t)/r^{2},\;\; \bar{B}^{r}(t)=B^{r}(r^{2}t)/r^{2},
\;\; \bar{I}^{r}(t)=I^{r}(r^{2}t)/r^{2} \\
&\bar{A}^{r}(t)=A(r^{2}t)/r^{2},\;\; \bar{S}^{r}(t)=S^{r}(r^{2}t)/r^{2}, \\
& \bar{Q}^{r}(t)=Q^{r}(r^{2}t)/r^{2}, \;\; \bar{W}^{r}(t)=W(r^{2}t)/r^{2} 
\end{split}
\end{align}%
and the diffusion scaled quantities 
\begin{align}\label{eq:eq938}
	\begin{split}
&\hat{T}^{r}(t)=T(r^{2}t)/r,\;\;\hat{B}^{r}(t)=B^{r}(r^{2}t)/r,\;\;\hat{I}%
^{r}(t)=I^{r}(r^{2}t)/r, \\
&\hat{A}^{r}(t)=(A(r^{2}t)-\lambda ^{r}r^{2}t)/r,\;\; \hat{S}%
^{r}(t)=(S^{r}(r^{2}t)-\mu ^{r}r^{2}t)/r, \\
&\hat{Q}^{r}(t)=Q(r^{2}t)/r,\;\; \hat{W}^{r}(t)=W^{r}(r^{2}t)/r\text{.}
\end{split}
\end{align}%
Note that, with $G^r \doteq KM^r$ and $\hat w^r \doteq G^r\hat q^r$, 
\begin{align}
\hat{W}^{r}(t) = G^r\hat{Q}^{r}(t) 
= \hat w^r +G^{r}(\hat{A}^{r}(t)-\hat{S}^{r}(\bar{B}^{r}(t)))+tr(K\varrho
^{r}-C)+r\bar{I}^{r}(t)\text{.}\label{eq:eq939}
\end{align}%
Let $h$ be a given $I$-dimensional strictly positive vector. Associated with
a control policy $B^r$, We will be interested in two types of cost
structures:

\begin{itemize}
\item \textbf{Infinite horizon discounted cost:} Fix $\theta \in (0,\infty)$.
\begin{equation}
	\label{eq: costdisc}
	J_D^r(B^r, q^r) \doteq
\int_0^{\infty} e^{-\theta t} E  \left(h \cdot \hat Q^r(t)\right) dt.
\end{equation} 

\item \textbf{Long-term cost per unit time:} 
\begin{equation}
		\label{eq: costerg}
	J_E^r(B^r, q^r) \doteq
\limsup_{T\to \infty} \frac{1}{T} \int_0^T E  \left(h \cdot \hat Q^r(t)\right)
dt.
\end{equation}
\end{itemize}
The goal of this work is to construct dynamic rate allocation policies that asymptotically achieve the Hierarchical Greedy Ideal(HGI) performance as $r\to \infty$. The next section gives the precise definition of HGI performance.

\section{Hierarchical Greedy Ideal}
\label{sec:hgi}
Similar to $M^r$ and $G^r$ in Section \ref{sec:backg}, let $M$ be the $%
J\times J$ diagonal matrix with entries $\{1/\mu_j\}_{j=1}^J$ and let $G
\doteq KM$. Define for $w \in \mathbb{R}_+^I$ (regarded as a workload
vector), the set of possible associated queue lengths $\mathcal{Q}(w)$ by
the relation 
\begin{equation*}
\mathcal{Q}(w)\doteq \{ q\in \mathbb{R}_+^J: Gq=w\}.
\end{equation*}
Note that by our assumption on $K$, $\mathcal{Q}(w)$ is compact for every $w \in \mathbb{R}_+^I$. 
Also the local traffic condition (Condition \ref{cond:loctrafcond}) ensures that $\clq(w)$ is nonempty for every $w \in \RR_+^I$. 
HGI performance introduced in \cite{harmandhayan} is motivated by the  Brownian control problem (BCP), as introduced in \cite{har1},  associated with the network in Section \ref{sec:backg} and the holding cost vector $h$. This BCP has an equivalent workload formulation (EWF) from the results of \cite{harvan} (see Section 10 of \cite{harmandhayan}). 
The EWF in the current setting 
is a singular control problem with state space that is all of the positive orthant $\RR_+^I$(due to the local traffic condition). 
In the EWF the cost is given by a nonlinear function $\clc$ defined as
\begin{equation}  \label{eq:eq942}
\clc(w) \doteq \inf_{q\in \mathcal{Q}(w)}\{ h\cdot q\}, \; w \in \mathbb{R}_+^I .
\end{equation}
One particular control in the EWF is the one corresponding to no-action in the interior and  normal reflection on the boundary of the orthant. This control  yields the (coordinate-wise) minimal controlled state process in the EWF given as  the $I$-dimensional reflected Brownian motion in $\RR_+^I$ with normal reflection.
The HGI performance is the cost, in terms of the the workload cost function $\clc$, associated with this minimal state process. We now give precise definitions.

We first recall the definition of the Skorohod problem and Skorohod map with normal reflections on the $d$-dimensional positive orthant.
\begin{definition}
	\label{def:smsp}
	Let $\psi \in D([0,T]: \RR^d)$ such that $\psi(0)\in \RR^d_+$. The pair $(\varphi,\eta) \in D([0,T]: \RR^d\times \RR^d)$ is said to solve the {\em Skorohod problem} for $\psi$ (in $\RR^d_+$, with normal reflection)
	if $\varphi = \psi+\eta$; $\varphi(t)\in \RR^d_+$ for all $t \ge 0$; $\eta(0)=0$; $\eta$ is nondecreasing and $\int_{[0,T]} 1_{\{\varphi_i(t)>0\}} d\eta_i(t) =0$.
	We write $\varphi = \Gamma_d(\psi)$ and refer to $\Gamma_d$ as the $d$-dimensional {\em Skorohod map}.
\end{definition}
It is  known that there is a unique solution to the above Skorohod problem for every $\psi \in D([0,T]: \RR^d)$ and that the Skorohod map has the following Lipschitz property: There exists $K_{\Gamma_d} \in (0,\infty)$ such that for all $T>0$ and $\psi_i \in D([0,T]: \RR^d)$ such that $\psi_i(0) \in \RR^d_+$, $i=1,2$
$$\sup_{0\le t \le T}|\Gamma_d(\psi_1)(t) - \Gamma_d(\psi_2)(t)| \le K_{\Gamma_d} \sup_{0\le t \le T}|\psi_1(t)- \psi_2(t)|.$$
Also note that for $\psi \in D([0,T]: \RR^d)$, $\Gamma_d(\psi)_i = \Gamma_1(\psi_i)$ for all $i = 1, \ldots d$.
When $d=I$ we will write $\Gamma_d=\Gamma_I$ as simply $\Gamma$.

Let $(\ch \Om, \ch \clf, \{\ch \clf_t\}, \ch P)$ be a filtered probability space on which is given a $J$-dimensional standard $\{\ch \clf_t\}$- Brownian motion $\{\ch B(t)\}$.
Let $\zeta_j \doteq 2\varrho_j/\mu_j$ for $j \in J$ and let $\mbox{{ Diag}}(\mathbf{\zeta})$ be the $J\times J$ diagonal matrix with $j$-th diagonal entry $\zeta_j$. Let
$\Lambda \doteq K(\mbox{{ Diag}}(\mathbf{\zeta}))^{1/2}$. For $w_0 \in \RR_+^I$, let $\ch W^{w_0}$ be a $\RR_+^I$ valued continuous stochastic process defined as 
\begin{equation}
	\label{eq:eqrbm}
\ch W^{w_0}(t) = \Gamma (w_0 - v^*\iota + \Lambda \ch B(\cdot))(t), \; t \ge 0
\end{equation}
where $\iota: [0,\infty) \to [0,\infty)$ is the identity map. Then $\ch W^{w_0}$ is a $I$-dimensional reflected Brownian motion with initial value $w_0$, drift $-v^*$ and covariance matrix $\Lambda \Lambda'$.
It is well known\cite{harwil1} that $\{\ch W^{w_0}\}_{w_0 \in \RR_+^I}$ defines a Markov process that has a unique invariant probability distribution which we denote as $\pi$. 

 Suppose $\hat q^r \doteq q^r/r \to q_0$ as $r\to \infty$ and let $w_0\doteq Gq_0$.  Then the HGI cost associated with the costs $J_D^r(B^r, q^r) $ and $J_E^r(B^r, q^r) $
are given respectively as
\begin{align}
	\mbox{HGI}_D(w_0) &\doteq \int_0^{\infty} e^{-\theta t} E  \left(\clc(\ch W^{w_0}(t))\right) dt\nonumber\\
		\mbox{HGI}_E &\doteq \int_{\RR^+_I} \clc(w) \pi(dw). \label{eq:hgicosts}
\end{align}
\section{Control Policy and Convergence to HGI}
\label{sec:polandmaires}
This section will introduce our final key condition on the model and present our main results.
Denote by $g_1, \ldots g_J$ the columns of the matrix $G$, i.e. $G=[g_1,
\ldots, g_J]$. We will partition the set $\mathbb{N}_J$ into sets $\cls^p$ and $%
\cls^s$ corresponding to the set of \emph{primary jobs} and the set of \emph{%
secondary jobs} respectively, defined as follows 
\begin{equation}
	\label{eq:primjobdef}
\cls^p \doteq \{j \in \mathbb{N}_J: \clc(g_j) < h_j\}, \; \cls^s \doteq \mathbb{N}%
_J\setminus \cls^p.
\end{equation}
Intuitively, $\cls^p$ corresponds to the set of jobs that we want to process
first.

Within the set of secondary jobs we will distinguish the set $\cls^1$,
introduced earlier, of all job types that use only one resource. Note that $%
\cls^1$ is indeed a subset of $\cls^s$ since for $j \in \cls^1$, $\mathcal{Q}(g_j) =
\{e_j\}$ and so 
\begin{equation*}
\clc(g_j) = \inf_{q\in \mathcal{Q}(g_j)} \{h \cdot q\} = h_j.
\end{equation*}

We now introduce the notion of \emph{minimal covering} sets associated with
any $j \in \mathbb{N}_J$ and also define, for given $F \subset \mathbb{N}_J
\setminus \{j\}$, minimal covering sets of $j$ that are not covering sets
for any $j^{\prime }\in F$.

\begin{definition}
\label{def:minCovCollect} Given $E\subset\mathbb{N}_J$ and $k\in \mathbb{N}_J$ we define $\mathcal{M}%
^{E,k}$ to be the collection of all minimal sets of jobs in $E$ other than $k$ such
that $N_k$ is contained in the set of all resources associated with the jobs
in the set, namely, 
\begin{equation*}
\mathcal{M}^{E,k} \doteq \left\{M\subset E\setminus \{k\}: N_{k}\subseteq
\bigcup _{j\in M}N_{j} \mbox{ and } N_k \nsubseteq \bigcup _{j\in M\setminus
\{l\}}N_{j} \mbox{ for all } l\in M\right\}.
\end{equation*}
In addition, given $F\subset \mathbb{N}_J$ define $\mathcal{M}_{F}^{E,k}$ to
be the collection of all $M \in \mathcal{M}^{E,k}$ such that the set of
resources associated with any job in $F$ is not contained in the set of
resources associated with jobs in $M$, namely, 
\begin{equation*}
\mathcal{M}_{F}^{E,k} \doteq \left\{M \in \mathcal{M}^{E,k}: N_{l}\nsubseteq
\bigcup _{j\in M}N_{j} \mbox{ for any } l\in F\right\}.
\end{equation*}
\end{definition}

Minimal covering sets will be used to determine the collection of jobs which
do not have lower priority than any other job in a given subset of $\mathbb{N%
}_J$. For that we introduce the following definition. Let $\cls^m \doteq
\cls^s\setminus \cls^1$ be the collection of secondary jobs that use multiple
resources and let $m \doteq |\cls^m|$. Denote the $j$-th column of $K$ by $K_j$%
, i.e. $K = [K_1, \ldots , K_J]$.

\begin{definition}
\label{def:optJob} Given sets $E,F\subset \cls^m$ define the set $\mathcal{O}%
_{F}^{E}\subset E$ by $j^{\prime }\in \mathcal{O}_{F}^{E}$ if and only if
for all $M\in \mathcal{M}_{F}^{E\bigcup \cls^1,j^{\prime }}$ 
\begin{equation}  \label{eq:eq929}
\mu _{j^{\prime }}h_{j^{\prime }}+\clc\left( \sum_{j\in M}K_j-K_{j^{\prime
}}\right) \leq \clc\left(\sum_{j\in M}K_j\right).
\end{equation}%
and the set $\mathcal{O}^{E}\subset E$ by $j^{\prime }\in \mathcal{O}^{E}$
if and only if \eqref{eq:eq929} holds for all $M\in \mathcal{M}^{E\bigcup
\cls^1,j^{\prime }}$.
\end{definition}
Note that since a $M\in \mathcal{M}_{F}^{E\bigcup \cls^1,j^{\prime }}$ covers $j'$, $\sum_{j\in M} K_j - K_{j'}$ is a nonnegative vector.
We now introduce the notion of a \emph{viable ranking} of jobs in $\cls^m$.

\begin{definition}
\label{def:viableRank}A viable ranking of jobs in $\cls^m$ is a bijection $\rho:%
\mathbb{N}_m\rightarrow \cls^m$, such that for all $k\in \mathbb{N}_m$,
$\rho(k)\in \mathcal{O}_{F_k}^{E_k}$,
where for $k \in \mathbb{N}_m$, $F_k \doteq \{\rho(1),...,\rho(k-1)\}$ and $%
E_k \doteq \cls^m \setminus F_k$, with the convention that $\mathcal{O}%
_{F_k}^{E_k}=\mathcal{O}^{\cls^m}$ for $k=1$.
\end{definition}

For an interpretation of a viable ranking, see Remark \ref{rem:poldisc}. The following will be one of main assumptions that will be taken to hold
throughout this work. This assumption (and Conditions \ref{cond:loctrafcond} and \ref{cond:HT1}) will not be noted explicitly in the
statements of the results.

\begin{condition}
\label{cond:viableRankExists}
There exists a viable ranking of jobs in $\cls^m$.
\end{condition}

In Section \ref{sec:example} we illustrate through examples that this
condition holds for a broad family of models.

We can now present our dynamic rate allocation policy.

\subsection{Resource Allocation Policy}
\label{sec:secallostra}
For $k\in \mathbb{N}_m$ let 
\begin{equation}\label{eq:eqzetaik}
\zeta _{i}^{k}=\{j\in \mathbb{N}_{J}\setminus F_{k+1}:K_{i,j}=1\}.
\end{equation}%
This class can be interpreted as the collection of jobs which impact node $i$ and have a higher processing priority than
job $\rho(k)$ (see Remark \ref{eq:jobprior}). \ 

Let $0<\alpha <1/2$ and $0<c_{1}<c_{2}$. \ Define 
\begin{equation*}
\sigma ^{r}(t)\doteq\left\{ j\in \mathbb{N}_{J}:Q_{j}^{r}(t)\geq
c_{2}r^{\alpha }\right\}
\end{equation*}%
to be the set of job-types whose queue length is at least $c_{2}r^{\alpha }$
at time $t$. \ Define 
\begin{equation*}
\varpi ^{r}(t)\doteq \bigcup _{j\in \sigma ^{r}(t)}N_{j}
\end{equation*}%
to be the subset of $\mathbb{N}_I$ consisting of resources associated with
job-types in $\sigma ^{r}(t)$, namely with queue lengths at least $%
c_{2}r^{\alpha }$.   We will use the
following work allocation scheme.

\begin{definition}
\label{def:workAllocScheme}Let $\delta =\frac{\min_{j}\varrho _{j}}{2J }$.
For $t\ge 0$, define the vector $y(t) = (y_j(t))_{j \in \mathbb{N}_J}$ as
follows.\newline

\noindent \textbf{Primary jobs.} For $j\in \cls^p$ 
\begin{equation}
y_{j}(t)\doteq \left\{ 
\begin{array}{cc}
\varrho _{j}+\delta, & \text{ if } j\in \sigma ^{r}(t) \\ 
\  &  \\ 
\varrho _{j}- \frac{\delta }{ J2^{m +3} }, & \text{ if } j\notin \sigma
^{r}(t).%
\end{array}
\right. \label{eq:A1prim}
\end{equation}
\newline
\noindent \textbf{Jobs in $\cls^m$.} For $k\in \mathbb{N}_m$ 
\begin{equation}
y_{\rho (k)}(t)\doteq \left\{ 
\begin{array}{cc}
\varrho _{\rho (k)}-2^{k-m -2}\delta, & \text{ if } \zeta _{i}^{k}\cap
\sigma ^{r}(t)\neq \emptyset \mbox{ for all } i\in N_{\rho (k)} \\ 
\  &  \\ 
\varrho _{\rho (k)}+2^{k-m -2}\delta, & \text{ if } \zeta _{i}^{k}\cap
\sigma ^{r}(t)= \emptyset \mbox{ for some } i\in N_{\rho (k)} \mbox{ and }
\rho (k)\in \sigma ^{r}(t) \\ 
\  &  \\ 
\varrho _{\rho (k)}-2^{-k-m -2}\delta, & \text{ if } \zeta _{i}^{k}\cap
\sigma ^{r}(t)= \emptyset \mbox{ for some } i\in N_{\rho (k)} \mbox{ and }
\rho (k)\notin \sigma ^{r}(t).%
\end{array}
\right. \label{eq:A2mult}
\end{equation}
\newline
\noindent \textbf{Jobs in $\cls^1$.} For $j \in \cls^1$
\begin{equation}
y_{j}(t)\doteq \left\{ 
\begin{array}{cc}
C_{\hat i(j)}-\sum_{l\neq j:K_{\hat i(j),l}=1}y_{l}(t), & \text{ if } \hat
i(j)\in \varpi^{r}(t) \\ 
\  &  \\ 
\varrho _{j}-\delta, & \text{ if } \hat i(j)\notin \varpi^{r}(t).%
\end{array}
\right. \label{eq:A3sing}
\end{equation}

For all $j$, define stopping times 
\begin{equation*}
\tau _{1}^{j}=\inf \{t\geq 0:Q_{j}^{r}(t)<c_{1}r^{\alpha }\}\text{,}
\end{equation*}%
\begin{equation*}
\tau _{2l}^{j}=\inf \{t\geq \tau _{2l-1}^{j}:Q_{j}^{r}(t)\geq c_{2}r^{\alpha
}\}\text{,}
\end{equation*}%
and 
\begin{equation*}
\tau _{2l+1}^{j}=\inf \{t\geq \tau _{2l}^{j}:Q_{j}^{r}(t)<c_{1}r^{\alpha }\}%
\text{,}
\end{equation*}%
for all $l>0$. \ Define $\mathcal{E}^{r}(t)\in \{0,1\}^{J}$ by%
\begin{equation*}
\mathcal{E}_{j}^r(t)\doteq \left\{ 
\begin{array}{cc}
1, & \text{ if } t\in \left[ \tau _{2l-1}^j,\tau _{2l}^j\right) \text{ for some }%
l>0 \\ 
\  &  \\ 
0, & \text{ otherwise. }%
\end{array}
\right. 
\end{equation*}
Finally, define $x(t) \in \mathbb{R}^J$ as $x_j(t) \doteq y_j(t) 1_{\{%
\mathcal{E}_{j}^r(t) =0\}}$ for $j \in \mathbb{N}_J$.
\end{definition}
We note that $y_j(t)$ and $x_j(t)$ depend on $r$ but this dependence is suppressed in the notation.
\begin{remark}
	\label{rem:poldisc}
Roughly speaking, under the allocation policy in Definition \ref{def:workAllocScheme}, jobs are prioritized as follows:
\begin{equation}
	\label{eq:jobprior}
	\cls^p \succ \cls^1 \succ \rho(m) \succ \rho(m-1) \cdots \succ \rho(1).
\end{equation} 
However the above priority order needs to be interpreted with some care. We will call the $j$-th queue {\em stocked}
at time instant $t$ if $Q^r_j(t) \ge c_2 r^{\alpha}$ and we will call it {\em depleted} at time instant $t$
if $Q^r_j(t) < c_1 r^{\alpha}$. The last line of Definition \ref{def:workAllocScheme} says that any queue once depleted does not get any rate allocation until it gets stocked again. 
Beyond that, rate allocation by a typical resource $i$ is decided as follows.

First we consider all the primary job-types associated with resource $i$, i.e. $j \in \cls^p$ such that $K_{ij}=1$.
If the associated queue is stocked then it gets higher than nominal rate allocation according to the first line in \eqref{eq:A1prim}
and otherwise a lower than nominal allocation as in the second line of \eqref{eq:A1prim}.
Next we look at all the job-types in $\cls^m$ associated with resource $i$. Denote these as $j_1, j_2, \ldots j_k$ and assume without loss of generality that $\rho(j_1)< \rho(j_2) \cdots < \rho(j_k)$. We consider the top ranked job $\rho(j_k)$ first and look at all the resources (including resource $i$) that process this job-type. 
If every associated resource has at least one job-type rated higher according to \eqref{eq:jobprior} with a stocked queue then rate allocated to job-type $\rho(j_k)$ is lower than nominal as given in the first line of \eqref{eq:A2mult}.  On the other hand, if there is at least one associated resource such that none of its  job-types that
are rated higher that $\rho(j_k)$ (according to \eqref{eq:jobprior}) has a stocked queue , we assign $\rho(j_k)$
a flow rate higher than nominal, according to the second line in \eqref{eq:A2mult} if the queue for job-type $\rho(j_k)$ is stocked and a lower than nominal flow rate according to the third line in \eqref{eq:A2mult} if the queue is not stocked.
Note that all resources processing job-type $\rho(j_k)$ allocate the same flow rate to it.
We then successively consider $\rho(j_{k-1}), \rho(j_{k-2}), \ldots \rho(j_{1})$ and allocate rate flows to it in a similar fashion as above.

Finally, if the unique job-type $\check j(i)$ queue associated with resource $i$ is stocked, we allocate it all remaining capacity of resource $i$ (this may be larger or small than nominal allocation) and if this queue is not stocked we assign it less than nominal allocation given by the second line in \eqref{eq:A3sing}.

Lemma \ref{lem:admispol} will show that $B^r(t) \doteq \int_0^t x(s) ds$  is nonnegative, nondecreasing and satisfies the resource constraint \eqref{eq:resconst}. Also, clearly the associated $Q^r$ defined by \eqref{eq:queleneqn}
satisfies \eqref{eq:qrnonneg}.
Finally, it can be checked that the process $B^r(t)$ is  non-anticipative in the sense of Definition  2.6 (iv) of \cite{budgho2}. Thus $B^r$ is a resource allocation policy as defined in Section \ref{sec:backg}.

We remark that the formal priority ordering given in \eqref{eq:jobprior} is consistent with the UFO priority scheme
proposed in Section 12 of \cite{harmandhayan} for 2LLN and 3LLN networks. However, the UFO scheme for C3LN network in \cite{harmandhayan}
appears to be of a  different form.
\end{remark}

\subsection{Main Results}
 Recall that we assume throughout that Conditions \ref{cond:loctrafcond}, \ref{cond:HT1} and \ref{cond:viableRankExists}
are satisfied. We now present the main results of this work.
The first result considers the ergodic cost whereas the second the discounted cost. Recall $q^r$ introduced in \eqref{eq:queleneqn}.
\begin{theorem}
	\label{thm:thm6.5}
Suppose $\hat q^r \doteq q^r/r$ satisfies $\sup_{r>0} \hat q^r <\infty$.
Let $t_{r} \uparrow \infty$ as $r\to \infty$. Then as $r\to \infty$
 $\frac{1}{t_{r}}\int_{0}^{t_{r}}
h\cdot \hat{Q}^{r}(t) dt$ converges in $L^1$ to $\int \clc(y)\pi (dy)$.
In particular, as $r\to \infty$,
$$J_E^r(B^r, q^r) \to \mbox{HGI}_E.$$
\end{theorem}

\begin{theorem}
	\label{thm:thm6.5disc}
Suppose that $\hat q^r \to q_0$ as $r\to \infty$. Let $w_0 = Gq_0$. Then
\begin{equation*}
\lim_{r\rightarrow \infty }J_D^r(B^r, q^r) = \mbox{HGI}_D(w_0).
\end{equation*}%
\end{theorem}
Proofs of the above theorems are given in Section \ref{sec:pfsmainthms}.

\section{Verification of Condition \ref{cond:viableRankExists}.}
\label{sec:example}
In this section we will give two  more transparent sets of criteria
which imply Condition \ref{cond:viableRankExists} and provide some examples
of networks which satisfy them. \ Note that these alternative conditions are
more restrictive and by no means necessary for Condition \ref%
{cond:viableRankExists} to hold. \ We present them because for certain types of
networks they provide an easy way to verify  Condition \ref%
{cond:viableRankExists}. \ We will then provide an example of a simple
network which does not satisfy Condition \ref{cond:viableRankExists} and
consequently does not fall in the family of systems analyzed here.

 Verifying Condition \ref{cond:viableRankExists} and
finding the optimal cost/queue length for a particular workload only
involves jobs in $\mathcal{S}^{s}$ (see Theorem \ref{thm:restOptJobs}). 
For this reason, sufficient conditions below impose conditions only on jobs in $\mathcal{S}^{s}$.
 Finally, for notational convenience, in this section we will denote the job type $j$
 that requires  service from nodes $i_{1},...,i_{n}$ by $\chi _{i_{1},...,i_{n}}$.  Similarly, we will use notation $h_{\chi _{i_{1},...,i_{n}}}$, $\mu _{\chi _{i_{1},...,i_{n}}}$, and $N_{\chi _{i_{1},...,i_{n}}}$ for the corresponding $h_j, \mu_j, N_j$.

\subsection{Some Simple Sufficient Conditions for Condition \protect\ref%
{cond:viableRankExists}}
\label{sec:simsuffcond}
We present below two basic sufficient (but not necessary) conditions for
Condition \ref{cond:viableRankExists} to be satisfied in order to illustrate
 networks that are covered by our approach.

\begin{theorem}
\label{thm:SuffSubsetJobs} If for all $j,k\in \cls^m$ either $%
N_{j}\subset N_{k}$, $N_{k}\subset N_{j}$, or $N_{j}\cap N_{k}=\emptyset $
then Condition \ref{cond:viableRankExists} is satisfied.
\end{theorem}

\begin{proof}
We will use the notation from  Definition \ref{def:viableRank}, namely
 $%
F_{k}\doteq \{\rho (1),...,\rho (k-1)\}$ and $E_{k}\doteq \cls^m\setminus
F_{k}$.  Take $\rho$ to be an arbitrary map from $\NN_m$ to $\cls^m$
with the property that for all $j,k \in \NN_m$ with $j<k$, either $N_{\rho
(k)}\subset N_{\rho (j)}$ or $N_{\rho (j)}\cap N_{\rho (k)}=\emptyset $.
Note that our assumption  in the statement of the theorem ensures that such a map always exists.
We now argue that this $\rho$ defines a viable ranking, namely Condition \ref{cond:viableRankExists}
is satisfied.  For this we need to show that for every $k \in \NN_m$, $\rho(k) \in \clo^{E_k}_{F_k}$, namely
for all $M \in \clm_{F_k}^{E_k\bigcup \cls^1, \rho(k)}$
\begin{align}
	\label{eq:eqmurhok}
	\mu_{\rho(k)} h_{\rho(k)} + \clc\left(\sum_{j\in M} K_j - K_{\rho(k)}\right) \le \clc\left(\sum_{j\in M} K_j \right).
\end{align}
Now consider such a $k$ and $M$. Note that $M \subset \{\rho(k+1), \ldots \rho(m)\} \bigcup \cls^1$.
Since $M$ defines a minimal covering, if for $l\neq l'$, $\rho(l), \rho(l') \in M$,  we must have that
$N_{\rho(l)} \cap N_{\rho(l')} = \emptyset$. From minimality of $M$ we also have that,
$(\bigcup_{j\in \cls^1\cap M} N_j) \cap N_{\rho(l)} = \emptyset$ for every $l\ge k+1$ such that $\rho(l)\in M$.
We thus have $\sum_{j\in M}|N_j| = |N_{\rho(k)}|$ which implies that 
\begin{equation}
	\label{eq:eqkjkk}
	\sum_{j\in M} K_j = K_{\rho(k)}.
\end{equation}
Therefore,
\begin{align*}
		\mu_{\rho(k)} h_{\rho(k)} + \clc\left(\sum_{j\in M} K_j - K_{\rho(k)}\right) = \mu_{\rho(k)} h_{\rho(k)} = \mu_{\rho(k)} \clc(g_{\rho(k)}) = \clc(\mu_{\rho(k)}g_{\rho(k)})
		= \clc(K_{\rho(k)}) = \clc(\sum_{j\in M} K_j),
\end{align*}
where the first and last equality use \eqref{eq:eqkjkk} and the second equality uses the fact that $\rho(k)$ is a secondary job.
This proves \eqref{eq:eqmurhok} (in fact with equality) and completes the proof of the theorem.

\end{proof}

\begin{remark}
	\label{rem:firsthm}
	One simple consequence of Theorem \ref{thm:SuffSubsetJobs} is that any network
where $\cls^m=\emptyset $ (meaning $\cls^{s}=\cls^1$) satisfies Condition \ref{cond:viableRankExists}. We note that condition $\cls^m=\emptyset $ does not rule out existence of 
jobs that require service from multiple nodes. Here is one elementary example to illustrate this point. Suppose $I=3$
and $J=6$ with $\mu_j=1$ for all $j$. Also let $h_{\chi_{1}} = h_{\chi_{2}}= h_{\chi_{2}}=1$,
$h_{\chi_{1,2,3}} = h_{\chi_{1,2}} = h_{\chi_{2,3}} =4$. It is easy to check that for this example $\cls^m=\emptyset$.

Another  consequence of Theorem \ref{thm:SuffSubsetJobs} is that any network where $%
\cls^m$ only contains one job (for instance a job which impacts all nodes)
satisfies Condition \ref{cond:viableRankExists}. \ In particular any $2$
node network  satisfies Condition \ref{cond:viableRankExists}.
\ Another basic network covered by Theorem \ref{thm:SuffSubsetJobs} is one
with $2n$ jobs where $\cls^m=\{\chi _{1,2,...,2n},\chi _{1,2},\chi
_{3,4},\chi _{2n-1,2n}\}$. Many other examples can be given. In particular 2LLN and 3LLN networks of
\cite{harmandhayan} satisfy the sufficient condition in Theorem \ref{thm:SuffSubsetJobs} .
\end{remark}
The following theorem provides another
sufficient condition for a network to satisfy Condition \ref%
{cond:viableRankExists}. Recall that $\mathcal{O}^{\cls^m}$ is the collection of all $j' \in \cls^m$ that satisfy \eqref{eq:eq929} for all  $M\subset \cls^s\setminus \{j'\}$ that are minimal covering sets for $j'$. 
\begin{theorem}
\label{thm:SuffRemoveJobCond}If for all $j\in \left. \cls^m\right\backslash 
\mathcal{O}^{\cls^m}$ and $M\in \mathcal{M}_{\mathcal{O}^{\cls^m}}^{\{ \left.
\cls^m\right\backslash \mathcal{O}^{\cls^m} \} \bigcup \cls^1,j}$ we have $\sum_{l\in
M}|N_{l}|=|N_{j}|$ then Condition \ref{cond:viableRankExists} is satisfied.
\end{theorem}

\begin{proof}
Consider the following ranking of jobs in $\cls^m$.
Assign the first $\bar m \doteq \left\vert \mathcal{O}^{\cls^m}\right\vert $
ranks arbitrarily to jobs in $\mathcal{O}^{\cls^m}$ and the remaining $m-\bar m$
ranks arbitrarily to jobs in $\cls^m\left\backslash \mathcal{O}%
^{\cls^m}\right. $. In particular  $\rho (k)\in \mathcal{O}^{\cls^m}$ for all $k\in
\{1,...,\bar m \}$ and $\rho (k)\in
\cls^m\left\backslash \mathcal{O}^{\cls^m}\right. $ for all $k\in \{\bar m +1,...,m \}$. \
Note that, for $k\in \{1,...,\bar m \}$
we have $\mathcal{O}^{\cls^m} \subset \mathcal{O}_{F_{k}}^{E_{k}}$ which says that $\rho
(k)\in \mathcal{O}_{F_{k}}^{E_{k}}$ for all $k\in \{1,...,\bar m \}$. \ Let now
$k\in \{\bar m +1,...,m \}$ be arbitrary and
note that $\mathcal{M}_{F_{k}}^{E_{k}\bigcup \cls^1,\rho (k)}\subset \mathcal{M}_{\mathcal{O}^{\cls^m}}^{ \{\cls^m\backslash \mathcal{O}^{\cls^m}\}\bigcup \cls^1,\rho (k)}$ so
for all $M\in \mathcal{M}_{F_{k}}^{E_{k}\bigcup \cls^1,\rho (k)}$ we have $\sum_{l\in
M}|N_{l}|=|N_{\rho (k)}|$. \ This implies that \eqref{eq:eqkjkk} is satisfied which as in the proof of Theorem \ref{thm:SuffSubsetJobs} shows that 
 (\ref{eq:eq929}) is satisfied for all $M\in \mathcal{M}%
_{F_{k}}^{E_{k}\bigcup \cls^1,\rho (k)}$ and therefore $\rho (k)\in \mathcal{O}_{F_{k}}^{E_{k}}$ for all $k\in \{\bar m +1,...,m \}$. Thus $\rho$ defines a viable ranking and so Condition \ref{cond:viableRankExists}
is satisfied.
\end{proof}
\begin{remark}
\label{rem:secthm}
 The above theorem provides an easy way to check that Condition \ref{cond:viableRankExists} is satisfied. \ For instance, for $3$ node networks
if $\chi _{1,2,3}\in \cls^m$, from Theorem \ref{thm:SuffRemoveJobCond}, verification of Condition \ref{cond:viableRankExists}
reduces to proving that $\chi _{1,2,3}\in \mathcal{O}^{\cls^m}$. \ This is due to the fact that
for a $3$ node network $\cls^m\subset \{\chi _{1,2,3},\chi _{1,2},\chi
_{1,3},\chi _{2,3}\}$, and consequently for any job $j\in \left.
\cls^m\right\backslash \{\chi _{1,2,3}\}$ and $M\in \mathcal{M}_{\mathcal{\{}%
\chi _{1,2,3}\}}^{\left. \cls^m\right\backslash \mathcal{\{}\chi
_{1,2,3}\}\bigcup \cls^1,j} $ we must have $M\cap \cls^m=\emptyset $ which says that $\sum_{l\in M}|N_{l}|=|N_{j}|$. \ 
In particular the C3LN in \cite{harmandhayan} satisfies the sufficient condition in Theorem \ref{thm:SuffRemoveJobCond}
with one viable ranking given as $\rho(1) = \chi _{1,2,3}$, $\rho(2) = \chi _{1,2}$, $\rho(3) = \chi _{2,3}$.

Similarly, for a $4$ node
network with $\cls^m\subset \{\chi _{1,2,3,4},\chi _{1,2,3},\chi
_{1,2,4},\chi _{1,3,4},\chi _{2,3,4}\}$, from  Theorem \ref{thm:SuffRemoveJobCond}, verification of Condition \ref{cond:viableRankExists}  reduces to proving that $\chi_{1,2,3,4}\in \mathcal{O}^{\cls^m}$. 
Many other examples can be given. In general  Theorem \ref{thm:SuffRemoveJobCond} can  be useful for verifying Condition \ref{cond:viableRankExists} for networks with high
number of nodes when  $\cls^m$ has few elements.
In particular the negative example in Section 13 of \cite{harmandhayan}  satisfies the sufficient condition in the above theorem. In that example $J=9$, $I=6$ and $\cls^m = \{\chi_{1,2,3}, \chi _{4,5,6}, \chi_{3,6}\}$. It is easy to see
that with the values of holding costs and job sizes in the above paper 
$\mathcal{O}^{\cls^m} = \{\chi_{1,2,3}, \chi _{4,5,6}\}$
and 
$\cls^m\left\backslash \mathcal{O}^{\cls^m}\right.
= \{\chi_{3,6}\}$  and so the only $M \in \mathcal{M}_{\mathcal{O}^{\cls^m}}^{\{ \left.
\cls^m\right\backslash \mathcal{O}^{\cls^m} \} \bigcup \cls^1,j}$ for $j = \chi_{3,6}$
is the set $\{\chi_{3}, \chi_{6}\}$ which clearly satisfies the property $\sum_{l\in
M}|N_{l}|=|N_{j}|$.
\end{remark}
It should be noted that Theorems \ref{thm:SuffSubsetJobs} and \ref{thm:SuffRemoveJobCond} are much more restrictive than necessary, meaning that the class of networks which satisfy Condition \ref{cond:viableRankExists} is much wider than those covered by Theorem \ref{thm:SuffSubsetJobs} or Theorem \ref{thm:SuffRemoveJobCond}.  To illustrate this we provide a simple example of one such network. 
\begin{example}
	\label{exam:out}
 Let $I=4$, $J=7$, and 
\begin{align*}\mu _{\chi _{1}}&=\mu
_{\chi _{2}}=\mu _{\chi _{3}}=\mu _{\chi _{4}}=\mu _{\chi _{1,2}}=\mu _{\chi _{2,3}}=\mu _{\chi _{1,2,3,4}}=1\\
h_{\chi _{1}}&=h_{\chi _{2}}=h_{\chi _{3}}=h_{\chi _{4}}=4, h_{\chi _{1,2}}=6, h_{\chi _{2,3}}=7, h_{\chi _{1,2,3,4}}=13.
\end{align*}
It is easy to verify that $\cls^m = \{\chi _{1,2},\chi _{2,3}, \chi _{1,2,3,4}\}$ and there is exactly one viable ranking as in Definition \ref{def:viableRank} which is $\rho(1)=\chi_{1,2,3,4}, \rho(2)=\chi _{1,2}, \rho(3)=\chi _{2,3}$ (so Condition \ref{cond:viableRankExists} is satisfied).  In particular this implies $\mathcal{O}^{\cls^m}=\{\chi _{1,2,3,4}\}$.  However, note that $N_{\chi _{1,2}}\not\subset N_{\chi _{2,3}}$, $N_{\chi _{2,3}}\not\subset N_{\chi _{1,2}}$, and $N_{\chi _{1,2}}\cap N_{\chi _{1,2}}\neq \emptyset$ so this network does not satisfy the conditions of Theorem \ref{thm:SuffSubsetJobs}.  In addition, $\chi _{2,3} \in \left. \cls^m\right\backslash 
\mathcal{O}^{\cls^m}$ and $\{ \chi _{1,2}, \chi _{3} \}\in \mathcal{M}_{\mathcal{O}^{\cls^m}}^{\{ \left.
\cls^m\right\backslash \mathcal{O}^{\cls^m} \} \bigcup \cls^1, \chi _{2,3}}$ but $|N_{\chi _{1,2}}|+|N_{\chi _{3}}|>|N_{\chi _{2,3}}|$ so the conditions of Theorem \ref{thm:SuffRemoveJobCond} are not satisfied either.  Consequently this simple network  satisfies Condition \ref{cond:viableRankExists} although it is outside the scope of Theorems  \ref{thm:SuffSubsetJobs} and \ref{thm:SuffRemoveJobCond}.
\end{example}
As seen in the last two theorems, Condition 
\ref{cond:viableRankExists} holds for a broad range of networks. However there are many interesting cases that are not covered by this condition.  We now illustrate this point through an example. In this example $I=3$ and $J=6$
and  $\clc$ is a non decreasing function, however a viable ranking does not exist and therefore techniques of this paper do not apply.
\begin{example}{\protect (Example That Doesn't Satisfy Condition \protect
\ref{cond:viableRankExists})}
\label{exam:outout}
Suppose that
\begin{align*}\mu _{\chi _{1}}&=\mu
_{\chi _{2}}=\mu _{\chi _{3}}=\mu _{\chi _{1,2}}=\mu _{\chi _{2,3}}=\mu
_{\chi _{1,2,3}}=1\\
h_{\chi _{1}}&=h_{\chi _{2}}=h_{\chi _{3}}=5, h_{\chi _{1,2}}=7, h_{\chi _{2,3}}=8, h_{\chi _{1,2,3}}=11.
\end{align*}
It is easy to check that in this case $\cls^m = \{\chi _{1,2},\chi _{2,3}, \chi _{1,2,3}\}$.
  This
network does not satisfy Condition \ref{cond:viableRankExists}\ because $%
\mathcal{O}^{\left\{ \chi _{1,2},\chi _{2,3},\chi _{1,2,3}\right\}
}=\varnothing $, since (\ref{eq:eq929}) does not hold for $\chi _{1,2}$%
, $\chi _{2,3}$, or $\chi _{1,2,3}$. We leave the verification of this fact to the reader.
Consequently a viable ranking cannot exist.

 {\em Workload cost and its minimizer.}
The workload $\clc$ for this example can be given explicitly as follows. Let for $w \in \mathbb{R}_+^3$, $w_{12} \doteq w_1 \wedge w_2$, $w_{23} \doteq w_2 \wedge w_3$,
$w_{123} \doteq w_1 \wedge w_2 \wedge w_3$.

For $w \in \mathbb{R}_+^3$
\begin{equation*}
\clc(w)\doteq \left\{ 
\begin{array}{cc}
5w_{2}+2w_{1}+3w_{3}, & \text{ if } w_{2}\geq w_{1}+w_{3} \\ 
 3w_{1}+4w_{2}+4w_{3}, & \text{ if } w_{1}+w_{3}>w_{2}\geq w_{1} \vee w_{3}\\
5(w_{1}+w_{2}+w_{3})+w_{123}-3w_{12}-2w_{23}, & \text{ if } w_{1} \vee w_{3}>w_{2}\\.%
\end{array}
\right. 
\end{equation*}
The optimal $q^*(w)$ in $\mathcal{Q}(w)$ is given as follows.
Let $q^* = (q_{\chi _{1}}^{\ast}, q_{\chi _{2}}^{\ast}, q_{\chi _{3}}^{\ast}, q_{\chi _{1,2}}^{\ast}, q_{\chi _{2,3}}^{\ast}, q_{\chi _{1,2,3}}^{\ast})$. Then
{\small \begin{align*}
&q^*(w)= \\
&\left\{ 
\begin{array}{cc}
(0,\,w_{2}-w_{1}-w_{3},\,0,\,w_{1},\,w_{3},\,0), & \text{ if } w_{2}\geq w_{1}+w_{3} \\ 
 ( 0,\,0,\,0,\,w_{2}-w_{3},\,w_{2}-w_{1},\,w_{1}+w_{3}-w_{2}), & \text{ if } w_{1}+w_{3}>w_{2}\geq w_{1} \vee w_{3}\\
(w_{1}-w_{12},\,w_{2}+w_{123}- w_{12}-w_{23},\, w_{3}-
w_{23},\, w_{12}- w_{123},\, 
w_{23}-w_{123},\, w_{123}), & \text{ if } w_{1} \vee w_{3}>w_{2}\\.%
\end{array}
\right. 
\end{align*}}
Note that $\clc$ and $q^*$ are continuous functions and $\clc$ is nondecreasing. In particular the HGI performance in this case is also the optimal cost in the associated BCP. However, as noted above, there does not exist a viable ranking for this example.  Thus the techniques developed in the current paper do not apply to this example.
\end{example}

\section{Some Properties of the Workload Cost Function}
\label{sec:secworkcost}

The following result on a continuous selection of a minimizer is well known
(cf. Theorem 2 in \cite{boh1} or Proposition 8.1 in \cite{harmandhayan}).

\begin{theorem}
\label{thm:costAchieved} There is a continuous map $\bar q: \mathbb{R}_+^I
\to \mathbb{R}_+^J$ such that for every $w \in \mathbb{R}_+^I$, $\bar q(w)
\in \mathcal{Q}(w)$ and 
\begin{equation*}
h \cdot \bar q(w) = \clc(w).
\end{equation*}
\end{theorem}

Define for a given workload vector $w \in \mathbb{R}_+^I$ the set $\mathcal{Q%
}^s(w)$ consisting of all queue-length vectors that produce the workload $w$
and have zero coordinates for queue-lengths corresponding to primary jobs,
namely, 
\begin{equation*}
\mathcal{Q}^{s}(w)=\left\{ q\in \mathcal{Q}(w): q_{j}=0\text{ for all }j\in 
\cls^p\right\} \text{.}
\end{equation*}

The following theorem shows that in computing the infimum in \eqref{eq:eq942}
we can replace $\mathcal{Q}(w)$ with $\mathcal{Q}^{s}(w)$.

\begin{theorem}
\label{thm:restOptJobs} For all $w\in \mathbb{R}_+^I$, $\bar q(w)
\in \mathcal{Q}^{s}(w)$. In particular, 
\begin{equation*}
\clc(w)=\inf_{q\in \mathcal{Q}^{s}(w)}\left\{ h\cdot q\right\} \text{.}
\end{equation*}
\end{theorem}

\begin{proof}
Fix $w\in \mathbb{R}_+^I$. With $\bar q$ as in Theorem \ref{thm:costAchieved}%
, we have $\clc(w)=h \cdot\bar{q}(w)$. Assume $\bar{q}_{k}(w)>0$ for some $k\in 
\cls^p$. Then with $q^* \doteq \bar q(g_k)$, we have from the
definition of $\cls^p$ that 
\begin{equation}
	\label{eq:hcqstar}
h\cdot q^* = \clc(g_k) < h_k.
\end{equation}
Define $\tilde q \in \mathbb{R}_+^J$ by $\tilde{q}_{k}=\bar{q}_{k}(w){q}%
^*_{k}$ and $\tilde{q}_{j}=\bar{q}_{j}(w)+\bar{q}_{k}(w){q}^*_{j}$ for $%
j\neq k$. \ Then for $i \in \mathbb{N}_I$, noting that 
\begin{equation*}
\sum_{j=1}^J G_{ij} q^*_j = \sum_{j=1}^J G_{ij} \bar q_j(g_k) = (g_k)_i =
G_{ik},
\end{equation*}
we have 
\begin{align*}
w = \sum_{j\neq k}G_{ij}\bar{q}_{j}(w) +G_{ik}\bar{q}_{k}(w) 
=\sum_{j\neq k}G_{ij}\bar{q}_{j}(w)+\left(\sum_{j=1}^{J}G_{ij}{q}%
^*_{j}\right)\bar{q}_{k}(w) = G\tilde q 
\end{align*}%
and consequently 
\begin{align*}
\clc(w) = \sum_{j\neq k}h_{j}\bar{q}_{j}(w)+h_{k}\bar{q}_{k}(w) 
 >\sum_{j\neq k}h_{j}\bar{q}_{j}(w)+\bar{q}_{k}(w)\sum_{j=1}^{J}h_{j}{q}%
^*_{j} 
= h\cdot\tilde{q} \ge \clc(w)
\end{align*}%
where the inequality in the above display is from \eqref{eq:hcqstar} and from the fact that, by assumption, $\bar{q}_{k}(w)>0$.
Thus we have a contradiction and therefore $\bar{q}_{k}(w)=0$ for all $k\in 
\cls^p$ which completes the proof.
\end{proof}

 Hereafter we fix a viable
ranking $\rho$. As was noted in Theorem \ref{thm:costAchieved}, there exists
a continuous selection of the minimizer in \eqref{eq:eq942}. We now show
that using the ranking $\rho$, one can give a rather explicit representation
for such a selection function.

Given $w \in \mathbb{R}_+^I$, define $q^*(w) \in \mathbb{R}_+^J$ as follows.
Set $q^*_j(w) = 0$ for $j \in \cls^p$. Define,  
\begin{equation}
q_{\rho(1)}^{*}(w)=\min_{i\in N_{\rho(1)}}\{w_{i}\}\mu _{\rho(1)}\text{.}
\label{eq:eq834}
\end{equation}%
For $k\in \{2,\ldots, m \}$ define, recursively, 
\begin{equation}  \label{eq:eq834b}
{q}_{\rho(k)}^{*}(w)=\min_{i\in
N_{\rho(k)}}\left\{w_{i}-\sum_{l=1}^{k-1}G_{i,\rho(l)} {q}%
_{\rho(l)}^{*}(w)\right\} \mu_{\rho(k)}\text{.}
\end{equation}%
Finally, for $j\in \cls^1$ define 
\begin{equation}  \label{eq:eq753}
{q}_{j}^{*}(w)=\left\{ w_{\hat i(j)}-\sum_{k=1}^{m }G_{\hat i(j),\rho(k)} {q}%
_{\rho(k)}^{*}(w)\right\} \mu _{j}\text{,}
\end{equation}
where recall that $\hat i(j)$ is the unique resource processing the job $j$.
By a recursive argument it is easy to check that $q^*(w)$ defined above is a
non-negative vector in $\mathbb{R}^J$. The following theorem shows that $q^*$
defined above is a continuous selection of the minimizer in \eqref{eq:eq942}.

\begin{theorem}
\label{thm:restCostJobOrd}For any $w\in \mathbb{R}_+^I$, $q^*(w) \in 
\mathcal{Q}^{s}(w)$ and 
\begin{equation}  \label{eq:eq755}
\clc(w)= h\cdot {q}^{*}(w) =\sum_{k=1}^{m }h_{\rho(k)}{q}_{\rho(k)}^{*}(w)+\sum_{j%
\in \cls^1}h_{j}{q}_{j}^{*}(w).
\end{equation}
\end{theorem}

\begin{proof}
Fix $w\in \mathbb{R}_+^I$. 
Let $\bar q(w)$ be as in Theorem \ref{thm:costAchieved}. Then $\clc(w) = h
\cdot \bar q(w)$ and the proof of Theorem \ref{thm:restOptJobs} shows that $%
\bar q(w) \in \mathcal{Q}^{s}(w)$. Define 
\begin{equation*}
s_{1}=\sup \left\{ q_{\rho(1)}:q\in \mathcal{Q}^{s}(w) \mbox{ and } h\cdot q
=\clc(w)\right\}.
\end{equation*}%
Clearly the supremum is achieved, namely there is a $\check{q} \in \mathcal{Q%
}^{s}(w)$ s.t. $h\cdot \check{q} =\clc(w)$ and $\check{q}_{\rho(1)} = s_1$. We
now show that $s_1 = q^*_{\rho(1)}(w)$. First note that $s_1 \le q^*_{\rho(1)}$
since from \eqref{eq:eq834} there is an $i^* \in N_{\rho(1)}$ such that 
\begin{equation*}
q^*_{\rho(1)}(w) = w_{i^*}\mu_{\rho(1)} = (G \check q)_{i^*} \mu_{\rho(1)}
\ge \check q_{\rho(1)} = s_1,
\end{equation*}
where the second equality holds since $\check q \in \mathcal{Q}^{s}(w)$ and
the next inequality is a consequence of the fact that $i^* \in N_{\rho(1)}$.
We now show that in fact the inequality can be replaced by equality. We
argue by contradiction and suppose that $s_{1}<{q}_{\rho(1)}^{*}(w)$. For all $%
i\in N_{\rho(1)}$ define 
\begin{equation}  \label{eq:eq517}
 j^*(i)=\arg \max_{j \neq \rho(1):i\in N_{j}}\left\{ \frac{\check{q}_{j}%
}{\mu _{j}}\right\}
\end{equation}%
and note that for any $i \in N_{\rho(1)}$ 
\begin{align*}
\frac{\check q_{j^*(i)}}{\mu_{j^*(i)}} \ge \frac{1}{J-1} \left( \sum_{j: i
\in N_j} \frac{\check q_j}{\mu_j} - \frac{\check q_{\rho(1)}}{\mu_{\rho(1)}}%
\right) > \frac{1}{J} \left( w_i - \frac{\check q_{\rho(1)}}{\mu_{\rho(1)}}%
\right) \ge \frac{1}{J} \left(\frac{q^*_{\rho(1)}(w)-\check q_{\rho(1)}}{%
\mu_{\rho(1)}}\right),
\end{align*}
where the second inequality uses the fact that $\check q \in \mathcal{Q}%
^{s}(w)$ while the third uses \eqref{eq:eq834} once more. Thus, 
\begin{equation}
\min_{i\in N_{\rho(1)}}\left\{ \frac{\check{q}_{j^*(i)}}{\mu _{j^{\ast
}(i)}}\right\} >\frac{q_{\rho(1)}^{*}(w)-s_{1}}{J\mu _{\rho(1)}}\text{.}
\label{eq:miniinnp}
\end{equation}%
We can choose a subset  $M\in \mathcal{M}^{\cls^s,\rho(1)}$ such that $M\subset
\left\{  j^*(i):i\in N_{\rho(1)}\right\} $. From the definition of $M$, 
$\sum_{j\in M} K_j - K_{\rho(1)}$ is a nonnegative vector. Since $\rho(1)
\in \mathcal{O}^{\cls^m}$, due to Definition \ref{def:optJob} 
\begin{equation}
\mu_{\rho(1)} h_{\rho(1)} + \clc(\sum_{j\in M} K_j - K_{\rho(1)}) \le
\clc(\sum_{j\in M} K_j).
\label{eq:eqmurho1}
\end{equation}
Thus there exists $v^{1}\in \mathcal{Q}^{s}\left( \sum_{j\in M}K_j-K_{\rho(1)}\right) $
i.e., 
\begin{equation}
\sum_{j\in M}K_j-K_{\rho(1)} = G v^{1} = \sum_{j=1}^JK_j b^{1}_{j}  \mbox{ where } v^{1}_{j}
= b^{1}_{j}\mu_j \mbox{ for } j \in \mathbb{N}_J,\label{eq:eqkjkrho}
\end{equation}
such that
\begin{equation}
\sum_{j=1}^J h_{j}b^{1}_{j}\mu_j  = h \cdot v^{1} = \clc(\sum_{j\in M}K_j-K_{\rho(1)}).
\label{eq:hjb1j}
\end{equation}
Furthermore, $b^{1}_{\rho(1)} = 0$, since if $b^{1}_{\rho (1)}>0$ then 
$
\sum_{j\in M}K_{i,j}-K_{i,\rho (1)}\geq 1
$
for all $i\in N_{\rho (1)}$, so that for any $l\in M$ we have
$
\sum_{j\in M\setminus \{l\}}K_{j}-K_{\rho (1)}\geq 0
$
which means $M$ is not minimal and contradicts $M\in \mathcal{M}^{\cls^{s},\rho (1)}$.

From \eqref{eq:eqmurho1} and \eqref{eq:hjb1j} we have
\begin{equation}
h_{\rho(1)}\mu _{\rho(1)}+\sum_{j=1}^{J}h_{j}b^{1}_{j}\mu _{j}\leq \sum_{j\in
M}h_{j}\mu _{j}\text{.}\label{eq:15.1}
\end{equation}%
Let 
\begin{equation}  \label{eq:eq651}
u_{1}\doteq \min_{j\in M}\left\{ \frac{\check{q}_{j}}{\mu _{j}}\right\}
\end{equation}%
Since $M \subset \left\{j^*(i):i\in N_{\rho(1)}\right\}$ , from \eqref{eq:miniinnp}
$
u_1 \ge \frac{q_{\rho(1)}^{*}(w)-s_{1}}{J\mu _{\rho(1)}}.
$
Define $\tilde{q}\in \mathbb{R}_+^J$ by 
\begin{equation}  \label{eq:eq653}
\tilde{q}_{\rho(1)}=\check{q}_{\rho(1)}+u_{1}\mu _{\rho(1)}, \mbox{ and } 
\tilde{q}_{j}=\check{q}_{j}-\mathbf{1}_{\{j\in M\}}u_{1}\mu
_{j}+u_{1}b^1_{j}\mu _{j} \mbox{ for } j\neq \rho(1).
\end{equation}
By definition of $u_1$, $\tilde q \in \mathbb{R}_+^J$. Also, 
\begin{eqnarray*}
w &=&\sum_{j=1}^{J}K_j\left( \frac{\check{q}_{j}}{\mu _{j}}-\mathbf{1}%
_{\{j\in M\}}u_{1}\right) +u_{1}\sum_{j\in M}K_j \\
&=&\sum_{j=1}^{J}K_j\left( \frac{\check{q}_{j}}{\mu _{j}}-\mathbf{1}_{\{j\in
M\}}u_{1}\right) +u_{1}\sum_{j=1}^{J}K_jb^{1}_{j}+u_{1}K_{\rho(1)} \\
&=&\sum_{j=1}^{J}K_{j}\frac{\tilde{q}_{j}}{\mu _{j}},
\end{eqnarray*}%
where the second equality uses \eqref{eq:eqkjkrho} and last equality uses the observation that $b^{1}_{\rho(1)}=0$. Thus $%
\tilde{q}\in \mathcal{Q}^{s}(w)$. Furthermore, 
\begin{eqnarray*}
\clc(w) &=&\sum_{j=1}^{J}h_{j}\left( \check{q}_{j}-\mathbf{1}_{\{j\in
M\}}u_{1}\mu _{j}\right) +u_{1}\sum_{j\in M}h_{j}\mu _{j} \\
&\geq &\sum_{j=1}^{J}h_{j}\left( \check{q}_{j}-\mathbf{1}_{\{j\in
M\}}u_{1}\mu _{j}\right) +u_{1}h_{\rho(1)}\mu
_{\rho(1)}+u_{1}\sum_{j=1}^{J}h_{j}b^{1}_{j}\mu _{j} \\
&=&\sum_{j=1}^{J}h_{j}\tilde{q}_{j} \ge \clc(w),
\end{eqnarray*}%
where the second line is from \eqref{eq:15.1} and the last inequality holds since $\tilde{q}\in \mathcal{Q}^{s}(w)$. So $h\cdot
\tilde q = \clc(w)$ and by definition of $s_1$, $\tilde q_{\rho(1)} \le s_1$.
However, since by assumption $s_1 < {q}_{\rho(1)}^{*}(w)$,
\begin{equation}
\tilde{q}_{\rho(1)}=s_{1}+ u_{1}\mu _{\rho(1)}\geq s_{1}+\frac{{q}%
_{\rho(1)}^{*}(w)-s_{1}}{J}>s_{1}  \label{eq:eq655}
\end{equation}
which is a contradiction. \ Thus we have shown $s_{1}={q}_{\rho(1)}^{*}(w)$.

Denote $\check q$ as $q^1$. Then $q^1_{\rho(1)} = q^*_{\rho(1)}(w)$. Note
that 
\begin{equation*}
\clc(w) = h \cdot q^1 = h_{\rho(1)} q^*_{\rho(1)} + \sum_{i\neq \rho(1)} h_i
q^1_i.
\end{equation*}
Let $w^1 = w - \frac{q^*_{\rho(1)}(w)}{\mu_{\rho(1)}}K_{\rho(1)}$. Then 
$
w^1 = G \left[q^1 - q^*_{\rho(1)}(w) e_{\rho(1)}\right]
$
and if for any $\tilde q \in \mathbb{R}_+^J$, $G\tilde q = w^1$, we have 
$
G\left[\tilde q + q^*_{\rho(1)}(w)e_{\rho(1)}\right] = Gq^1 = w
$
and so 
\begin{equation*}
h\cdot (\tilde q + q^*_{\rho(1)}(w)e_{\rho(1)}) \ge \clc(w) = h_{\rho(1)}
q^*_{\rho(1)}(w) + \sum_{i\neq \rho(1)} h_i q^1_i.
\end{equation*}
Thus $h\cdot \tilde q \ge \sum_{i\neq \rho(1)} h_i q^1_i$ and since $\tilde q
$ is arbitrary vector in $\mathbb{R}_+^J$ satisfying $G\tilde q = w^1$ 
\begin{equation*}
\clc(w^1) = h\cdot q^1 - h_{\rho(1)}q^*_{\rho(1)}(w) = \clc(w) -
h_{\rho(1)}q^*_{\rho(1)}(w).
\end{equation*}
We now proceed via induction. Suppose that for some $k \in \{2, \ldots, m\}$
and all $w \in \mathbb{R}_+^I$ 
\begin{equation}  \label{eq:eq651b}
\clc(w)=\sum_{l=1}^{k-1}h_{\rho(l)}{q}_{\rho(l)}^{*}(w)+\clc\left( w^{k-1}\right)
\end{equation}%
where 
\begin{equation*}
w^{k-1}=w-\sum_{l=1}^{k-1} \frac{{q}_{\rho(l)}^{*}(w)}{\mu_{\rho(l)}}%
K_{\rho(l)}\text{.}
\end{equation*}%
Note that we have shown \eqref{eq:eq651b} for $k=2$. With $\bar q$ as in
Theorem \ref{thm:costAchieved} $\bar{q}\left( w^{k-1}\right) \in \mathcal{Q}^{s}\left(
w^{k-1}\right) $ and 
\begin{equation*}
\clc\left( w^{k-1}\right) = \bar{q}\left( w^{k-1}\right)\cdot h \text{.}
\end{equation*}%
Define 
\begin{equation*}
s_{k}=\sup \left\{ q_{\rho(k)}:q\in \mathcal{Q}^{s}\left( w^{k-1}\right) ,
q\cdot h =\clc(w^{k-1})\right\}.
\end{equation*}%
Then there is $\check q \in \mathcal{Q}^{s}(w^{k-1})$ such that  $\check{q}%
_{\rho(k)}=s_{k}$,  and $\check{q}\cdot h =\clc(w^{k-1})$.
Also, using \eqref{eq:eq834b} we have for every $l<k$ an $i^* \in N_{\rho(l)}
$ such that 
$
\frac{q^*_{\rho(l)}(w)}{\mu_{\rho(l)}} = w_{i^*}^{l-1}.
$
Thus, 
\begin{equation*}
0\le w^{k-1}_{i^*} \le w_{i^*} - \sum_{u=1}^l G_{i^*,\rho(u)} q^*_{\rho(u)}(w) =
w_{i^*}^{l-1} - \frac{q^*_{\rho(l)}(w)}{\mu_{\rho(l)}} = 0.
\end{equation*}
Consequently, 
\begin{equation}  \label{eq:eq608}
\mbox{ for every } l \in {1, \ldots k-1} \mbox{ there is an } i \in
N_{\rho(l)} \mbox{ such that } w_i^{k-1}=0.
\end{equation}
Since $G\check q = w^{k-1}$, this in turn says that $\check q_{\rho(l)} = 0$
for $l \in {0, 1, \ldots k-1}$. Next, as for the case $k=1$, we can show that 
$s_k = {q}_{\rho(k)}^{*}(w)$. Indeed, the inequality $s_k \le q^*_{\rho(k)}$
follows on noting from \eqref{eq:eq834b} that for some $i^* \in N_{\rho(k)}$ 
\begin{equation*}
q^*_{\rho(k)}(w) = w_{i^*}^{k-1}\mu_{\rho(k)} = (G \check q)_{i^*}
\mu_{\rho(k)} \ge \check q_{\rho(k)} = s_k.
\end{equation*}
Next suppose $s_{k}<{q}_{\rho(k)}^{*}(w)$. Define $j^*(i)$ as in %
\eqref{eq:eq517} replacing $\rho(1)$ with $\rho(k)$, then as before (using %
\eqref{eq:eq834b} instead of \eqref{eq:eq834}) 
\begin{equation}
\min_{i\in N_{\rho(k)}}\left\{ \frac{\check{q}_{j^*(i)}}{\mu _{j^*(i)}}
\right\} >\frac{q_{\rho(k)}^{*}(w)-s_{k}}{J\mu _{\rho(k)}}\text{.}
\label{eq:eq607}
\end{equation}

Thus from \eqref{eq:eq607} we have that $j^*(i) \notin \{\rho(1), \ldots
\rho(k)\}$. We next claim that the set of resources associated with $\rho(l)$
for any $l<k$ is not a subset of the set of resources associated with $%
\{j^*(i): i \in N_{\rho(k)}\}$. Indeed, if that were the case for some $l<k$%
, then we will have 
\begin{equation}  \label{eq:eq612}
\sum_{i \in N_{\rho(k)}} K_{j^*(i)} - K_{\rho(l)} \ge 0.
\end{equation}
From \eqref{eq:eq608} there is an $i^*$ such that $K_{i^*,\rho(l)}=1$ and $%
w_{i^*}^{k-1}=0$. Then from \eqref{eq:eq612} $K_{i^*, j^*(i)} =1$ for some $%
i \in N_{\rho(k)}$. Since from \eqref{eq:eq607} $\check q_{j^*(i)} >0$, we
have $w_{i^*}^{k-1}>0$ which is a contradiction. This proves the claim, namely 
$N_{\rho(l)} \not \subset \bigcup_{i \in N_{\rho(k)}} N_{j^*(i)}$ for $l = 1, \ldots, k-1$.

We  can now choose a subset $M^{k}\in \mathcal{M}_{F_k }^{\cls^s\setminus F_k
,\rho(k)}$ such that $M^{k}\subset \left\{ j^{\ast}(i):i\in
N_{\rho(k)}\right\} $.

Since by definition $\rho(k) \in \mathcal{O}^{E_k}_{F_k}$ and by our choice $%
M^k \in \mathcal{M}_{F_k }^{\cls^s\setminus F_k ,\rho(k)}$, we have from
Definition \ref{def:optJob} that there exists $b^{k}\in \mathbb{R}_+^J$ such that
$b^k_{\rho(k)}=0$ and 
\begin{equation*}
K_{\rho(k)}+\sum_{j=1}^{J}K_{j}b^{k}_{j}=\sum_{j\in M^{k}}K_{j}, \mbox{ and }
h_{\rho(k)}\mu _{\rho(k)}+\sum_{j=1}^{J}h_{j}b^{k}_{j}\mu _{j}\leq \sum_{j\in
M^{k}}h_{j}\mu _{j}.
\end{equation*}%
With $u_k$ as defined in \eqref{eq:eq651} with $M$ replaced by $M^k$ (and
with $\check q$ as above) 
\begin{equation*}
u_{k} \ge \frac{q_{\rho(k)}^{*}(w)-s_{k}}{J\mu _{\rho(k)}}.
\end{equation*}%
Define $\tilde q$ as in \eqref{eq:eq653} replacing $\rho(1)$ with $\rho(k)$, $u_{1}$ with $u_{k}$,
and $M$ with $M^k$. Then as before $h\cdot \tilde q = \clc(w^{k-1})$ and $%
G\tilde q = w^{k-1}$; and as in the proof of \eqref{eq:eq655} we see using %
\eqref{eq:eq607} that $\tilde q_{\rho(k)} > s_k$ which contradicts the
definition of $s_k$. This completes the proof that $s_k = {q}%
_{\rho(k)}^{*}(w)$.

Setting $q^k = \check q$ we have that $q^k_{\rho(k)} = q^*_{\rho(k)}(w)$. Also,
recalling that 
\begin{equation*}
w^k = w^{k-1} - \frac{q^*_{\rho(k)}(w)}{\mu_{\rho(k)}}K_{\rho(k)}
\end{equation*}
and since $G q^k = w^{k-1}$, we have $G[q^k - q^*_{\rho(k)}(w) e_{\rho(k)}] =
w^k$ and $h\cdot (q^k - q^*_{\rho(k)}(w) e_{\rho(k)}) = \clc(w^{k-1}) -
q^*_{\rho(k)}(w) h_{\rho(k)}$. Furthermore, using the fact that $h\cdot q^k =
\clc(w^{k-1})$, we have that if for $\tilde q \in \mathbb{R}^J_+$, $G\tilde q =
w^k$, then $h\cdot \tilde q \ge \clc(w^{k-1}) - q^*_{\rho(k)}(w) h_{\rho(k)}$.
Thus we have that $\clc(w^k) = \clc(w^{k-1}) - q^*_{\rho(k)}(w) h_{\rho(k)}$.
Combining this with the induction hypothesis \eqref{eq:eq651b}, we have that %
\eqref{eq:eq651b} holds with $k-1$ replaced with $k$. This completes the
induction step and proves \eqref{eq:eq651b} for all $k=2, \ldots m+1$, in
particular 
\begin{equation}  \label{eq:eq651c}
\clc(w)=\sum_{l=1}^{m}h_{\rho(l)}{q}_{\rho(l)}^{*}(w)+\clc\left( w^{m}\right)
\end{equation}%
where  
\begin{equation}  \label{eq:eq750}
w^{m}=w-\sum_{l=1}^{m}K_{\rho(l)}{q}_{\rho(l)}^{*}(w)\text{.}
\end{equation}%
Next, using \eqref{eq:eq608} with $k-1$ replaced with $m$ we see that for
any $q \in \mathcal{Q}^s(w^m)$, $q_{\rho(l)} =0$ for all $l = 1, \ldots m$.
Namely, 
\begin{equation*}
\clc(w^m) = \sum_{j \in \cls^1} h_j \mu_j w^m_{\hat i(j)}.
\end{equation*}
From the definition of $w^m$ in \eqref{eq:eq750} and the definition of $%
q^*_j(w)$ for $j \in \cls^1$ in \eqref{eq:eq753} we then have that 
\begin{equation*}
\clc(w^m) = \sum_{j\in \cls^1}h_{j}{q}_{j}^{*}(w),\; w = \sum_{l=1}^{m}K_{\rho(l)}{%
q}_{\rho(l)}^{*}(w) + \sum_{j \in \cls^1} K_{j}{q}_{j}^{*}(w).
\end{equation*}
This proves \eqref{eq:eq755} and the statement that $q^*(w) \in \mathcal{Q}%
^s(w)$, and completes the proof of the theorem.
\end{proof}

Analogous to  $\zeta _{i}^{k}$ introduced in Section \ref{sec:secallostra}, let 
\begin{equation}\label{eq:eqzetaiz}
\zeta _{i}^{0}=\{j\in \cls^p:K_{i,j}=1\}
\end{equation}%
be the set of primary jobs which impact node $i$.
\begin{theorem}
 \label{thm:costInefIneq} There exists  $B\in (0,\infty)$ such
that for any $q\in \mathbb{R}_+^{J}$ and the corresponding
workload, $w=Gq$,
we have 
\begin{equation*}
\left\vert  h\cdot q -\clc(w)\right\vert \leq B\left(
\sum_{k=1}^{m }\min_{i\in N_{\rho (k)}}\left\{ \sum_{j\in \zeta
_{i}^{k}}q_{j}\right\} +\sum_{i=1}^{I}\sum_{j\in \zeta _{i}^{0}}q_{j}\right) 
\text{.}
\end{equation*}
\end{theorem}

\begin{proof}
Recall from Theorem \ref{thm:restCostJobOrd}\ that with $q^{\ast} = q^{\ast}(w)$
\begin{equation*}
\clc(w)=q^{\ast }\cdot h =\sum_{k=1}^{m }h_{\rho
(k)}q_{\rho (k)}^{\ast }+\sum_{j\in \mathit{\cls}^{1}}h_{j}q_{j}^{\ast }%
\text{.}
\end{equation*}%
Since
\begin{equation*}
\frac{q_{\rho (1)}^{\ast }}{\mu _{\rho (1)}} =\min_{i\in N_{\rho (1)}}\left\{
w_{i}\right\} =\min_{i\in N_{\rho (1)}}\left\{ \sum_{j\in \zeta
_{i}^{1}} \frac{q_{j}}{\mu _{j}}\right\} + \frac{q_{\rho (1)}}{\mu _{\rho (1)}}
\end{equation*}%
we have%
\begin{equation*}
q_{\rho (1)}=q_{\rho (1)}^{\ast }-\min_{i\in N_{\rho (1)}}\left\{ \sum_{j\in
\zeta _{i}^{1}}\frac{q_{j}}{\mu _{j}}\right\} \mu _{\rho (1)}
\end{equation*}%
from which we have
\begin{equation*}
\frac{1}{\mu _{\rho (1)}}\left\vert q_{\rho (1)}^{\ast }-q_{\rho (1)}\right\vert \leq
\min_{i\in N_{\rho (1)}}\left\{ \sum_{j\in \zeta _{i}^{1}} \frac{q_{j}}{\mu_{j}}\right\} \text{.}
\end{equation*}%
In 
general, for $2\leq k\leq m $ we have%
\begin{eqnarray*}
\frac{q_{\rho (k)}^{\ast }}{\mu _{\rho (k)}} &=&\min_{i\in N_{\rho (k)}}\left\{
w_{i}-\sum_{l=1}^{k-1}K_{i,\rho (l)} \frac{q^{*}_{\rho (l)}}{\mu _{\rho
(l)}}\right\}  \\
&=&\min_{i\in N_{\rho (k)}}\left\{ \sum_{j\in \zeta _{i}^{k}} \frac{q_{j}}{\mu_{j}}
-\sum_{l=1}^{k-1}K_{i,\rho (l)}\frac{(q_{\rho (l)}^{\ast }-q_{\rho (l)})}{\mu
_{\rho (l)}}\right\} + \frac{q_{\rho (k)}}{\mu _{\rho (k)}}
\end{eqnarray*}%
which gives%
\begin{equation*}
\frac{1}{\mu _{\rho (k)}}\left|q_{\rho (k)}^{\ast }-q_{\rho (k)}\right|\leq \min_{i\in N_{\rho
(k)}}\left\{ \sum_{j\in \zeta _{i}^{k}} \frac{q_{j}}{\mu_j} \right\}
+\sum_{l=1}^{k-1}\frac{\left\vert q_{\rho (l)}^{\ast }-q_{\rho (l)}\right\vert}{
\mu _{\rho (l)}}.
\end{equation*}%
Consequently
 for $k\in \{2,...,m \}$ we have 
\begin{align*}
\frac{1}{\mu _{\rho (k)}}\left\vert q_{\rho (k)}^{\ast }-q_{\rho (k)}\right\vert 
&\leq \min_{i\in N_{\rho (k)}}\left\{ \sum_{j\in \zeta _{i}^{k}}\frac{q_{j}}{\mu_j}
\right\} +\sum_{l=0}^{k-2}2^{l}\min_{i\in N_{\rho (k-1-l)}}\left\{ \sum_{j\in \zeta
_{i}^{k-1-l}}\frac{q_{j}}{\mu _{j}}\right\} \text{.}
\end{align*}%
For $j\in \mathit{S}^{1}$ we have with $i = \hat i(j)$
\begin{align*}
\frac{q_{j}^{\ast }}{\mu _{j}} =
w_{i}-\sum_{k=1}^{m }K_{i,\rho (k)}\frac{q_{\rho (k)}^{\ast }}{\mu _{\rho (k)}}
=  \sum_{j'\in \zeta _{i}^{0}} \frac{q_{j'}}{\mu_{j'}}-\sum_{k=1}^{m}K_{i,\rho (k)}
\frac{(q_{\rho (k)}^{\ast }-q_{\rho (k)})}{\mu_{\rho (k)}} + \frac{q_{j}}{\mu _{j}}
\end{align*}%
which gives%
\begin{equation*}
\frac{1}{\mu _{j}}\left|q_{j}^{\ast }-q_{j}\right|\leq  \sum_{j'\in \zeta
_{i}^{0}} \frac{q_{j'}}{\mu_j'}  +\sum_{l=1}^{m }\frac{\left\vert q_{\rho
(l)}^{\ast }-q_{\rho (l)}\right\vert}{\mu _{\rho (l)}}
\end{equation*}%
This, combined with our bounds on $\left\vert q_{\rho (k)}^{\ast }-q_{\rho
(k)}\right\vert$ for $k\in \{1,...,m \}$, gives the
 following bound for  $j \in \cls^1$
\begin{align*}
\frac{\left\vert q_{j}^{\ast }-q_{j}\right\vert}{\mu _{j}} \leq  \sum_{j'\in \zeta _{\hat i(j)}^{0}}
\frac{q_{j'}}{\mu _{j'}} 
+\sum_{l=0}^{m-1}2^{l}\min_{i\in N_{\rho (m -l)}}\left\{ \sum_{j'\in \zeta
_{i}^{m -l}}\frac{q_{j'}}{\mu _{j'}}\right\} \text{.}
\end{align*}%
Finally, for $j\in \cls^p$ we have 
\begin{equation*}
\frac{\left\vert q_{j}^{\ast }-q_{j}\right\vert}{\mu _{j}} = \frac{q_{j}}{\mu _{j}}\leq 
\min_{i\in N_{j}}\left\{ \sum_{j'\in \zeta
_{i}^{0}} \frac{q_{j'}}{\mu _{j'}}\right\} \text{.}
\end{equation*}%
Combining the above bounds 
\begin{align*}
h\cdot q  &= h\cdot q^{\ast}
+h\cdot (q-q^{\ast})  \leq  \clc(w)+ \sum_{j \in \NJ} h_j |q_j-q^{\ast}|
\\
&\leq \clc(w)+\max_{j}\{h_{j}\}\sum_{j}\left\vert q_{j}-q^{\ast}_{j}\right\vert  \\
&\leq \clc(w)+\max_{j}\{h_{j}\}\max_{j}\{\mu_{j}\}J^22^{J}\left( \sum_{k=1}^{m }\min_{i\in N_{\rho
(k)}}\left\{ \sum_{j\in \zeta _{i}^{k}}\frac{q_{j}}{\mu _{j}}\right\}
+\sum_{i=1}^{I}\sum_{j\in \zeta _{i}^{0}}\frac{q_{j}}{\mu _{j}}\right)  \\
&\leq \clc(w)+\frac{\max_{j}\{h_{j}\}\max_{j}\{\mu_{j}\}}{\min_{j}\{\mu _{j}\}}J^22^{J}\left(
\sum_{k=1}^{m }\min_{i\in N_{\rho (k)}}\left\{ \sum_{j\in \zeta
_{i}^{k}}q_{j}\right\} +\sum_{i=1}^{I}\sum_{j\in \zeta _{i}^{0}}q_{j}\right) 
\text{.}
\end{align*}%
Because $h\cdot q \geq \clc(w)$ we have 
\begin{equation*}
\left\vert h\cdot q -\clc(w)\right\vert \leq B\left(
\sum_{k=1}^{m }\min_{i\in N_{\rho (k)}}\left\{ \sum_{j\in \zeta
_{i}^{k}}q_{j}\right\} +\sum_{i=1}^{I}\sum_{j\in \zeta _{i}^{0}}q_{j}\right) 
\end{equation*}%
where 
$
B=\frac{\max_{j}\{h_{j}\}\max_{j}\{\mu_{j}\}}{\min_{j}\{\mu _{j}\}}J^22^{J}\text{.}
$
\end{proof}

\section{Some Properties of the Rate Allocation Policy}
In this section we record some important properties of the rate allocation policy $x(\cdot)$ introduced in
Definition \ref{def:workAllocScheme}. 
Throughout this section $y(t), x(t)$ and $\mathcal{E}_{j}^r(t)$  will be as in
Definition \ref{def:workAllocScheme}.
Our first result shows that $x$ satisfies basic conditions for admissibility, namely, it is nonnegative and satisfies the capacity constraint. 
\label{sec:ratallpol}
\begin{lemma}
	\label{lem:admispol}
For all $t\ge 0$, $x(t) \ge 0$ and $Kx(t) \le C$.
\end{lemma}

\begin{proof}
For the first statement in the lemma it suffices to show that $y_{j}(t)\geq 0
$ for all $j \in \mathbb{N}_J$ and $t\ge 0$. From definition of $\delta $ it
is clear that $y_{j}(t)\geq 0$ for all $j\in \mathbb{N}_{J}\setminus \cls^1$
and for $j\in \cls^1$ with $\hat i(j)\notin \varpi ^{r}(t)$. Consider now a $%
j\in \cls^1$ for which $\hat i(j)\in \varpi ^{r}(t)$. Then 
\begin{equation*}
y_{j}(t)=C_{\hat i(j)}-\sum_{l\neq j:K_{\hat i(j),l}=1}y_{l}(t)\text{.}
\end{equation*}%
Also note that 
\begin{eqnarray*}
\sum_{l\neq j:K_{\hat i(j),l}=1}y_{l}(t) \leq \sum_{l\neq j:K_{\hat
i(j),l}=1}\left( \varrho _{l}+\delta \right) \leq \sum_{l\neq j:K_{\hat
i(j),l}=1}\varrho _{l}+\frac{\min_{j'}\{\varrho _{j'}\}}{2}
\end{eqnarray*}
and thus since $K\varrho =C$
\begin{eqnarray*}
y_{j}(t) = C_{\hat i(j)}-\sum_{l\neq j:K_{\hat i(j),l}=1}y_{l}(t) \geq
C_{\hat i(j)}-\sum_{l\neq j:K_{\hat i(j),l}=1}\varrho _{l}-\frac{%
\min_{j'}\{\varrho _{j'}\}}{2} \geq \varrho _{j}-\frac{\min_{j'}\{\varrho _{j'}\}%
}{2} \geq 0.
\end{eqnarray*}%
This completes the proof of the first statement in the lemma. We now show
that $Kx(t) \le C$ for all $t\ge 0$. Let $i\in \mathbb{N}_{I}$ be arbitrary.
It suffices to show that for all $t\ge 0$,  
$
C_{i}\geq \sum_{j=1}^{J}K_{i,j}y_{j}(t)\text{.}
$
From definition of $y_j(t)$ for $j \in \cls^1$ in Definition \ref{def:workAllocScheme}, it is clear that when $i\in \varpi ^{r}(t)$, 
$
C_{i}=\sum_{j=1}^{J}K_{i,j}y_{j}(t)\text{.}
$
Finally, if $i\notin \varpi ^{r}(t)$, then Definition \ref%
{def:workAllocScheme} gives $y_{j}(t)<\varrho _{j}$ for all $j$ with $%
K_{i,j}=1$ and so 
\begin{eqnarray*}
\sum_{j=1}^{J}K_{i,j}y_{j}(t) <\sum_{j=1}^{J}K_{i,j}\varrho _{j} <C_{i}.
\end{eqnarray*}%
This completes the proof.
\end{proof}

The following two results are used in the proof of Theorem \ref{thm:IdleTimeExp}.
\begin{lemma}
For all $t\ge 0$ and $i\in \varpi
^{r}(t)$ such that $\sum_{j=1}^{J}K_{i,j}\mathcal{E}_{j}^r(t)=0, $ we have 
$C_{i}=\sum_{j=1}^{J}K_{i,j}x_{j}(t)\text{.} $
\label{lem:fullWorkloadCond}
\end{lemma}

\begin{proof}
Let $t\ge 0$ and $i\in \varpi ^{r}(t)$ satisfy $\sum_{j=1}^{J}K_{i,j}%
\mathcal{E}_{j}^r(t)=0$. Then for all $j$ with $K_{i,j}=1$ we have $%
x_{j}(t)=y_{j}(t)$ and so it suffices to prove that $C_{i}=%
\sum_{j=1}^{J}K_{i,j}y_{j}(t)\text{.} $ However, this is an immediate
consequence of the definition of $y_j(t)$ for $j \in \cls^1$ and $\hat i(j) \in
\varpi ^{r}(t)$ in Definition \ref{def:workAllocScheme}.
\end{proof}

From Condition \ref{cond:HT1} we can find  $\hat{R} \in (0,\infty)$ such that
for all $r\geq \hat{R}$ and $j\in \mathbb{N}_{J}$ we have%
\begin{equation}
\left\vert \varrho _{j}-\varrho _{j}^{r}\right\vert \leq 2^{-2m -6}\frac{%
\delta }{J }\text{,}\;\;
2\lambda _{j}\geq \lambda _{j}^{r}\geq \lambda _{j}/2\text{, and }\;
2\mu _{j}\geq \mu _{j}^{r}\geq \mu _{j}/2\text{.}\label{def:Rhat} 
\end{equation}
For the rest of this work we will assume without loss of generality that $%
r\ge \hat R$.

\begin{lemma}
\label{lem:posDriftCond}For all $t\ge 0$ and $j\in \mathbb{N}_{J}$ if $%
c_{1}r^{\alpha }\leq Q_{j}^{r}(t)<c_{2}r^{\alpha }$ then 
\begin{equation*}
\lambda _{j}^{r}-\mu _{j}^{r}x_{j}(t)\geq \mu _{j}2^{-2m -5}\frac{\delta }{J 
}\text{.}
\end{equation*}
\end{lemma}

\begin{proof}
Note that if $\mathcal{E}_{j}^r(t)=1$ then $x_{j}(t)=0$ which, since $r
\ge \hat R$, implies on recalling the definition of $\delta$ from Definition \ref{def:workAllocScheme} that 
\begin{equation*}
\lambda _{j}^{r}-\mu _{j}^{r}x_{j}(t)=\lambda _{j}^{r}\geq \lambda
_{j}/2=\mu _{j}\varrho _{j}/2\geq \mu _{j}\delta.
\end{equation*}%
Thus the result holds in this case.

We now consider the case $\mathcal{E}_{j}^r(t)=0$ so that $%
x_{j}(t)=y_{j}(t)$.  If $j\in \mathbb{N}_{J}\setminus \cls^1$ or $j\in \cls^1$
and $\hat i(j)\notin \varpi ^{r}(t)$, Definition \ref{def:workAllocScheme}
gives%
\begin{equation*}
y_{j}(t)\leq \varrho _{j}-2^{-2m-3}\frac{\delta }{J }
\end{equation*}%
which combined with  \eqref{def:Rhat} implies 
\begin{eqnarray*}
\lambda _{j}^{r}-\mu _{j}^{r}x_{j}(t) &\geq& \lambda _{j}^{r}-\mu
_{j}^{r}\left( \varrho _{j}-2^{-2m -3}\frac{\delta }{J }\right) = \mu
_{j}^{r}\left( \varrho _{j}^{r}-\varrho _{j}\right) +\mu _{j}^{r}2^{-2m -3}%
\frac{\delta }{J } \\
&\geq& -\mu _{j}^{r}2^{-2m -6}\frac{\delta }{J }+\mu _{j}2^{-2m -4}\frac{%
\delta }{J } \ge \mu _{j}2^{-2m -5}\frac{\delta }{J }
\end{eqnarray*}%
and the result again holds. \ Finally we consider the remaining case, namely 
$j\in \cls^1$, $\mathcal{E}_{j}^r(t)=0$ and $\hat i(j)\in \varpi ^{r}(t)$.
We will consider two sub-cases.\newline

\noindent \textbf{Case 1:} $\zeta _{\hat i(j)}^{0}\cap \sigma ^{r}(t) \neq
\emptyset$. Let $l^{\ast }\in \zeta _{\hat i(j)}^{0}\cap \sigma ^{r}(t)$.
Then $y_{l^{\ast }}(t)=\varrho _{l^{\ast }}+\delta $ and 
\begin{align*}
y_{j}(t) = C_{\hat i(j)}-\sum_{l\neq j:K_{\hat i(j),l}=1}y_{l}(t) 
= C_{\hat i(j)}-y_{l^{\ast }}(t)-\sum_{k=1}^m K_{\hat i(j),\rho (k)}y_{\rho
(k)}(t)-\sum_{l\in \zeta _{\hat i(j)}^{0}:l\neq l^{\ast }}y_{l}(t)\text{.}
\end{align*}%
Furthermore, 
\begin{align*}
-\sum_{k=1}^{m }K_{\hat i(j),\rho (k)}y_{\rho (k)}(t) \leq -\sum_{k=1}^{m
}K_{\hat i(j),\rho (k)}\left( \varrho _{\rho (k)}-2^{k-m -2}\delta \right) 
\leq -\sum_{k=1}^{m }K_{\hat i(j),\rho (k)}\varrho _{\rho (k)}+\delta
\left( 1-2^{-m -2}\right)
\end{align*}%
and%
\begin{align*}
-\sum_{l\in \zeta _{\hat i(j)}^{0}:l\neq l^{\ast }}y_{l}(t) \leq
-\sum_{l\in \zeta _{\hat i(j)}^{0}:l\neq l^{\ast }}\left( \varrho_{l}-2^{-m
-3}\frac{\delta }{J }\right) 
\leq -\sum_{l\in \zeta _{\hat i(j)}^{0}:l\neq l^{\ast }}\varrho _{l}+2^{-m
-3}\delta.
\end{align*}%
Consequently 
\begin{eqnarray*}
y_{j}(t) \leq C_{\hat i(j)}-\sum_{l\neq j:K_{\hat i(j),l}=1}\varrho
_{l}-\delta +\delta \left( 1-2^{-m -2}\right) +2^{-m -3}\delta \leq \varrho
_{j}-\delta 2^{-m -2}
\end{eqnarray*}%
which combined with  \eqref{def:Rhat}\ gives 
\begin{eqnarray*}
\lambda _{j}^{r}-\mu _{j}^{r}y_{j}(t) &\geq &\lambda _{j}^{r}-\mu
_{j}^{r}\left( \varrho _{j}-2^{-m -2}\delta \right) \geq \mu _{j}^{r}\left(
\varrho _{j}^{r}-\varrho _{j}\right) +\mu _{j}^{r}2^{-m -2}\delta \geq -\mu
_{j}^{r}2^{-2m -6}\frac{\delta }{J }+\mu _{j}2^{-m -3}\delta \\
&\geq &-\mu _{j}2^{-2m -5}\frac{\delta }{J }+\mu _{j}2^{-m -3}\delta \geq
\mu _{j}2^{-m -4}\delta
\end{eqnarray*}%
and the result holds. \newline

\noindent \textbf{Case 2:} $\zeta _{\hat i(j)}^{0}\cap \sigma
^{r}(t)=\emptyset $. In this case the assumption $\hat i(j)\in \varpi ^{r}(t)
$ implies that there exists some $k\in \mathbb{N}_m$ such that $K_{\hat
i(j),\rho (k)}=1$ and $\rho (k)\in \sigma ^{r}(t)$. \ Let 
\begin{equation*}
k^{\ast }=\max \{k\in \mathbb{N}_m:K_{\hat i(j),\rho (k)}=1\text{ and }\rho
(k)\in \sigma ^{r}(t)\}\text{.}
\end{equation*}%
Consequently $\zeta _{\hat i(j)}^{k^{\ast }}\cap \sigma ^{r}(t)=\emptyset $
and $\rho (k^{\ast })\in \sigma ^{r}(t)$ so 
$
y_{\rho (k^{\ast })}=\varrho _{\rho (k^{\ast })}+2^{k^{\ast }-m -2}\delta 
$.
Recall that%
\begin{eqnarray*}
y_{j}(t) &=&C_{\hat i(j)}-\sum_{l\neq j:K_{\hat i(j),l}=1}y_{l}(t) \\
&=&C_{\hat i(j)}-y_{\rho (k^{\ast })}-\sum_{k=1}^{k^{\ast }-1}K_{\hat
i(j),\rho (k)}y_{\rho (k)}(t)-\sum_{k=k^{\ast }+1}^{m }K_{\hat i(j),\rho
(k)}y_{\rho (k)}(t)-\sum_{l\in \zeta _{\hat i(j)}^{0}}K_{\hat i(j),l}y_{l}(t)%
\text{.}
\end{eqnarray*}%
For the third term on the right side, we have 
\begin{eqnarray*}
-\sum_{k=1}^{k^{\ast }-1}K_{\hat i(j),\rho (k)}y_{\rho (k)}(t) &\leq
&-\sum_{k=1}^{k^{\ast }-1}K_{\hat i(j),\rho (k)}\left( \varrho _{\rho
(k)}-2^{k-m -2}\delta \right) \\
&\leq &-\sum_{k=1}^{k^{\ast }-1}K_{\hat i(j),\rho (k)}\varrho _{\rho
(k)}+\left( 1-2^{-k^{\ast }+1}\right) 2^{k^{\ast }-m -2}\delta \\
&\leq &-\sum_{k=1}^{k^{\ast }-1}K_{\hat i(j),\rho (k)}\varrho _{\rho
(k)}+2^{k^{\ast }-m -2}\delta -2^{-m -1}\delta \text{.}
\end{eqnarray*}%
By the definition of $k^{\ast }$ for all $k\in \{k^{\ast }+1,...,m \}$ if $%
K_{\hat i(j),\rho (k)}=1$ we have $\rho (k)\notin \sigma ^{r}(t)$ and
$\zeta^k_{\hat i(j)}\cap \sigma^r(t) = \emptyset$,
consequently 
$
y_{\rho (k)}(t)=\varrho _{\rho (k)}-2^{-k-m -2}\delta \text{.}
$
This gives%
\begin{align*}
-\sum_{k=k^{\ast }+1}^{m }K_{\hat i(j),\rho (k)}y_{\rho (k)}(t)
&= -\sum_{k=k^{\ast }+1}^{m }K_{\hat i(j),\rho (k)}\left( \varrho _{\rho
(k)}-2^{-k-m -2}\delta \right) \\
&\le -\sum_{k=k^{\ast }+1}^{m }K_{\hat i(j),\rho (k)}\varrho _{\rho
(k)}+2^{-k^{\ast }-m -2}\delta (1-2^{-m+k^*})\text{.}
\end{align*}%
Finally, by assumption, $\zeta _{\hat i(j)}^{0}\cap \sigma ^{r}(t)=\emptyset $
and therefore 
\begin{align*}
-\sum_{l\in \zeta _{\hat i(j)}^{0}}K_{\hat i(j),l}y_{l}(t) = -\sum_{l\in
\zeta _{\hat i(j)}^{0}}K_{\hat i(j),l}\left( \varrho _{l}-2^{-m -3}\frac{%
\delta }{J }\right) 
\leq  -\sum_{l\in \zeta _{\hat i(j)}^{0}}K_{\hat i(j),l}\varrho _{l}+2^{-m
-3}\delta \text{.}
\end{align*}%
This gives 
\begin{eqnarray*}
y_{j}(t) &\le C_{\hat i(j)}-\left(\sum_{l\neq j:K_{\hat i(j),l}=1}\varrho
_{l}\right)-2^{k^{\ast }-m -2}\delta +2^{k^{\ast }-m -2}\delta 
-2^{-m -1}\delta +2^{-k^{\ast }-m -2}\delta +2^{-m -3}\delta \\
&\le \varrho _{j}-2^{-m -3}\delta
\end{eqnarray*}
which combined with  \eqref{def:Rhat} implies 
\begin{eqnarray*}
\lambda _{j}^{r}-\mu _{j}^{r}y_{j}(t) &\geq &\lambda _{j}^{r}-\mu
_{j}^{r}\left( \varrho _{j}-2^{-m -3}\delta \right) = \mu _{j}^{r}\left(
\varrho _{j}^{r}-\varrho _{j}\right) +\mu _{j}^{r}2^{-m -3}\delta \geq -\mu
_{j}^{r}2^{-2m -6}\frac{\delta }{J }+\mu _{j}2^{-m -4}\delta \\
&\geq &-\mu _{j}2^{-2m -5}\frac{\delta }{J }+\mu _{j}2^{-m -4}\delta \geq
\mu _{j}2^{-m -5}\delta
\end{eqnarray*}%
and completes the proof.
\end{proof}

The following lemma will be used in the proofs of Propositions \ref{thm:initInef} 
and \ref{thm:runningInef}.
\begin{lemma}
	\label{lem:lem3_4}
 (a) Let $t\ge 0$ and $k\in \mathbb{N}_m$ be such that $\zeta
_{i'}^{k}\cap \sigma ^{r}(t)\neq \emptyset $ for all $i'\in N_{\rho (k)}$.
Then for any $i\in N_{\rho (k)}$ satisfying $\sum_{j\in \zeta _{i}^{k}}%
\mathcal{E}_{j}^r(t)=0$, we have 
\begin{equation}
\sum_{j\in \zeta _{i}^{k}}\left( \varrho _{j}^{r}-x_{j}(t)\right) \leq
-2^{-m -2}\delta \text{.}\label{eq:eq636}
\end{equation}%
(b) Let $i\in \mathbb{N}_{I}$ and $t\ge 0$ be such that $\zeta _{i}^{0}\cap
\sigma ^{r}(t)\neq \emptyset $ and $\sum_{j\in \zeta _{i}^{0}}\mathcal{E}%
_{j}^{r}(t)=0$. Then, we have 
\begin{equation*}
\sum_{j\in \zeta _{i}^{0}}\left( \varrho _{j}^{r}-x_{j}(t)\right) \leq
-2^{-2}\delta \text{.}
\end{equation*}
\end{lemma}

\begin{proof}
(a) Recall that we assume $r\geq \hat{R}$ and consequently \eqref{def:Rhat} holds. \ Let $k\in \mathbb{N}_m$ and $t\ge 0$
be such that $\zeta _{i'}^{k}\cap \sigma^{r}(t)\neq \emptyset $ for
all $i'\in N_{\rho (k)}$. Let $i\in N_{\rho (k)}$ be such that $\sum_{j\in
\zeta _{i}^{k}}\mathcal{E}_{j}^r(t)=0$. We need to show that \eqref{eq:eq636} holds for such an $i$.
Since $\zeta
_{i'}^{k}\cap \sigma ^{r}(t)\neq \emptyset $ for all $i'\in N_{\rho (k)}$,
Definition \ref{def:workAllocScheme} gives 
\begin{equation}
y_{\rho (k)}(t)=\varrho _{\rho (k)}-2^{k-m -2}\delta \text{.}\label{eq:eq751}
\end{equation}%
Since $\sum_{j\in \zeta _{i}^{k}}\mathcal{E}_{j}^r(t)=0$, for all $j\in
\zeta _{i}^{k}$,
$x_{j}(t)=y_{j}(t)$ so to prove \eqref{eq:eq636} it suffices to show
\begin{equation}
\sum_{j\in \zeta _{i}^{k}}\left( \varrho _{j}^{r}-y_{j}(t)\right) \leq
-2^{-m -2}\delta \text{.}\label{eq:eq752}
\end{equation}%
Due to the assumption that $\zeta _{i}^{k}\cap \sigma ^{r}(t)\neq \emptyset $
we have $i\in \varpi ^{r}(t)$ and consequently Definition \ref{def:workAllocScheme} gives 
\begin{equation*}
y_{\check j(i)}(t)=C_{i}-\sum_{j\neq \check j(i):K_{i,j}=1}y_{j}(t)\text{.}
\end{equation*}%
Therefore 
\begin{eqnarray*}
\sum_{j\in \zeta _{i}^{k}}y_{j}(t) &=&y_{\check j(i)}(t)+\sum_{j\in \zeta
_{i}^{k}:j\neq \check j(i)}y_{j}(t) \\
&=&C_{i}-\sum_{j\neq \check j(i):K_{i,j}=1}y_{j}(t)+\sum_{j\in \zeta
_{i}^{k}:j\neq \check j(i)}y_{j}(t) \\
&=&C_{i}-y_{\rho (k)}(t)-\sum_{v=1}^{k-1}K_{i,v}y_{\rho (v)}(t)\text{.}
\end{eqnarray*}%
However, from \eqref{eq:eq751} and Definition \ref{def:workAllocScheme}
\begin{align*}
C_{i}-y_{\rho (k)}(t)-\sum_{v=1}^{k-1}K_{i,v}y_{\rho (v)}(t) &\geq
 C_{i}-\left( \varrho _{\rho (k)}-2^{k-m -2}\delta \right) 
-\sum_{v=1}^{k-1}K_{i,v}\left( \varrho _{\rho (v)}+2^{v-m -2}\delta \right)
\\
&\geq C_{i}-\sum_{v=1}^{k}K_{i,v}\varrho _{\rho (v)}+2^{k-m -2}\delta
-2^{k-m -2}\delta +2^{-m -1}\delta \\
&\geq \sum_{j\in \zeta _{i}^{k}}\varrho _{j}+2^{-m -1}\delta
\end{align*}%
which gives%
\begin{equation*}
\sum_{j\in \zeta _{i}^{k}}y_{j}(t)\geq \sum_{j\in \zeta _{i}^{k}}\varrho
_{j}+2^{-m -1}\delta .
\end{equation*}%
Combining this with  \eqref{def:Rhat} gives%
\begin{align*}
\sum_{j\in \zeta _{i}^{k}}\left( \varrho _{j}^{r}-y_{j}(t)\right)
= \sum_{j\in \zeta _{i}^{k}}\varrho _{j}^{r}-\sum_{j\in \zeta
_{i}^{k}}y_{j}(t) 
\leq \sum_{j\in \zeta _{i}^{k}}\left( \varrho _{j}^{r}-\varrho _{j}\right)
-2^{-m -1}\delta 
\leq J 2^{-2m -6}\frac{\delta }{J }-2^{-m -1}\delta 
\leq -2^{-m -2}\delta \text{.}
\end{align*}%
This proves \eqref{eq:eq752} and completes the proof of part (a).\\

\noindent (b) Suppose now that $i\in \mathbb{N}_{I}$  and $t\ge 0$ are such that 
$\zeta _{i}^{0}\cap \sigma ^{r}(t)\neq \emptyset $ and $\sum_{j\in \zeta
_{i}^{0}}\mathcal{E}_{j}^r(t)=0$. 
From the latter property we have $x_{j}(t)=y_{j}(t)$ for all $j\in \zeta _{i}^{0}$,
and because $\zeta _{i}^{0}\cap \sigma ^{r}(t)\neq \emptyset $ there exists $%
l^{\ast }\in \zeta _{i}^{0}$ such that $l^{\ast }\in \sigma ^{r}(t)$. \
From Definition \ref{def:workAllocScheme} 
$y_{l^{\ast }}(t)=\varrho _{l^{\ast }}+\delta$
and 
\begin{align*}
\sum_{j\in \zeta _{i}^{0}}y_{j}(t) &= y_{l^{\ast }}(t)+\sum_{j\in \zeta
_{i}^{0}:j\neq l^{\ast }}y_{j}(t) 
\geq \varrho _{l^{\ast }}+\delta +\sum_{j\in \zeta _{i}^{0}:j\neq l^{\ast
}}\left( \varrho _{j}-2^{-m -3}\frac{\delta }{J }\right) \\
&\geq \sum_{j\in \zeta _{i}^{0}}\varrho _{j}+\delta -2^{-m -3}\delta 
\geq \sum_{j\in \zeta _{i}^{0}}\varrho _{j}+\frac{\delta}{2}\text{.}
\end{align*}%
This combined with  \eqref{def:Rhat} gives%
\begin{align*}
\sum_{j\in \zeta _{i}^{0}}\left( \varrho _{j}^{r}-x_{j}(t)\right)
&= \sum_{j\in \zeta _{i}^{0}}\varrho _{j}^{r}-\sum_{j\in \zeta
_{i}^{0}}y_{j}(t) 
\leq \sum_{j\in \zeta _{i}^{0}}\varrho _{j}^{r}-\sum_{j\in \zeta
_{i}^{0}}\varrho _{j}-\frac{\delta}{2} 
\leq \sum_{j\in \zeta _{i}^{0}}\left( \varrho _{j}^{r}-\varrho _{j}\right)
-\frac{\delta}{2} \\
&\leq J 2^{-2m -6}\frac{\delta }{J }-\frac{\delta}{2} 
\leq -2^{-2}\delta \text{.}
\end{align*}%
This completes the proof of (b).
\end{proof}

\section{\protect Large Deviation Estimates}
\label{sec:ldest}
Recall the allocation scheme $x(\cdot)$ given by Definition \ref{def:workAllocScheme} and define processes $Q^r, B^r, T^r$ associated with this allocation scheme with $\dot B^r(t)=x(t)$, $t\ge 0$, as in Section \ref{sec:backg}. Also recall the other associated processes as defined in \eqref{eq:eq935} --\eqref{eq:eq939}. Note that the allocation scheme depends on a parameter $\alpha \in (0, 1/2)$ and $c_1, c_2\in (0,\infty)$.
Let 
$
X^{r}(t)=\left( Q^{r}(t),\mathcal{E}^{r}(t)\right)
$
and let 
\begin{equation}
	\label{eq:eqhatxrt}
\hat{X}^{r}(t) =\left( \hat{Q}^{r}(t),\mathcal{E}^{r}(r^{2}t)\right) 
=\left( Q^{r}(r^{2}t)/r,\mathcal{E}^{r}(r^{2}t)\right),\; t\ge 0.
\end{equation}%
Note that although $\hat Q^r$ is not Markovian, the pair $\hat{X}^{r}$ defines a strong Markov process with state space
$\cls^r \doteq (\RR_+ \cap \frac{1}{r}\NN_0)^J \times \{0,1\}^J$. Expectations of various functionals of the Markov process $\hat{X}^{r}$
when $\hat{X}^{r}(0)=x$ will be denoted as $E_x$ and the associated probabilities by $P_x$. 
The  following theorem is a key step in estimating the idleness terms in state dynamics.
\begin{theorem}
\label{thm:IdleTimeExp}For any $\epsilon \in(0,\infty)$ and $j\in \mathbb{N}_{J}$ there
exist  $\hat{B}_{1},\hat{B}_{2},\hat{B}_{3},\hat{B}_{4},R \in (0,\infty) $
such that for all $r\geq R$,  $t\geq 1$ and $x\in \cls^r$ we have
\begin{equation} 
P_x\left( \int_{0}^{tr^{1/2}}\mathit{I}_{\{\mathcal{E}_{j}^r(s)=1\}}ds\geq
\epsilon r^{1/4+\alpha /2}t\right) \leq \hat{B}_{1}e^{-r^{1/4+\alpha /2}t%
\hat{B}_{2}}+ \left( 1+\frac{\hat{B}_{3}}{r^{1/4+\alpha /2}}\right)
^{-\hat{B}_{4}r^{1/2}t}
\label{eq:rootIdleTimeResult}
\end{equation}%
and 
\begin{equation}
P_x\left( \int_{0}^{tr^{2}}\mathit{I}_{\{\mathcal{E}_{j}^r(s)=1\}}ds\geq
\epsilon rt\right) \leq \hat{B}_{1}e^{-rt\hat{B}_{2}}+ \left( 1+\frac{%
\hat{B}_{3}}{r^{1+\alpha }}\right) ^{-\hat{B}_{4}r^{2}t}\text{.}
\label{eq:squareIdleTimeResult}
\end{equation}%
\end{theorem}
\begin{proof}
	Let $j\in \mathbb{N}_{J}$, $x \in \cls^r$  and $\epsilon >0$ be arbitrary. Recall $c_1, c_2$ from Section \ref{sec:secallostra}. \ Define 
	\begin{equation*}
	\tau _{0}^{r,j}\doteq\inf \left\{ s\geq 0:Q_{j}^{r}(s)\geq c_{2}r^{\alpha
	}\right\} ,
	\end{equation*}%
	\begin{equation*}
	\tau _{2l-1}^{r,j}\doteq\inf \left\{ s\geq \tau
	_{2l-2}^{r,j}:Q_{j}^{r}(s)<r^{\alpha }\frac{c_{2}+c_{1}}{2}\right\} ,
	\end{equation*}%
	and%
	\begin{equation*}
	\tau _{2l}^{r,j}\doteq \inf \left\{ s\geq \tau _{2l-1}^{r}:Q_{j}^{r}(s)\geq
	c_{2}r^{\alpha }\right\} 
	\end{equation*}%
	for all $l\geq 1$. Recall the functions $\cle_j$ introduced in Definition \ref{def:workAllocScheme}. \ Define the indicator functions 
	\[
	\theta _{l}^{r,j} \doteq  \left\{ 
	\begin{array}{cc}
		1, & \text{ if }\mathcal{E}_{j}^r(s)=1\text{ for some }s\in
	\left( \tau _{2l-1}^{r,j},\tau _{2l}^{r,j}\right]  \\
	\  &  \\ 
	0, & \text{ otherwise.}
	\end{array}
	\right.
	\]%
	For $t>0$ let%
	\begin{equation}
	\eta _{t}^{r,j}=\max \left\{ l:\tau _{2l-1}^{r,j}\leq tr^{1/2}\right\},\; \hat{\eta}_{t}^{r,j}=\max \left\{ l:\tau _{2l-1}^{r,j}\leq tr^{2}\right\}   \label{eq:etaDef}
	\end{equation}
	and %
	$N_{k}^{r,j}=\sum_{l=1}^{k}\theta _{l}^{r,j}$.
	Consider the events,
	\begin{equation*}
	\mathcal{B}_{1}^{r,j}=\left\{ \eta _{t}^{r,j}\leq 2\lambda
	_{j}^{r}r^{1/2}t\right\} \text{,}\;\; \mathcal{\hat{B}}_{1}^{r,j}=\left\{ \hat{\eta} _{t}^{r,j}\leq 2\lambda
	_{j}^{r}r^{2}t\right\} \text{,}
	\end{equation*}%

	\begin{equation*}
	\mathcal{B}_{2}^{r,j}=\left\{ N_{\left\lceil 2\lambda
	_{j}^{r}tr^{1/2}\right\rceil }^{r,j}\leq \frac{\lambda _{j}^{r}\epsilon }{%
	2(c_{2}-c_{1})}r^{1/4-\alpha /2}t\right\}, \; \mathcal{\hat{B}}_{2}^{r,j}=\left\{ N_{\left\lceil 2\lambda
	_{j}^{r}tr^{2}\right\rceil }^{r,j}\leq \frac{\lambda _{j}^{r}\epsilon }{%
	2(c_{2}-c_{1})}r^{1-\alpha }t\right\} \text{.}
	\end{equation*}%
	Let 
	$$\clc^r \doteq \left\{\int_{0}^{r^{1/2}t}\mathit{I}_{\left\{ \mathcal{E}_{j}^r(s)=1\right\}
	}ds\geq \epsilon r^{1/4+\alpha /2}t\right\},\; \hat{\clc}^r \doteq \left\{\int_{0}^{r^{2}t}\mathit{I}_{\left\{ \mathcal{E}_{j}^r(s)=1\right\}
	}ds\geq \epsilon rt\right\}.$$
	Then
	\begin{equation}
	P\left( \clc^r\right)  \leq P\left( (\mathcal{B}_{1}^{r,j})^c\right) +P\left( (\mathcal{B}_{2}^{r,j})^c\right) + P\left( \mathcal{B}_{1}^{r,j}\cap \mathcal{B}_{2}^{r,j}\cap \clc^r\right)
	\label{eq:mainIneqRoot} 
	\end{equation}%
	and 
	\begin{equation}
	P\left( \hat{\clc}^r\right)  \leq P\left( (\hat{\mathcal{B}}_{1}^{r,j})^c\right) +P\left( (\hat{\mathcal{B}}_{2}^{r,j})^c\right) + P\left( \hat{\mathcal{B}}_{1}^{r,j}\cap \hat{\mathcal{B}}_{2}^{r,j}\cap \hat{\clc}^r\right).
	\label{eq:mainIneqSquare}
	\end{equation}%
	Noting that each
	occurrence of $\tau _{2l-1}^{r,j}$ requires an arrival of a job of type $j$, we have
	$$
	P\left( ({\mathcal{B}}_{1}^{r,j})^c\right) = P\left(\eta _{t}^{r,j} > 2\lambda
	_{j}^{r}r^{1/2}t\right) \le P\left(A_j^r(tr^{1/2}) \ge 2\lambda
	_{j}^{r}r^{1/2}t\right).$$
	Similarly,
	$$
	P\left( (\hat{\mathcal{B}}_{1}^{r,j})^c\right)  \le P\left(A_j^r(tr^{2}) \ge 2\lambda
	_{j}^{r}r^{2}t\right).$$
	Thus from the first inequality in Theorem \ref{thm:LDP} in Appendix we can find $R_1 \in (0,\infty)$ and $\kappa_1, \kappa_2 \in (0,\infty)$
	such that for all $r \ge R_1$, $t\ge 1$ and $j \in \mathbb{N}_J$
	\begin{equation}
		P\left( ({\mathcal{B}}_{1}^{r,j})^c\right) \le \kappa_1 e^{-tr^{1/2}\kappa_2}, \; P\left( (\hat{\mathcal{B}}_{1}^{r,j})^c\right) \le \kappa_1 e^{-tr^{2}\kappa_2}. \label{eq:firstTermRoot}
	\end{equation}
	We now estimate $P\left( ({\mathcal{B}}_{2}^{r,j})^c\right)$, $P\left( (\hat{\mathcal{B}}_{2}^{r,j})^c\right)$.
	Note that the $\left\{ \theta _{l}^{r,j}\right\} _{l=1}^{\infty }$ are
	i.i.d. Bernoulli with parameter $p(r)$ where 
	\begin{equation*}
	p(r) = P(\theta _{l}^{r,j}=1)=P\left( Q_{j}^{r}(\varsigma _{l}^{r,j})<
	c_{1}r^{\alpha }\right) 
	\end{equation*}%
	and
	\begin{equation}
	\varsigma _{l}^{r,j}\doteq \inf \left\{ s\geq \tau _{2l-1}^{r,j}:Q_{j}^{r}(t)<
	c_{1}r^{\alpha }\text{ or }Q_{j}^{r}(t)\geq c_{2}r^{\alpha }\right\} \text{.}
	\label{eq:bndryHitTime}
	\end{equation}%
	The probability $p(r)$ can be estimated as follows. Note that from
	Lemma \ref{lem:posDriftCond}, for $\tau _{2l-1}^{r,j}\leq s<\varsigma
	_{l}^{r,j}$ 
	\begin{equation}\label{eq:eq1023}
	\lambda _{j}^{r}-\mu _{j}^{r}x_{j}(s)\geq \mu _{j}\kappa 
	\end{equation}%
	where 
	$
	\kappa \doteq 2^{-2m -5}\frac{\delta }{%
	 J }$.
	Letting $\bar C = \max_i\{C_i\}$ and $d_j \doteq (c_2-c_1)/(\mu_j\kappa)$, define
	\begin{align*}
	\mathcal{A}_{l}^{r,j} &=\Bigg\{ \sup_{0\leq s\leq 
	d_jr^{\alpha } }\left\vert A_{j}^{r}(\tau
	_{2l-1}^{r,j}+s)-A_{j}^{r}(\tau _{2l-1}^{r,j})-\lambda _{j}^{r}s\right\vert
	 \\
	&+\sup_{0\leq s\leq \bar C d_jr^{\alpha} }\left\vert S_{j}^{r}(B_{j}^{r}(\tau
	_{2l-1}^{r,j})+s)-S_{j}^{r}(B_{j}^{r}(\tau _{2l-1}^{r,j}))-\mu
	_{j}^{r}s\right\vert   \geq \frac{\left( c_{2}-c_{1}\right) r^{\alpha }}{4}\Bigg\}. 
	\end{align*}%
	From Theorem \ref{thm:LDP} and strong Markov property there exist $\kappa_3, \kappa_4 \in (0, \infty)$ and $R_{2}\in \left[ R_{1},\infty \right) $ such that for all $%
	r\geq R_{2}$, $j \in \mathbb{N}_J$, and $l\geq 1$%
	\begin{equation*}
	P\left( \mathcal{A}_{l}^{r,j}\right) \leq \kappa_3 e^{-r^{\alpha }\kappa_4}.
	\end{equation*}%
	We can also assume without loss of generality that for $r \ge R_2$,
	$
	r^{\alpha }\frac{c_{2}-c_{1}}{4}>2\text{.}
	$
	From \eqref{eq:eq1023},
	on the event $\left( \mathcal{A}_{l}^{r,j}\right) ^{c}$, we have for $s\in %
	\left[ \tau _{2l-1}^{r,j},\varsigma _{l}^{r,j}\wedge \left( \tau
	_{2l-1}^{r,j}+d_j r^{\alpha } \right)
	\right) $%
	\begin{eqnarray*}
	Q_{j}^{r}(s) &\geq &r^{\alpha }\frac{c_{2}+c_{1}}{2}-1+\left(
	A_{j}^{r}(s)-A_{j}^{r}(\tau _{2l-1}^{r,j})\right) -\left(
	S_{j}^{r}(B_{j}^{r}(s))-S_{j}^{r}(B_{j}^{r}(\tau _{2l-1}^{r,j}))\right)  \\
	&\geq &r^{\alpha }\frac{c_{2}+c_{1}}{2}-1-r^{\alpha }\frac{c_{2}-c_{1}}{4}%
	+(s- \tau _{2l-1}^{r,j})\mu _{j}\Delta.
	\end{eqnarray*}%
	Since the expression on the right side with  $s=\tau_{2l-1}^{r,j}+d_j r^{\alpha }$ is larger than $c_2 r^{\alpha}$
	we have  that on $\left( \mathcal{A}_{l}^{r,j}\right) ^{c}$, $\varsigma _{l}^{r,j}<\tau
	_{2l-1}^{r,j}+d_j r^{\alpha }$ and so
	$Q_{j}^{r}(\varsigma _{l}^{r,j}) > c_1r^{\alpha}$. Thus
	 $\left( \mathcal{A}_{l}^{r,j}\right) ^{c}\cap \{\theta_{l}^{r,j}=1\}=\emptyset$ and 
	\begin{equation*}
	p(r) \le  P\left( \mathcal{A}_{l}^{r,j}\right)
	\leq \kappa_3e^{-r^{\alpha }\kappa_4}\text{.}
	\end{equation*}%
	Choose $R_{3}\in \left[
	R_{2},\infty \right) $ such that for all $r\geq R_{3}$ we have%
	\begin{equation}
	\epsilon /[10(c_{2}-c_{1})r^{1+\alpha }]\geq 2p(r)\text{,}\;\; \epsilon /[5(c_{2}-c_{1})r^{1/4+\alpha /2}]\leq 1/2\text{,}\;\;
	\left( 2\lambda _{j}^{r}r^{1/2}+1\right) /5\leq \lambda _{j}^{r}r^{1/2}/2.
	\label{eq:greaterThanPr}
	\end{equation}%

	so in particular from the third inequality, for all $t\geq 1$, 
	\begin{equation}
	\left\lceil 2\lambda _{j}^{r}tr^{1/2}\right\rceil \left( \epsilon
	/[5(c_{2}-c_{1})r^{1/4+\alpha /2}]\right) \leq \lambda _{j}^{r}tr^{1/4-\alpha
	/2}\epsilon /[2(c_{2}-c_{1}])  \label{eq:rootIdleTimeProofIneq}
	\end{equation}%
	and 
	\begin{equation}
	\left\lceil 2\lambda _{j}^{r}tr^{2}\right\rceil \left( \epsilon
	/[5(c_{2}-c_{1})r^{1+\alpha }]\right) \leq \lambda _{j}^{r}tr^{1-\alpha }\epsilon/[2(c_{2}-c_{1})].  \label{eq:squareIdleTimeProofIneq}
	\end{equation}%
	Note that if $Z\sim \mbox{Bin}(L,p)$ then, for all $u>0$
	$$P(Z \ge u) \le (1+ p(e-1))^L e^{-u}.$$
	Thus we have
	\begin{eqnarray*}
	P\left( N_{\left\lceil 2\lambda _{j}^{r}tr^{1/2}\right\rceil }^{r}\geq \frac{%
	\lambda _{j}^{r}\epsilon }{2(c_{2}-c_{1})}r^{1/4-\alpha /2}t\right)  &\leq
	&e^{-\frac{\lambda _{j}^{r}\epsilon }{2(c_{2}-c_{1})}r^{1/4-\alpha
	/2}t}\left( 1+p(r)(e^{1}-1)\right) ^{\left\lceil 2\lambda
	_{j}^{r}tr^{1/2}\right\rceil } \\
	&\leq &\left( \frac{1+2p(r)}{e^{\epsilon
	/5(c_{2}-c_{1})r^{1/4+\alpha /2}}}\right) ^{\left\lceil 2\lambda
	_{j}^{r}tr^{1/2}\right\rceil }\text{.}
	\end{eqnarray*}%
	where the second line uses (\ref{eq:rootIdleTimeProofIneq}) and the fact that if for positive $a,b,c,d$, $ab\le c$, then
	$$e^{-c}(1+ d(e-1))^b \le \left(\frac{1+2d}{e^a}\right)^b.$$
	For all $%
	r\geq R_{3}$ we have 
	\begin{align}
	\left( \frac{1+2p(r)}{e^{\epsilon /[5(c_{2}-c_{1})r^{1/4+\alpha /2}]}}%
	\right) ^{\left\lceil 2\lambda _{j}^{r}tr^{1/2}\right\rceil } 
	&\leq \left( \frac{1+\epsilon /[10(c_{2}-c_{1})r^{1/4+\alpha /2}]}{1+\epsilon
	/[5(c_{2}-c_{1})r^{1/4+\alpha /2}]}%
	\right) ^{\left\lceil 2\lambda _{j}^{r}tr^{1/2}\right\rceil }  \notag \\
	&\leq \left( \frac{1}{1+4\epsilon /[50(c_{2}-c_{1})r^{1/4+\alpha /2}]}\right)
	^{\left\lceil 2\lambda _{j}^{r}tr^{1/2}\right\rceil }  \notag \\
	&\leq  \left( 1+\frac{4\epsilon /[50(c_{2}-c_{1})]}{r^{1/4+\alpha /2}}%
	\right)^{-\lambda _{j}r^{1/2}t}  \label{eq:secondTermRoot}
	\end{align}
	where the first line uses the inequality $e^x \ge 1+x$ and the   first bound in (\ref{eq:greaterThanPr}), the second uses 
	the second bound in (\ref{eq:greaterThanPr}) along with the inequality $(1+x)/(1+2x) \le 5/(5+4x)$ for $x \in [0,1/4]$, and the third uses  \eqref{def:Rhat} to
	bound $\lambda _{j}^{r}$ by $\lambda _{j}$. 
	Thus we have shown
	\begin{equation}
		\label{eq:eq1121}
		P\left( ({\mathcal{B}}_{2}^{r,j})^c\right) \le \left( 1+\frac{\hat{B}_{3}}{r^{1/4+\alpha /2}}\right)
		^{-\hat{B}_{4}r^{1/2}t}
	\end{equation}
	where $\hat B_3 = 4\epsilon /[50(c_{2}-c_{1})]$ and $\hat B_4=1$.
	A similar calculation shows
	that
	\begin{equation}
		\label{eq:eq1125}
		P\left( (\hat{\mathcal{B}}_{2}^{r,j})^c\right) \le \left( 1+\frac{\hat{B}_{3}}{r^{1+\alpha}}\right)
		^{-\hat{B}_{4}r^{2}t}
	\end{equation}
	Finally we estimate the third probability on the right sides of \eqref{eq:mainIneqRoot} and \eqref{eq:mainIneqSquare}. 
	Note that 
	\begin{equation*}
	\int_{0}^{tr^{1/2}}\mathit{I}_{\left\{ \mathcal{E}_{j}^r(s)=1\right\} }ds\leq
	\int_{0}^{\tau _{0}^{r,j}}\mathit{I}_{\left\{ \mathcal{E}_{j}^r(s)=1\right\}
	}ds+\sum_{l=1}^{\eta _{t}^{r,j}}\int_{\tau _{2l-1}^{r,j}}^{\tau _{2l}^{r,j}}%
	\mathit{I}_{\left\{ \mathcal{E}_{j}^r(s)=1\right\} }ds
	\end{equation*}%
	From (\ref{eq:bndryHitTime}) we see that%
	\begin{equation}\label{eq:eq1136}
	\int_{\tau _{2l-1}^{r,j}}^{\tau _{2l}^{r,j}}\mathit{I}_{\left\{ \mathcal{E}%
	_{j}(s)=1\right\} }ds=\tau _{2l}^{r,j}-\varsigma _{l}^{r,j}.
	\end{equation}%
	Indeed, if $\theta _{l}^{r,j}=0$ then $\varsigma _{l}^{r,j}=\tau _{2l}^{r,j}$
	and the integral on the left side is $0$. Also, if 
	 $\theta _{l}^{r,j}=1$ then $Q_{j}^{r}(\varsigma
	_{l}^{r,j})=\left\lceil c_{1}r^{\alpha }\right\rceil -1$,  $\varsigma _{l}^{r,j}<\tau _{2l}^{r,j}$ and $\cle_j(s) =1$ for all $s \in [\varsigma _{l}^{r,j}, \tau _{2l}^{r,j}]$, giving once more the 
	identity in \eqref{eq:eq1136}. In the latter case we also have the representation
	\begin{equation}
	\tau _{2l}^{r,j}-\varsigma _{l}^{r,j}
	=\inf \left\{ s\geq 0:A_{j}^{r}(\varsigma
	_{l}^{r,j}+s)-A_{j}^{r}(\varsigma _{l}^{r,j})\geq \left\lceil c_{2}r^{\alpha
	}\right\rceil -\left\lceil c_{1}r^{\alpha }\right\rceil +1\right\} \text{.}
	\label{eq:idleTimeDefGen}
	\end{equation}%
	Similarly if we define 
	\begin{equation*}
	\varsigma _{0}^{r,j}=\inf \left\{ s\geq 0:Q_{j}^{r}(t)\leq c_{1}r^{\alpha }%
	\text{ or }Q_{j}^{r}(t)\geq c_{2}r^{\alpha }\right\} 
	\end{equation*}%
	then 
	\begin{equation*}
	\int_{0}^{\tau _{0}^{r,j}}\mathit{I}_{\left\{ \mathcal{E}_{j}^r(s)=1\right\}
	}ds=\tau _{0}^{r,j}-\varsigma _{0}^{r,j}
	\end{equation*}%
	where if $\varsigma _{0}^{r,j}<\tau _{0}^{r,j}$ we
	have 
	\begin{equation}
	\tau _{0}^{r,j}-\varsigma _{0}^{r,j}=\inf \left\{ s\geq 0:A_{j}^{r}(\varsigma
	_{0}^{r,j}+s)-A_{j}^{r}(\varsigma _{0}^{r,j})\geq \left\lceil c_{2}r^{\alpha
	}\right\rceil -\left\lceil c_{1}r^{\alpha }\right\rceil +1\right\} \text{.}
	\label{eq:idleTimeDefInit}
	\end{equation}%
	Consequently , since  on $\mathcal{B}_{1}^{r,j}$, $\eta _{t}^{r,j} \le 2\lambda
	_{j}^{r}r^{1/2}t$, by taking $r$ suitably large
	\begin{align*}
	& P\left( \mathcal{B}_{1}^{r,j}\cap \mathcal{B}_{2}^{r,j}\cap \clc^r \right)  \\
	& \quad \leq P\left( \mathcal{B}_{2}^{r,j}\cap \left\{ \tau
	_{0}^{r,j}-\varsigma _{0}^{r,j}+\sum_{l=1}^{\left\lceil 2\lambda
	_{j}^{r}tr^{1/2}\right\rceil }\left( \tau _{2l}^{r,j}-\varsigma
	_{l}^{r,j}\right) \geq \epsilon r^{1/4+\alpha /2}t\right\} \right)  \\
	& \quad \leq P\left( \inf \left\{ s\geq 0:\check{A}_{j}^{r}(s)\geq \left( 
	\frac{3\lambda _{j}^{r}\epsilon }{4}r^{1/4+\alpha /2}t\right) \right\} \geq
	\epsilon r^{1/4+\alpha /2}t\right) \\
	& \quad \leq P \left (\check{A}_{j}^{r}(\epsilon r^{1/4+\alpha /2}t) \le \frac{3\lambda _{j}^{r}\epsilon }{4}r^{1/4+\alpha /2}t\right)
	\end{align*}%
	where $\check{A}_{j}^{r}$ is a Poisson process with rate $\lambda _{j}^{r}$
	and the second inequality comes from the representations in (\ref{eq:idleTimeDefInit}) and (\ref%
	{eq:idleTimeDefGen}). 
	From Theorem \ref{thm:LDP}  there exist 
	$\kappa_5, \kappa_6 \in (0,\infty)$
	 and $R_{4}\in \left[ R_{3},\infty \right) $ such
	that for all $r\geq R_{4}$ 
	\begin{align}
	 P\left( \mathcal{B}_{1}^{r,j}\cap \mathcal{B}_{2}^{r,j}\cap \clc^r \right)   
	\leq P\left( \sup_{0\leq s\leq \epsilon r^{1/4+\alpha
	/2}t}\left\vert \check{A}_{j}^{r}(s)-\lambda ^{r}s\right\vert >\frac{%
	\epsilon }{2}r^{1/4+\alpha /2}t\right)   
	\leq \kappa_5e^{-r^{1/4+\alpha /2}t\kappa_6}.
	\label{eq:thirdTermRoot} 
	\end{align}%
	A similar calculation shows that
	\begin{align}
		 P\left( \mathcal{\hat{B}}_{1}^{r,j}\cap \mathcal{\hat{B}}_{2}^{r,j}\cap
	\hat{\clc}^r \right)  \leq \kappa_5e^{-rt\kappa_6}\text{.}  
	\label{eq:thirdTermSquare} 
	\end{align}%

	Finally (\ref{eq:firstTermRoot}), (\ref{eq:eq1121}), (\ref%
	{eq:thirdTermRoot}), and (\ref{eq:mainIneqRoot}) prove (\ref%
	{eq:rootIdleTimeResult}) while (\ref{eq:firstTermRoot}), \eqref{eq:eq1125}, \eqref{eq:thirdTermSquare}
	and \eqref{eq:mainIneqSquare} prove  (\ref{eq:squareIdleTimeResult}).
	This completes the proof.
\end{proof}

Let $c_3 \doteq \frac{2Jc_{2}}{\min_{j}\mu_{j}}$ and recall that $\bar{C} \doteq \max_{i \in \mathbb{N}_I}\{C_i\}$.
Note that if for given $s \ge 0$, $W^{r}_{i}(s)>c_{3}r^{\alpha}$ for some $i \in \mathbb{N}_{I}$ then we must have that
$Q^r_j(s)\ge c_2 r^{\alpha}$ for some $j \in \mathbb{N}_{J}$ with $K_{ij}=1$, namely $i \in \hat \omega^r(s)$. From Lemma \ref{lem:fullWorkloadCond} it then follows that
for such a $s$ if $C_i > \sum_{j=1}^J K_{ij} x_j(s)$, the $\cle_j^r(t)\neq 0$ for some $j$ with $K_{ij}=1$. From this it follows that for any $t\ge 0$
$$\int_0^t \mathit{I}_{\{W^{r}_{i}(s)>c_{3}r^{\alpha}\}}(s)dI^{r}_{i}(s) \le \bar{C} \sum_{j: K_{ij}=1}\int_0^t \mathit{I}_{ \{\cle_j^r(s)=1\}} ds.$$
This along with Theorem \ref{thm:IdleTimeExp} implies that for any $\epsilon >0$ and $i\in \mathbb{N}_{I}$ there
exist  $\hat{B}_{1},\hat{B}_{2},\hat{B}_{3},\hat{B}_{4},R \in (0,\infty)$
such that for all $r\geq R$,  $t\geq 1$ and $x\in \cls^r$ we have
\begin{equation} 
P_x\left( \int_{0}^{tr^{1/2}}\mathit{I}_{\{W^{r}_{i}(s)\ge c_{3}r^{\alpha}\}}(s)dI^{r}_{i}(s)\geq
\epsilon r^{1/4+\alpha /2}t\right) \leq \hat{B}_{1}e^{-r^{1/4+\alpha /2}t%
\hat{B}_{2}}+\left( 1+\frac{\hat{B}_{3}}{r^{1/4+\alpha /2}}\right)
^{-\hat{B}_{4} r^{1/2} t}\label{eq:eqpxtr2}
\end{equation}%
and 
\begin{equation}
P_x\left( \int_{0}^{tr^{2}}\mathit{I}_{\{W^{r}_{i}(s)\ge c_{3}r^{\alpha}\}}(s)dI^{r}_{i}(s)\geq
\epsilon rt\right) \leq \hat{B}_{1}e^{-rt\hat{B}_{2}}+\left(  1+\frac{%
\hat{B}_{3}}{r^{1+\alpha }}\right) ^{-\hat{B}_{4}r^{2}t}\text{.}
\label{eq:squareIdleTimeResultImplication}
\end{equation}%

\subsection{Estimating holding cost through workload cost.}
Recall the matrix
$M$ introduced in Section \ref{sec:hgi}. Along with the process $\hat W^r = KM^r \hat Q^r$, it will be convenient to also consider
the process $\tilde W^r \doteq KM \hat Q^r$.
The following is the main result of the section which says that under the scheme introduced in Definition \ref{def:workAllocScheme}, the queue lengths for the associated workload are 
`asymptotically optimal' in a certain sense. This result will be key in showing that under our policy, property (II) of 
HGI holds asymptotically.
\begin{theorem}
\label{thm:discCostInefBnd}
There exist  $B,R \in (0,\infty) $ such that for
all $r\geq R$, $x=(q,z) \in \cls^r$, $\theta >0$ and $T\ge 1$, we have%
\begin{equation*}
\left\vert E_x\left[ \int_{0}^{\infty }e^{-\theta t}  h\cdot \hat{Q}%
^{r}(t) dt\right] -E_x\left[ \int_{0}^{\infty }e^{-\theta
t}\clc\left( \tilde{W}^{r}(t)\right) dt\right] \right\vert \leq Br^{\alpha
-1/2} \frac{1+|q|^2}{1-e^{-\theta}}
\end{equation*}
and
\begin{equation*}
\left\vert E_x\left[ \frac{1}{T}\int_{0}^{T} h \cdot \hat{Q}%
^{r}(t) dt\right] -E_x\left[ \frac{1}{T}\int_{0}^{T}
\clc\left( \tilde{W}^{r}(t)\right) dt\right] \right\vert \leq Br^{\alpha
-1/2} (1+|q|^2)\text{.}
\end{equation*}
\end{theorem}

In order to prove the result we begin with the following two propositions.

Recall the sets
$\zeta^0_i$, $\zeta_i^k$ from \eqref{eq:eqzetaiz} and \eqref{eq:eqzetaik} and that $c_3 = \frac{2Jc_{2}}{\min_{j}\mu_{j}}$.
For  $\xi\geq 0$, $i\in \mathbb{N}_{I}$ and $k \in \NN_m$ let 
\begin{equation}
\hat{\tau}_{i}^{1}(\xi)\doteq \inf \left\{ t\geq \xi:\sum_{j\in \zeta
_{i}^{0}} \frac{Q_{j}^{r}(s)}{\mu _{j}^{r}} <2c_{3}r^{\alpha }\right\},\;  \hat{\tau}_{k}^{s}(\xi)\doteq \inf \left\{ t\geq \xi:\min_{i\in N_{\rho (k)}}\left\{
\sum_{j\in \zeta _{i}^{k}}\frac{Q_{j}^{r}(s)}{\mu _{j}^{r}}\right\}
<2c_{3} r^{\alpha }\right\}. \label{eq:defTauS}
\end{equation}%

\begin{proposition}
\label{thm:initInef} There exist  $R,B\in (0,\infty)$ 
 such that for all $r\geq
R$, $i\in \mathbb{N}_{I}$, $x=(q,z) \in \cls^r$, and $k\in \NN_m$ we have 
\begin{equation*}
\frac{1}{r^{3}}E_x  \int_{0}^{\hat{\tau}_{i }^{1}(0)}\sum_{j\in \zeta
_{i}^{0}} \frac{Q_{j}^{r}(s)}{\mu _{j}}ds \leq B (1+|q|^2) r^{-1}
\end{equation*}%
and%
\begin{equation*}
\frac{1}{r^{3}}E_x \int_{0}^{\hat{\tau}_{k }^{s}(0)}\min_{i\in N_{\rho
(k)}}\left\{ \sum_{j\in \zeta _{i}^{k}} \frac{Q_{j}^{r}(s)}{\mu _{j}}\right\} ds%
 \leq B(1+|q|^2)r^{-1}.
\end{equation*}
\end{proposition}

\begin{proof}
 Let $k\in \NN_m$ be arbitrary.
 Note that under $P_x$, $Q^{r}(0)=r\hat{Q}^{r}(0)=rq$%
. \ Choose $\check{i}(0)\in N_{\rho (k)}$ such that 
\begin{equation*}
\sum_{j\in \zeta _{\check{i}(0)}^{k}}r \frac{q_{j}}{\mu _{j}}=\min_{i\in N_{\rho
(k)}}\left\{ \sum_{j\in \zeta _{i}^{k}}r\frac{q_{j}}{\mu _{j}}\right\} 
\end{equation*}%
and define 
\begin{equation}
d=\sum_{j\in \zeta _{\check{i}(0)}^{k}}\frac{q_{j}}{\mu _{j}} \mbox{ and } \Delta =2^{-m -2}\delta \text{,} \label{eq:defTriangle}
\end{equation}%
where $\delta$ is as in Definition \ref{def:workAllocScheme}.
If $rd<2c_{3}r^{\alpha }$ then $\hat{\tau}_{k }^{s}(0)=0$ and the result
holds trivially. \ Consider now $rd\ge 2c_{3}r^{\alpha }$ so that $\hat{\tau}_{k }^{s}(0)>0$. We claim that for $t\in \left[ 0,\hat{\tau}_{k }^{s}(0)\right) $ and $i' \in N_{\rho(k)}$ 
we have 
$\zeta _{i'}^{k} \cap \sigma^r(t)\neq \emptyset$.
To see the claim note that for such $t$, for all $i' \in N_{\rho(k)}$, from the definition of $\hat{\tau}_{k }^{s}(0)$
$$\sum_{j \in \zeta _{i'}^{k}} \frac{Q_j^r(t)}{\mu_j^r} \ge \min_{i \in N_{\rho(k)}} \sum_{j \in \zeta _{i}^{k}} \frac{Q_j^r(t)}{\mu_j^r} \ge 2c_3 r^{\alpha}.$$
Thus, from the definition of $c_3$ there is a $j \in \zeta_{i'}^k$ such that
$$Q_j^r(t) \ge \frac{2c_3}{J} r^{\alpha}\mu_j^r \ge \frac{c_3}{J} r^{\alpha}\mu_j \ge c_2r^{\alpha},$$
namely $j\in \sigma^r(t)$. Thus we have $\zeta _{i'}^{k} \cap \sigma^r(t)\neq \emptyset$ proving the claim.  From Lemma \ref{lem:lem3_4}(a) we now have that
for 
$i\in N_{\rho (k)}$ and $t\in \left[ 0,\hat{\tau}_{k }^{s}(0)\right) $ such that\ $\sum_{j\in \zeta _{i}^{k}}\mathcal{E}_{j}^r(t)=0$
\begin{equation}
\sum_{j\in \zeta _{i}^{k}}\left( \varrho _{j}^{r}-x_{j}(t)\right) \leq
-2^{-m -2}\delta =-\Delta \text{.}\label{eq:eq129}
\end{equation}%
Recall that $\bar{C} = \max_{i} \{C_i\}$.
Define for $y\ge 0$, the events
\begin{equation*}
\mathcal{A}_{y}^{r}=\left\{ \sum_{i\in \mathbb{N}_{I}}\int_{0}^{(2ry/\Delta ) \wedge \hat{\tau}_{k }^{s}(0) }\one_{\left\{ \sum_{j\in \zeta _{i}^{k}}\mathcal{E}%
_{j}^r(s)>0\right\} }ds\geq \frac{yr}{4(\bar C\vee \Delta)}\right\} 
\end{equation*}%
and 
\begin{eqnarray*}
\mathcal{B}_{y}^{r} &=& \bigcup_{j \in \NN_J}\left\{ \sup_{0\leq t\leq 2ry/\Delta }\left\vert
A_{j}^{r}(t)-t\lambda _{j}^{r}\right\vert 
 +\sup_{0\leq t\leq 2\bar{C}ry/\Delta }\left\vert
S_{j}^{r}(t)-t\mu _{j}^{r}\right\vert \geq \frac{y\bar{\mu}_{\min}r}{%
4J}\right\} .
\end{eqnarray*}%
From Theorem \ref{thm:IdleTimeExp}  (cf. \eqref{eq:rootIdleTimeResult} with $\frac{2r^{1/2}y}{\Delta}$ substituted in for $t$) and Theorem \ref{thm:LDP} there
exist  $B_{1},B_{2} \in (0,\infty) $ and $R_{1}\in \left[ \hat{R},\infty
\right) $ (recall (\eqref{def:Rhat})) such that for all $r\geq R_{1}$ and $y\geq \max \{\frac{\Delta}{2},d,1\}$, 
\begin{equation*}
P\left( \mathcal{A}_{y}^{r}\bigcup \mathcal{B}_{y}^{r}\right) \leq
B_{1}e^{-B_{2}y}\text{.}
\end{equation*}%
Also on the event $\left( \mathcal{A}_{y}^{r}\bigcup \mathcal{B}%
_{y}^{r}\right) ^{c}$ for all $t\in \left[ 0,\hat{\tau}_{k }^{s}(0)\wedge
2ry/\Delta \right) $ we have 
\begin{eqnarray*}
\min_{i\in N_{\rho (k)}}\left\{ \sum_{j\in \zeta _{i}^{k}}
\frac{Q_{j}^{r}(t)}{\mu_{j}^{r}}\right\}  &\leq &\sum_{j\in \zeta _{\check{i}(0)}^{k}}\frac{Q_{j}^{r}(t)}{\mu _{j}^{r}} \\
&\leq &rd+\sum_{j\in \zeta _{\check{i}(0)}^{k}} \frac{A_{j}^{r}(t)}{\mu _{j}^{r}}-\sum_{j\in \zeta _{\check{i}(0)}^{k}}
\frac{S_{j}^{r}(B_{j}(t))}{\mu _{j}^{r}} \\
&\leq &rd+\sum_{j\in \zeta _{\check{i}(0)}^{k}}\frac{y\bar{\mu}_{\min}r}{%
4 J \mu _{j}^{r}}+\sum_{j\in \zeta _{\check{i}(0)}^{k}}\left( t\varrho _{j}^{r}-B_{j}(t)\right)  
\end{eqnarray*}
where the last line follows from the definition of the event $\mathcal{B}_{y}^{r}$.
Next note that
$$B_j(t) = \int_0^t x_j(s) ds = \int_0^t x_j(s) \one_{\left\{ \sum_{j\in \zeta _{\check{i}(0)}^{k}}\mathcal{E}_{j}^r(s)=0\right\} } ds +
\int_0^t x_j(s) \one_{\left\{ \sum_{j\in \zeta _{\check{i}(0)}^{k}}\mathcal{E}_{j}^r(s)>0\right\} } ds.$$
From \eqref{eq:eq129} , on the above event, for $t\in \left[ 0,\hat{\tau}_{k }^{s}(0)\wedge
2ry/\Delta \right)$
$$\sum_{j\in \zeta _{\check{i}(0)}^{k}}\int_0^t x_j(s) \one_{\left\{ \sum_{j\in \zeta _{\check{i}(0)}^{k}}\mathcal{E}_{j}^r(s)=0\right\} } ds \ge
\int_0^t (\sum_{j\in \zeta _{\check{i}(0)}^{k}} \varrho _{j}^{r})\one_{\left\{ \sum_{j\in \zeta _{\check{i}(0)}^{k}}\mathcal{E}_{j}^r(s)=0\right\} } ds
+ \Delta \int_0^t \one_{\left\{ \sum_{j\in \zeta _{\check{i}(0)}^{k}}\mathcal{E}_{j}^r(s)=0\right\} } ds$$
Thus, recalling the definition of $\mathcal{A}_{y}^{r}$
\begin{align*}
\sum_{j\in \zeta _{\check{i}(0)}^{k}}\left( t\varrho _{j}^{r}-B_{j}(t)\right) 
&\le \int_0^t \sum_{j\in \zeta _{\check{i}(0)}^{k}} (\varrho _{j}^{r}-x_j(s)) \one_{\left\{ \sum_{j\in \zeta _{\check{i}(0)}^{k}}\mathcal{E}_{j}^r(s)\neq0\right\} } ds
- \Delta t +\Delta \int_0^t \one_{\left\{ \sum_{j\in \zeta _{\check{i}(0)}^{k}}\mathcal{E}_{j}^r(s)\neq0\right\} } ds\\
&\le \frac{C_{\check{i}(0)}yr}{4\bar C}-\Delta t  + \Delta \frac{yr}{4\Delta}\\
&\le r\frac{y}{2}-\Delta t\text{}
\end{align*}
and consequently on the event $\left( \mathcal{A}_{y}^{r}\bigcup \mathcal{B}%
_{y}^{r}\right) ^{c}$ for all $t\in \left[ 0,\hat{\tau}_{k }^{s}(0)\wedge
2ry/\Delta \right) $ we have (since $y\ge d$)
\begin{eqnarray*}
\min_{i\in N_{\rho (k)}}\left\{ \sum_{j\in \zeta _{i}^{k}}
\frac{Q_{j}^{r}(t)}{\mu_{j}^{r}}\right\}  &\leq r(d+y)-\Delta t \le 2ry - \Delta t .
\end{eqnarray*}
Since at $t = 2ry/\Delta$, $2ry - \Delta t=0$, we must have 
$
\hat{\tau}_{k }^{s}(0)< 2ry/\Delta 
$
so that on the above event
\begin{equation*}
\int_{0}^{\hat{\tau}_{k }^{s}(0)}\min_{i\in N_{\rho (k)}}\left\{ \sum_{j\in
\zeta _{i}^{k}}\frac{Q_{j}^{r}(t)}{\mu _{j}^{r}}\right\} dt\leq \frac{4}{\Delta }%
r^{2}y^{2}\text{.}
\end{equation*}%
This gives for $r\geq R_{1}$ and $y\geq \max \{d,1\}$%
\begin{equation*}
P_x\left( \int_{0}^{\hat{\tau}_{k }^{s}(0)}\min_{i\in N_{\rho (k)}}\left\{
\sum_{j\in \zeta _{i}^{k}}\frac{Q_{j}^{r}(t)}{\mu _{j}^{r}}\right\} dt> \frac{4}{%
\Delta }r^{2}y^{2}\right) \leq B_{1}e^{-B_{2}y}\text{.}
\end{equation*}%
A straightforward calculation now shows that
\begin{eqnarray*}
E_x\left[ \int_{0}^{\hat{\tau}_{k }^{s}(0)}\min_{i\in N_{\rho (k)}}\left\{
\sum_{j\in \zeta _{i}^{k}}\frac{Q_{j}^{r}(t)}{\mu _{j}^{r}}\right\} dt\right]  
&\leq &
r^{2}B_{3} (1+ |q|^2)
\end{eqnarray*}%
where $B_3$ depends only on $B_1, B_2$ and $\delta$.
This proves the second statement in the lemma. \ The
proof of the first statement
follows in a very similar manner and is omitted.
\end{proof}

The following proposition will be the second ingredient in the proof of Theorem \ref{thm:discCostInefBnd}.
\begin{proposition}
\label{thm:runningInef}There exist  $H,R\in(0,\infty) $ such that for
all $r\geq R$, $i\in \mathbb{N}_{I}$, $k\in\mathbb{N}_{m}$, and $0\leq T_{1}<T_{2}<\infty $
satisfying $T_{2}-T_{1}\geq 1$ we have 
\begin{equation*}
\frac{1}{r^{3}}E\left[ \int_{\hat{\tau}_{i}^{1}(r^{2}T_{1})}^{\hat{\tau}%
_{i}^{1}(r^{2}T_{2})}\sum_{j\in \zeta _{i}^{0}}\frac{Q_{j}^{r}(s)}{\mu _{j}^{r}}ds%
\right] \leq (T_{2}-T_{1})H r^{\alpha -1/2}
\end{equation*}%
and%
\begin{equation*}
\frac{1}{r^{3}}E\left[ \int_{\hat{\tau}_{k}^{s}(r^{2}T_{1})}^{\hat{\tau}%
_{k}^{s}(r^{2}T_{2})}\min_{i\in
N_{\rho (k)}}\left\{ \sum_{j\in \zeta _{i}^{k}}\frac{Q_{j}^{r}(s)}{\mu
_{j}^{r}}\right\} ds\right] \leq
(T_{2}-T_{1})Hr^{\alpha -1/2}
\end{equation*}
\end{proposition}

\begin{proof}
	Once again we only prove the second statement since the proof of the first statement is similar. 
	Many steps in the proof are similar to those in Proposition \ref{thm:initInef} but we give details to keep the proof self contained.
Let $k\in \mathbb{N}_{m}$ be arbitrary. \ Recall 
$
\bar \mu_{\min}=\min_{j\in \mathbb{N}_{J}}\{\mu _{j}\}\text{.}
$
and $
\bar C=\max_{i\in \mathbb{N}_{I}}\{C_{i}\}$. Also let  for $k \in \mathbb{N}_{m}$
\begin{equation}
	Z^r_k(t) \doteq \min_{i\in
N_{\rho (k)}}\left\{ \sum_{j\in \zeta _{i}^{k}}\frac{Q_{j}^{r}(t)}{\mu
_{j}^{r}}\right\}.\label{eq:eqzrkt}
\end{equation}
Define the stopping times, $\tau_0 \doteq r^2T_1$ and for $l\in \mathbb{N}$
\begin{equation*}
	\tau _{2l-1}\doteq \inf \left\{ t\geq \tau _{2l-2}: Z^r_k(t) \geq
	2c_{3}r^{\alpha }\right\},\; 
\tau _{2l}=\inf \left\{ t\geq \tau _{2l-1}:Z^r_k(t) <2c_{3}r^{\alpha
}\right\} \text{.}
\end{equation*}%
Let $\hat{l}\doteq \min \{l\geq 0:\tau
_{2l+1}>r^{2}T_{2}\}$. Then recalling the definition of $\hat \tau_k^s(\xi)$ from 
 (\ref{eq:defTauS}),
$\hat{\tau}_{k}^{s}(r^{2}T_{2})=r^{2}T_{2}\vee \tau _{2\hat{l}}$. \ Consequently we can write 
\begin{align}
E\left[ \int_{\hat{\tau}_{k}^{s}(r^{2}T_{1})}^{\hat{\tau}%
_{k}^{s}(r^{2}T_{2})}Z^r_k(s) ds\right] 
&\leq E\left[ \int_{\hat{\tau}_{k}^{s}(r^{2}T_{1})}^{\tau _{1}\wedge
r^{2}T_{2}}Z^r_k(s) ds\right]   
+ E\left[ \sum_{l=1}^{\infty }\mathit{I}_{\{\tau _{2l}\leq
r^{2}T_{2}\}}\int_{\tau _{2l}}^{\tau _{2l+1}\wedge r^{2}T_{2}}Z^r_k(s) ds\right]  \notag\\ 
&\quad+E\left[ \sum_{l=0}^{\infty }\mathit{I}_{\{\tau _{2l+1}\leq
r^{2}T_{2}\}}\int_{\tau _{2l+1}}^{\tau _{2l+2}}Z^r_k(s) ds%
\right] \text{.}  \label{eq:runCostDecomp} 
\end{align}%
By definition, for all $s\in \left[ \hat{\tau}_{k}^{s}(r^{2}T_{1}),\tau
_{1}\wedge r^{2}T_{2}\right) $ and $s\in \left[ \tau _{2l},\tau
_{2l+1}\wedge r^{2}T_{2}\right) $ we have
$
Z^r_k(s) \leq 2c_{3}r^{\alpha }
$
which gives 
\begin{equation}
E\left[ \int_{\hat{\tau}_{k}^{s}(r^{2}T_{1})}^{\tau _{1}\wedge
r^{2}T_{2}}Z^r_k(s) ds\right] 
+E\left[ \sum_{l=1}^{\infty }\mathit{I}_{\{\tau _{2l}\leq
r^{2}T_{2}\}}\int_{\tau _{2l}}^{\tau _{2l+1}\wedge r^{2}T_{2}}Z^r_k(s) ds\right]  
\leq 2c_{3}r^{\alpha +2}(T_{2}-T_{1})\text{.}  
	\label{eq:runCostBndedComp} 
\end{equation}%
For all $l\in \mathbb{N}$ let $\check{i}(l)\in N_{\rho (k)}$ satisfy 
\begin{equation*}
\sum_{j\in \zeta _{\check{i}(l)}^{k}}\frac{Q_{j}^{r}(\tau _{2l+1})}{\mu
_{j}^{r}}=\min_{i\in N_{\rho (k)}}\left\{ \sum_{j\in \zeta
_{i}^{k}}\frac{Q_{j}^{r}(\tau _{2l+1})}{\mu _{j}^{r}}\right\} = Z^r_k(\tau _{2l+1})
\end{equation*}%
and note that 
\begin{equation*}
\sum_{j\in \zeta _{\check{i}(l)}^{k}}\frac{Q_{j}^{r}(\tau _{2l+1})}{\mu
_{j}^{r}}\leq 2c_{3}r^{\alpha }+\frac{2}{\bar \mu_{\min}}\text{.}
\end{equation*}%
Recall the definition of $\Delta$ in \eqref{eq:defTriangle}
and define for $y \in \mathbb{R}_+$ and $l \in \mathbb{N}$, the events 
\begin{equation*}
\mathcal{A}_{l,y}^{r}=\left\{ \sum_{i\in \mathbb{N}_{I}}\int_{\tau
_{2l+1}}^{\left( \tau _{2l+1}+2r^{1/4+\alpha /2}y/\Delta \right) \wedge \tau
_{2l+2}}\mathit{I}_{\left\{ \sum_{j\in \zeta _{i}^{k}}\mathcal{E}_{j}^r(s)>0\right\} }ds\geq \frac{r^{1/4+\alpha /2}}{4\left( \bar C\vee
\Delta \right) }y\right\} 
\end{equation*}%
and 
\begin{eqnarray*}
\mathcal{B}_{l,y}^{r} &=&\left\{ \sum_{j\in \mathbb{N}_{J}}\sup_{\tau
_{2l+1}\leq t\leq \tau _{2l+1}+2r^{1/4+\alpha /2}y/\Delta }\left\vert
A_{j}^{r}(t)-A_{j}^{r}(\tau _{2l+1})-(t-\tau _{2l+1})\lambda
_{j}^{r}\right\vert \right.  \\
&&\left. +\sum_{j\in \mathbb{N}_{J}}\sup_{\tau _{2l+1}\leq t\leq \tau
_{2l+1}+\bar C2r^{1/4+\alpha /2}y/\Delta }\left\vert
S_{j}^{r}(t)-S_{j}^{r}(\tau _{2l+1})-(t-\tau _{2l+1})\mu _{j}^{r}\right\vert
\geq \frac{\bar \mu_{\min}r^{1/4+\alpha /2}}{8}y\right\} 
\end{eqnarray*}%
From the strong Markov property, Theorems \ref{thm:IdleTimeExp} (cf. \eqref{thm:IdleTimeExp}) and \ref{thm:LDP} there
exist  $B_{1},B_{2} \in (0,\infty) $ and $R_{1}\in \left[ \hat{R},\infty
\right) $ such that for all $r\geq R_{1}$, $y\geq \Delta /2$, and $l\in 
\mathbb{N}$ we have%
\begin{equation}
r^{1/4+\alpha /2}\Delta /2>\frac{2}{\bar \mu_{\min}} \mbox{ and } 	P\left( \mathcal{A}_{l,y}^{r}\bigcup \mathcal{B}_{l,y}^{r}\right) \leq
	B_{1}e^{-B_{2}y}\text{.} \label{eq:triangleRcond}
\end{equation}%
We claim that for  $t\in \left[ \tau_{2l+1},\tau_{2l+2} \right) $ we have 
$\zeta _{i'}^{k} \cap \sigma^r(t)\neq \emptyset$ for all $i' \in N_{\rho(k)}$.
To see the claim note that for all $i' \in N_{\rho(k)}$, 
$\sum_{j \in \zeta _{i'}^{k}} \frac{Q_j^r(t)}{\mu_j^r} \ge \min_{i \in N_{\rho(k)}} \sum_{j \in \zeta _{i}^{k}} \frac{Q_j^r(t)}{\mu_j^r} \ge 2c_3 r^{\alpha}.$
Thus, from the definition of $c_3$ there is a $j \in \zeta_{i'}^{k}$ such that
$$Q_j^r(t) \ge \frac{2c_3}{J} r^{\alpha}\mu_j^r \ge \frac{c_3}{J} r^{\alpha}\mu_j \ge c_2r^{\alpha},$$
namely $j\in \sigma^r(t)$. Thus we have $\zeta _{i'}^{k} \cap \sigma^r(t)\neq \emptyset$ proving the claim.  From Lemma \ref{lem:lem3_4}(a)
for $i\in N_{\rho (k)}$ and $t\in \left[
\tau _{2l+1},\tau _{2l+2}\right) $ such that\ $\sum_{j\in \zeta _{i}^{k}}%
\mathcal{E}_{j}^r(t)=0$ we now have 
\begin{equation}
\sum_{j\in \zeta _{i}^{k}}\left( \varrho _{j}^{r}-x_{j}(t)\right) \leq
-2^{-m -2}\delta =-\Delta \text{%
.}  \label{eq:arrProcRateDif}
\end{equation}%
Consequently on the event $\left( \mathcal{A}_{l,y}^{r}\bigcup \mathcal{B}%
_{l,y}^{r}\right) ^{c}$ for all $t\in \left[ \tau _{2l+1},\tau _{2l+2}\wedge
\left( \tau _{2l+1}+2r^{1/4+\alpha /2}y/\Delta \right) \right) $ we have 
\begin{eqnarray*}
\min_{i\in N_{\rho (k)}}\left\{ \sum_{j\in \zeta _{i}^{k}}\frac{Q_{j}^{r}(t)}{\mu
_{j}^{r}}\right\}  &\leq &\sum_{j\in \zeta _{\check{i}(l)}^{k}}\frac{Q_{j}^{r}(\tau
_{2l+1})}{\mu _{j}^{r}}+\sum_{j\in \zeta _{\check{i}(l)}^{k}}\frac{Q_{j}^{r}(t)}{\mu
_{j}^{r}}-\sum_{j\in \zeta _{\check{i}(l)}^{k}}\frac{Q_{j}^{r}(\tau _{2l+1})}{\mu
_{j}^{r}} \\
&\leq &2c_{3}r^{\alpha }+\frac{2}{\bar \mu_{\min}}+\sum_{j\in \zeta _{\check{i}(l)}^{k}}
\frac{1}{\mu
_{j}^{r}}\left[\left( A_{j}^{r}(t)-A_{j}^{r}(\tau _{2l+1})\right) + \left(
S_{j}^{r}(B_{j}(t))-S_{j}^{r}(B_{j}(\tau _{2l+1}))\right)\right]  \\
&\leq &2c_{3}r^{\alpha }+\frac{2}{\bar \mu_{\min}}+\frac{r^{1/4+\alpha /2}}{4}y+\sum_{j\in \zeta _{%
\check{i}(l)}^{k}}\left( (t-\tau _{2l+1})\varrho
_{j}^{r}-(B_{j}(t)-B_{j}(\tau _{2l+1}))\right) 
\end{eqnarray*}%
where the last line comes from the definition of the event $\mathcal{B}%
_{l,y}^{r}$. \ Note that for all $j\in \mathbb{N}_{J}$ and $t\geq \tau _{2l+1}$ we have 
\begin{eqnarray*}
B_{j}(t)-B_{j}(\tau _{2l+1})&=&\int_{\tau _{2l+1}}^{t}x_{j}(s)ds \\
&=&\int_{\tau _{2l+1}}^{t}x_{j}(s)\mathit{I}_{\left\{ \sum_{j\in \zeta _{i}^{k}}\mathcal{E}_{j}^r(s)>0\right\} }ds+
\int_{\tau _{2l+1}}^{t}x_{j}(s)\mathit{I}_{\left\{ \sum_{j\in \zeta _{i}^{k}}\mathcal{E}_{j}^r(s)=0\right\} }ds\text{.}
\end{eqnarray*}
From \eqref{eq:arrProcRateDif}, on the above event and for $t\in \left[ \tau
_{2l+1},\tau _{2l+2}\wedge \left( \tau _{2l+1}+2r^{1/4+\alpha /2}y/\Delta
\right) \right]$
\begin{equation}
\int_{\tau _{2l+1}}^{t}\sum_{j\in \zeta _{i}^{k}}x_{j}(s)\mathit{I}_{\left\{ \sum_{j\in \zeta _{i}^{k}}\mathcal{E}_{j}^r(s)=0\right\} }ds
\geq \int_{\tau _{2l+1}}^{t} \left( \sum_{j\in \zeta_{i}^{k}}\varrho^{r}_{j} \right) \mathit{I}_{\left\{ \sum_{j\in \zeta _{i}^{k}}\mathcal{E}_{j}^r(s)=0\right\} }ds
+\Delta\int_{\tau _{2l+1}}^{t}\mathit{I}_{\left\{ \sum_{j\in \zeta _{i}^{k}}\mathcal{E}_{j}^r(s)=0\right\} }ds
\end{equation}%
so that
\begin{eqnarray*}
\sum_{j\in \zeta _{\check{i}(l)}^{k}}\left( (t-\tau _{2l+1})\varrho
_{j}^{r}-(B_{j}(t)-B_{j}(\tau _{2l+1}))\right)  &\leq &\int_{\tau
_{2l+1}}^{t}\sum_{j\in \zeta _{\check{i}(l)}^{k}}\left( \varrho
_{j}^{r}-x_{j}(s)\right) \mathit{I}_{\left\{ \sum_{j\in \zeta _{i}^{k}}%
\mathcal{E}_{j}^r(s)>0\right\} }ds \\
&&-\Delta (t-\tau _{2l+1})+\Delta \int_{\tau _{2l+1}}^{t}\mathit{I}%
_{\left\{ \sum_{j\in \zeta _{i}^{k}}\mathcal{E}_{j}^r(s)>0\right\} }ds \\
&\leq &\frac{C_{\check{i}(l)}r^{1/4+\alpha /2}}{4\bar C}y-\Delta (t-\tau
_{2l+1})+\frac{r^{1/4+\alpha /2}}{4}y \\
&\leq &\frac{r^{1/4+\alpha /2}}{2}y-\Delta (t-\tau _{2l+1})
\end{eqnarray*}
where the second line follows because we are on the set  $\left( \mathcal{A}_{l,y}^{r}\right) ^{c}$.  Consequently on the event $\left( \mathcal{A}_{l,y}^{r}\bigcup \mathcal{B}%
_{l,y}^{r}\right) ^{c}$ for $t\in \left[ \tau _{2l+1},\tau _{2l+2}\wedge
\left( \tau _{2l+1}+2r^{1/4 +\alpha /2}y/\Delta \right) \right) $ 
\begin{equation}
Z^r_k(t) \leq 2c_{3}r^{\alpha }+\frac{2}{\bar \mu_{\min}}+r^{1/4+\alpha /2}y-\Delta
(t-\tau _{2l+1})\text{.}
\label{eq:minUpperBnd}
\end{equation}%
The right side of \eqref{eq:minUpperBnd} with $t=\tau _{2l+1}+2r^{1/4+\alpha /2}y/\Delta $ equals
\begin{eqnarray*}
2c_{3}r^{\alpha }+\frac{2}{\bar \mu_{\min}}+r^{1/4+\alpha /2}y-\Delta
(2r^{1/4+\alpha /2}y/\Delta) < 2c_{3}r^{\alpha }
\end{eqnarray*}%
where the inequality is from \eqref{eq:triangleRcond},
and so we must have $\tau _{2l+2}<\tau
_{2l+1}+2r^{1/4+\alpha /2}y/\Delta $.  This combined with \eqref{eq:minUpperBnd} gives on the event $\left( \mathcal{A}_{l,y}^{r}\bigcup \mathcal{B}%
_{l,y}^{r}\right) ^{c}$
\begin{equation*}
\int_{\tau _{2l+1}}^{\tau _{2l+2}}Z^r_k(s) ds\leq Ky^{2}r^{1/2+\alpha }
\end{equation*}%
for a $K<\infty $ depending only on $c_3, \bar \mu_{\min}$ and $\Delta$. \ Then for $y\geq B_{3}=\max
\{2c_{3}+\frac{2}{\bar \mu_{\min}},\frac{\Delta}{2}\}$ we have from \eqref{eq:triangleRcond}
\begin{equation*}
P_{X^{r}(\tau _{2l+1})}\left(  \int_{\tau _{2l+1}}^{\tau _{2l+2}}Z^r_k(t)
dt >  Ky^{2}r^{1/2+\alpha }\right) \leq
B_{1}e^{-B_{2}y}
\end{equation*}%
and a standard argument now gives 
\begin{eqnarray}
E_{X^{r}(\tau _{2l+1})}\left[  \int_{\tau _{2l+1}}^{\tau _{2l+2}}Z^r_k(t)
dt\right] \leq B_{4}r^{1/2+\alpha }  \label{eq:runCostInefBndStopTime}
\end{eqnarray}%
where the constant $B_{4}$  depends only on  $B_1, B_2, B_3$ and $K$.   Let 
\begin{equation*}
L^{r}=\max \{l\geq 1:\tau _{2l+1}\leq T_{2}r^{2}\}\text{.}
\end{equation*}%
Note that for all $l\geq 1$ each occurence of $\tau _{2l+1}$ implies an
arrival of a job of type $j\in \bigcup _{i\in N_{\rho (k)}}\zeta _{i}^{k}$ in
the interval $\left( \tau _{2l},\tau _{2l+1}\right] $, so that for some $K_1 \in (0,\infty)$
\begin{equation*}
E_x L^{r}\leq K_1r^2(T_2-T_1) \mbox{ for all } x  \in \cls^r
\end{equation*}%
Consequently 
\begin{eqnarray*}
E\left[ \sum_{l=0}^{\infty }\mathit{I}_{\{\tau _{2l+1}  \le r^2T_2\}} \int_{\tau_{2l+1} }^{\tau_{2l+2}} Z^r_k(s) ds\right] 
\leq
B_{4}r^{1/2+\alpha }E_x\left[ {L}^{r}\right]  \leq B_{5}r^{2+1/2+\alpha }(T_{2}-T_{1}),
\end{eqnarray*}
where $B_5 \doteq K_1B_4$.
 This, combined with (\ref{eq:runCostDecomp}) and (\ref{eq:runCostBndedComp}) gives 
\begin{equation*}
\frac{1}{r^{3}}E\left[ \int_{\hat{\tau}_{k}^{s}(r^{2}T_{1})}^{\hat{\tau}%
_{k}^{s}(r^{2}T_{2})}Z^r_k(s)ds\right] \leq \left(2c_{3}r^{\alpha
-1}+B_{5}r^{\alpha -1/2}\right)(T_{2}-T_{1})\text{.}
\end{equation*}%
The result follows.
\end{proof}

We can now complete the proof of Theorem \ref{thm:discCostInefBnd}.

{\bf Proof of Theorem \ref{thm:discCostInefBnd}.}
Let $R<\infty $ be given by the maximum of the two $R$ values from Propositions %
\ref{thm:initInef} and \ref{thm:runningInef}. \ Note that by (\ref{eq:eq942}%
) , for all $t\ge 0$
\begin{equation}\label{eq:eq851}
	h \cdot \hat{Q}^{r}(t) \ge \clc\left( \tilde{W}^{r}(t)\right)
\end{equation}
and by Theorem \ref{thm:costInefIneq}, there is a $B_1 \in (0,\infty)$ such that for all $t,r$,
\begin{equation}\label{eq:eq853}
	h \cdot \hat{Q}^{r}(t) - \clc\left( \tilde{W}^{r}(t)\right) \le B_1 \left(\sum_{k\in \mathbb{N}_{m}}\min_{i\in N_{\rho (k)}}\left\{ \sum_{j\in \zeta _{i}^{k}}\frac{\hat{%
	Q}_{j}^{r}(t)}{\mu^r_j}\right\} + \sum_{i=1}^{I}\sum_{j\in \zeta _{i}^{0}}\frac{\hat{Q}_{j}^{r}(t)}{\mu^r_j}\right).
\end{equation}
Let $Z_k^r$ be as in \eqref{eq:eqzrkt}. From monotone convergence we have for $\theta \ge 0$
\begin{equation}
\lim_{n\rightarrow \infty }\frac{1}{r^{3}}E\left[ \int_{0}^{\hat{\tau}%
_{k}^{s}(r^{2}n)}e^{-\theta t/r^{2}}Z^r_k(t) dt\right] =E\left[
\int_{0}^{\infty }e^{-\theta t}\min_{i\in N_{\rho (k)}}\left\{ \sum_{j\in \zeta _{i}^{k}}\frac{\hat{%
Q}_{j}^{r}(t)}{\mu^r_j}\right\} dt\right] \text{.}\label{eq:eqmct931}
\end{equation}%
Note that 
\begin{eqnarray*}
\frac{1}{r^{3}}E\left[ \int_{0}^{\hat{\tau}_{k}^{s}(r^{2}n)}e^{-\theta
t/r^{2}}Z^r_k(t) dt\right]  &=&\frac{1}{r^{3}}E\left[ \int_{0}^{%
\hat{\tau}_{k}^{s}(0)}e^{-\theta t/r^{2}}Z^r_k(t)dt\right]  \\
&&+\sum_{l=1}^{n}\frac{1}{r^{3}}E\left[ \int_{\hat{\tau}%
_{k}^{s}(r^{2}(l-1))}^{\hat{\tau}_{k}^{s}(r^{2}l)}e^{-\theta
t/r^{2}}Z^r_k(t) dt\right] \text{.}
\end{eqnarray*}%
From Proposition \ref{thm:initInef}, we have for some $B_2 \in (0,\infty)$,
for $r\geq R$, $\theta\ge 0$ and $x \in \cls^r$,
\begin{equation}
	\label{eq:eq903}
\frac{1}{r^{3}}E_x\left[ \int_{0}^{\hat{\tau}_{k}^{s}(0)}e^{-\theta
t/r^{2}}Z^r_k(t) dt\right]  \leq  B_{2}r^{-1}(1+|q|^2).
\end{equation}%
Also, from  Theorem \ref{thm:runningInef}, there is $B_3 \in (0,\infty)$ such that for $k \in \mathbb{N}_m$,
$r\geq R$ and any $l\in \mathbb{N}$ 
\begin{eqnarray*}
\frac{1}{r^{3}}E\left[ \int_{\hat{\tau}_{k}^{s}(r^{2}(l-1))}^{\hat{\tau}%
_{k}^{s}(r^{2}l)}e^{-\theta t/r^{2}}Z^r_k(t) dt\right]  &\leq &\frac{1}{%
r^{3}}e^{-\theta (l-1)}E\left[ \int_{\hat{\tau}_{k}^{s}(r^{2}(l-1))}^{\hat{%
\tau}_{k}^{s}(r^{2}l)}Z^r_k(t) dt\right]  \\
&\leq &B_{3}e^{-\theta (l-1)}r^{\alpha -1/2}
\end{eqnarray*}%

Consequently for $r\geq R$%
\newline
\begin{equation*}
\frac{1}{r^{3}}E\left[ \int_{0}^{\hat{\tau}_{k}^{s}(r^{2}n)}e^{-\theta
t/r^{2}}Z^r_k(t) dt\right] \leq \left(
B_{2} (1+|q|^2)+B_{3}\sum_{l=0}^{n-1}e^{-l\theta }\right) r^{\alpha
-1/2}.
\end{equation*}%
Sending $n\to \infty$, using \eqref{eq:eqmct931}, we have for $\theta >0$ and all $k \in \mathbb{N}_m$
\begin{eqnarray*}
E\left[ \int_{0}^{\infty }e^{-\theta t}	\min_{i\in N_{\rho (k)}}\left\{ \sum_{j\in \zeta _{i}^{k}}\frac{\hat{%
	Q}_{j}^{r}(t)}{\mu^r_j}\right\} dt \right]  &\leq
&\left( B_{2} (1+|q|^2)+\frac{B_{3}}{1-e^{-\theta }}\right)
r^{\alpha -1/2}.
\end{eqnarray*}%
A similar argument shows that  there are $B_4,B_5 \in (0,\infty)$ such that for all $i\in \mathbb{N}_{I}$ and $r\geq R$%
\begin{equation*}
E\left[ \int_{0}^{\infty }e^{-\theta t}\sum_{j\in \zeta _{i}^{0}}	\frac{\hat{%
		Q}_{j}^{r}(t)}{\mu^r_j}dt\right] \leq \left( B_{4} (1+|q|^2)+\frac{B_{5}}{%
1-e^{-\theta }}\right) r^{\alpha -1/2}.
\end{equation*}%
Combining the above two estimates with \eqref{eq:eq851} and \eqref{eq:eq853} we have the first inequality in the theorem.

For the second inequality, we write
\begin{align*}
E_x\left[ \frac{1}{T}\int_{0}^{T}\min_{i\in N_{\rho (k)}}\left\{ \sum_{j\in
\zeta _{i}^{k}}\frac{\hat{Q}_{j}^{r}(t)}{\mu_j^r}\right\} dt\right]  &= E_x\left[ \frac{1}{%
Tr^{3}}\int_{0}^{Tr^{2}}Z^r_k(t) dt\right]  \\
&\leq \frac{1}{r^{3}}E_x\left[ \frac{1}{T}\int_{0}^{\hat{\tau}%
_{k}^{s}(0)}Z^r_k(t) dt\right]  
+\frac{1}{r^{3}}E_x\left[ \frac{1}{T}\int_{\hat{\tau}_{k}^{s}(0)}^{\hat{\tau}%
_{k}^{s}(r^{2}T)}Z^r_k(t) dt\text{.}\right] 
\end{align*}%
 Applying \eqref{eq:eq903} with $\theta=0$ we have for $T\ge 1$
\begin{equation}
	\label{eq:eq903b}
\frac{1}{r^{3}}E_x\left[ \frac{1}{T}\int_{0}^{\hat{\tau}_{k}^{s}(0)}
  Z^r_k(t) dt\right]  \leq  B_{2}r^{-1}(1+|q|^2).
\end{equation}%
Also, from Theorem \ref{thm:runningInef} for $%
r\geq R$ we have, for some $\tilde B_3 \in (0,\infty)$ and all $T\ge 1$, $k \in \mathbb{N}_m$,
\begin{eqnarray*}
\frac{1}{r^{3}}E_x\left[ \frac{1}{T}\int_{\hat{\tau}_{k}^{s}(0)}^{\hat{\tau}%
_{k}^{s}(r^{2}T)}Z^r_k(t) dt\right]  &\leq &\frac{1}{T}%
\tilde B_{3} r^{\alpha -1/2}T 
\leq \tilde B_{3} r^{\alpha -1/2}.
\end{eqnarray*}%
 Consequently, for all $T\ge 1$ and $k \in \mathbb{N}_m$
\begin{equation*}
E_x\left[ \frac{1}{T}\int_{0}^{T}\min_{i\in N_{\rho (k)}}\left\{ \sum_{j\in
\zeta _{i}^{k}}\frac{\hat{Q}_{j}^{r}(t)}{\mu_j^r}\right\} dt\right] \leq
B_{2}r^{-1}(1+ |q|^2)+\tilde B_{3}r^{\alpha -1/2}
\end{equation*}%
A similar argument shows that for some $\tilde B_4, \tilde B_5 \in (0,\infty)$, and all $i\in \mathbb{N}_{I}$, $T\ge 1$%
\begin{equation*}
E_x\left[ \frac{1}{T}\int_{0}^{T}\sum_{j\in \zeta^{0}_{i}}\frac{\hat{Q}_{j}^{r}(t)}{\mu^r_j}dt%
\right] \leq \tilde B_{4}r^{-1} (1+|q|^2)+\tilde B_{5} r^{\alpha -1/2}\text{.}
\end{equation*}%
Combining the above two estimates with \eqref{eq:eq851} and \eqref{eq:eq853} once more, we have the second inequality in the theorem. \hfill \qed

\subsection{Lyapunov function and uniform moment estimates.}
\label{sec:unifmom}
In this section we establish uniform in $t$ and $r$ moment bounds on $\hat W^r(t)$. 
The following is the main result of this section. 
\begin{theorem}
 \label{thm:WmomBnd}There exist  $\beta, \gamma, R,H \in (0,\infty)$ such that
for all $i \in \mathbb{N}_I$, $t\ge 0$ and $r\ge R$
$$E_x\left[e^{\gamma \hat W^r_i(t)}\right] \le H\left( 1 + e^{-\beta t} V_i(x)\right).$$
\end{theorem}
The proof is given at the end of the section. 
Let
\begin{equation}
\check{\tau}_{i,\xi }^{r}\doteq \inf \left\{  t\geq \xi :\left\vert \hat{W%
}_{i}^{r}(t)\right\vert \leq 2c_{3}\right\},\label{eq:eq1012} 
\end{equation}
where recall that $c_3 = \frac{2Jc_{2}}{\min_{j}\mu_{j}}$.
We begin by establishing
 a bound on certain exponential moments of $\check{\tau}_{i,\xi }^{r}$.
\begin{proposition}
 \label{thm:stopTimeExpMomBnd} There
exist  $\delta^{*},R\in (0,\infty)$ and $H_1: \mathbb{R}_+ \to \mathbb{R}_+$ such that for all  $i\in \mathbb{N}_{I}$, $r\geq R$ 
and $0<\beta <\delta^{*} $ 
\begin{equation*}
E_x\left[ e^{\beta \check{\tau}_{i,\xi }^{r}}\right] <H_{1}(\beta)e^{H_{1}(\beta)(w_{i}+%
\xi )}
\end{equation*}%
for all  $x=(q,z)\in \cls^r$ and $\xi \geq 0$, where $w=G^{r}q$.
\end{proposition}

\begin{proof}
Fix $i \in \NI$. Given $x=\left( q,z\right) \in  \cls^r$ let $w_{i}=(G^{r}q)_{i}$. \ Recall
the definition of $v^{\ast }$ given in Condition \ref{cond:HT1}. \ Fix $\xi\ge 0$ and let  
\begin{equation}\label{eq:Mdef}
t\geq \max \{2\xi
,8w_{i} /v_{i}^{\ast },1\}\doteq M_{\xi}.
\end{equation}
Consider the events 
\begin{equation*}
\mathcal{A}_{i,t}^{r}=\left\{ \int_{0}^{r^{2}t}\mathit{I}_{\left\{
W_{i}^{r}(t) \ge c_{3}r^{\alpha }\right\} }(s)dI_{i}^{r}(s)\geq \frac{v_{i}^{\ast}}{32C_{i}}rt\right\} 
\end{equation*}%
and 
\begin{equation*}
\mathcal{B}_{i,t}^{r}=\bigcup _{j\in \mathbb{N}_{J}}\left\{ \sup_{0\leq s\leq
r^{2}t}\left\vert A_{j}^{r}(s)-s\lambda _{j}^{r}\right\vert +\sup_{0\leq
s\leq C_{i}r^{2}t}\left\vert S_{j}^{r}(s)-s\mu _{j}^{r}\right\vert \geq 
\frac{\min \{1,\bar{\mu}_{\min}\}v_{i}^{\ast }}{256J}%
rt\right\} \text{.}
\end{equation*}%
Using \eqref{eq:squareIdleTimeResultImplication} and Theorem \ref{thm:LDP} we can
choose  $\hat{H}_{1},\hat{H}_{2} \in (0,\infty)$ and $R \in (\hat{R}, \infty)$ such that for all $r\geq R$ and $t\geq 1$ 
\begin{equation*}
P_x\left( \mathcal{A}_{i,t}^{r}\bigcup \mathcal{B}_{i,t}^{r}\right) \leq \hat{H}%
_{1}e^{-t\hat{H}_{2}}\text{,}
\end{equation*}%
where $\hat R$ was introduced above \eqref{def:Rhat}.
Furthermore, we can assume that $R$ is large enough so that for all $r\ge R$, 
\begin{equation}\label{eq:larger1155}
\frac{2v_{i}^{\ast }}{r}\ge C_{i}-\sum_{j=1}^{J}K_{i,j}\rho _{j}^{r}\geq \frac{v_{i}^{\ast }}{2r}\text{,}\;\;\; c_{3}r^{\alpha -1}+\frac{2}{\bar{\mu}_{\min}}r^{-1}\leq 2c_{3}.
\end{equation}%
Then, for all $r\geq R$  and $s_{1},s_{2}\in \lbrack 0,r^{2}t]$ satisfying $s_{2}>s_{1}$, on the event 
$
\left( \mathcal{A}_{i,t}^{r}\bigcup \mathcal{B}_{i,t}^{r}\right) ^{c}
$,
we have 
\begin{eqnarray*}
\sum_{j=1}^{J}\frac{K_{i,j}}{\mu_{j}^{r}} S_{j}^{r}(B_{j}^{r}(s_{2}))-\sum_{j=1}^{J}\frac{K_{i,j}}{\mu_{j}^{r}} S_{j}^{r}(B_{j}^{r}(s_{1}))
 &\geq &\sum_{j=1}^{J}K_{i,j}\left(
B_{j}^{r}(s_{2})-B_{j}^{r}(s_{1})\right) -\sum_{j=1}^{J}K_{i,j}\frac{\min
\{1,\bar{\mu}_{\min}\}v_{i}^{\ast }}{128J\mu _{j}^{r}}rt \\
&\geq &\sum_{j=1}^{J}K_{i,j}\left( B_{j}^{r}(s_{2})-B_{j}^{r}(s_{1})\right) -%
\frac{v_{i}^{\ast }}{64 }rt
\end{eqnarray*}%
and%
\begin{eqnarray*}
\sum_{j=1}^{J} \frac{K_{i,j}}{\mu_{j}^{r}} A_{j}^{r}(s_{2})-
\sum_{j=1}^{J}\frac{K_{i,j}}{\mu_{j}^{r}}A_{j}^{r}(s_{1})
&\leq &\sum_{j=1}^{J}K_{i,j}\rho _{j}^{r}(s_{2}-s_{1})+\sum_{j=1}^{J}K_{i,j}%
\frac{\min \{1,\bar{\mu}_{\min}\}v_{i}^{\ast }}{128J\mu
_{j}^{r}}rt \\
&\leq &\sum_{j=1}^{J}K_{i,j}\rho _{j}^{r}(s_{2}-s_{1})+\frac{v_{i}^{\ast }}{%
64}rt\text{.}
\end{eqnarray*}%
Let $\sigma_0=0$ and for $k\ge 1$
\begin{equation*}
\sigma _{2k-1}=\inf \{s\geq \sigma _{2k-2}:W_{i}^{r}(t)\geq c_{3}r^{\alpha }\},\; \sigma _{2k}=\inf \{s\geq \sigma _{2k-1}:W_{i}^{r}(t)<c_{3}r^{\alpha }\}.%
\end{equation*}%
\ Then, on the  event $
\left( \mathcal{A}_{i,t}^{r}\bigcup \mathcal{B}_{i,t}^{r}\right) ^{c}
$,  for any $\sigma _{2k-1}<\xi r^{2}$, $k\ge 1$, we have 
on noting that $W_{i}(\sigma _{2k-1}) \le c_{3}r^{\alpha }+\frac{2}{\bar{\mu}_{\min}}+w_{i}r$
\begin{align*}
\sup_{\sigma _{2k-1}\leq s\leq \sigma _{2k}\wedge \xi r^{2}}W_{i}^{r}(s)&
\leq \sup_{\sigma _{2k-1}\leq s\leq \sigma _{2k}\wedge \xi r^{2}}\left(
W_{i}^r(\sigma _{2k-1})+\sum_{j=1}^{J}\frac{K_{i,j}}{\mu_{j}^{r}} A_{j}^{r}(s)-\sum_{j=1}^{J}\frac{K_{i,j}}{\mu_{j}^{r}} A_{j}^{r}(\sigma _{2k-1})\right .  \\
& \qquad \qquad \qquad \quad \left. -\left( \sum_{j=1}^{J}\frac{K_{i,j}}{\mu_{j}^{r}}
S_{j}^{r}(B_{j}^{r}(s))-\sum_{j=1}^{J}\frac{K_{i,j}}{\mu_{j}^{r}}
S_{j}^{r}(B_{j}^{r}(\sigma _{2k-1}))\right) \right)  \\
& \leq \sup_{\sigma _{2k-1}\leq s\leq \sigma _{2k}\wedge \xi r^{2}}\left(
\sum_{j=1}^{J}K_{i,j}\rho _{j}^{r}(s-\sigma _{2k-1})-
\sum_{j=1}^{J}K_{i,j}\left( B_{j}^{r}(s)-B_{j}^{r}(\sigma _{2k-1})\right)
 \right.  \\
& \qquad \qquad \quad \qquad \left. +c_{3}r^{\alpha }+\frac{2}{\bar{\mu}_{\min}}+w_{i}r+\frac{%
v_{i}^{\ast }}{32}rt\right)  \\
\qquad \qquad \qquad \qquad \quad & \leq c_{3}r^{\alpha
}+\frac{2}{\bar{\mu}_{\min}}+w_{i}r +\frac{v_{i}^{\ast }}{16}rt
\end{align*}%
where the third inequality follows from recalling that we are on the event $\left( \mathcal{A}_{i,t}^{r}\right) ^{c}$ so
\begin{align*}
\sum_{j=1}^{J}K_{i,j} \left[\rho _{j}^{r}(s-\sigma _{2k-1}) - ( B_{j}^{r}(s)-B_{j}^{r}(\sigma _{2k-1})\right]	
&\le \left(\sum_{j=1}^{J}K_{i,j}\rho _{j}^{r} - C_i\right) (s-\sigma _{2k-1}) + \frac{v_{i}^{\ast }}{32}rt\\
&\le -\frac{v_{i}^{\ast }}{2r} (s-\sigma _{2k-1}) + \frac{v_{i}^{\ast }}{32}rt \le \frac{v_{i}^{\ast }}{32}rt.
\end{align*}
Thus on the event $\left( \mathcal{A}_{i,t}^{r}\bigcup \mathcal{B}%
_{i,t}^{r}\right) ^{c}$ we have%
\begin{equation*}
\hat{W}_{i}^{r}(\xi )\leq \frac{v_{i}^{\ast }}{16}t+w_{i}+c_{3}r^{\alpha
-1}+\frac{2}{\bar{\mu}_{\min}}r^{-1}\text{.}
\end{equation*}%
Consequently on the event $\left( \mathcal{A}_{i,t}^{r}\bigcup \mathcal{B}%
_{i,t}^{r}\right) ^{c}\cap \{\check{\tau}_{i,\xi }^{r}> t\}$ we have, by a similar calculation,
\begin{eqnarray*}
\hat{W}_{i}^{r}(t) &=&\hat{W}_{i}^{r}(\xi )+\left( \hat{W}_{i}^{r}(t)-\hat{W}%
_{i}^{r}(\xi )\right)  \\
&\leq &\frac{v_{i}^{\ast }}{16}t+w_{i}+c_{3}r^{\alpha -1}+\frac{2}{\bar{\mu}_{\min}}r^{-1}+r(t-\xi
)\sum_{j=1}^{J}K_{i,j}\rho _{j}^{r}-rC_{i}(t-\xi )+\frac{v_{i}^{\ast }}{16}%
t \\
&\leq &\frac{v_{i}^{\ast }}{8}t+w_{i}+c_{3}r^{\alpha -1}+\frac{2}{\bar{\mu}_{\min}}r^{-1}-\frac{t}{2}%
\left( C_{i}-\sum_{j=1}^{J}K_{i,j}\rho _{j}^{r}\right)r  \\
&\leq & \frac{v_{i}^{\ast }}{8}t-t\frac{v_{i}^{\ast }}{4}+w_{i}+c_{3}r^{\alpha -1}+\frac{2}{\bar{\mu}_{\min}}r^{-1} \\
&\leq &2c_{3},
\end{eqnarray*}%
where the third and the fourth inequalities follow from \eqref{eq:larger1155} and recalling that $t \ge \max\{2\xi, 8w_i/v^*_i\}$.
Since on the set $\{\check{\tau}_{i,\xi }^{r}> t\}$ we must have $\hat{W}_{i}^{r}(t) >2c_3$ we have arrived at a contradiction. 
\ Consequently $\left( \mathcal{A}_{i,t}^{r}\bigcup 
\mathcal{B}_{i,t}^{r}\right) ^{c}\cap \{\check{\tau}_{i,\xi }^{r}(x)>
t\}=\emptyset $ and%
\begin{eqnarray*}
P_x\left( \check{\tau}_{i,\xi }^{r}> t\right)  &=&P_x\left( \left( 
\mathcal{A}_{i,t}^{r}\bigcup \mathcal{B}_{i,t}^{r}\right) \cap \{\check{\tau}%
_{i,\xi }^{r}> t\}\right)  
\leq P_x\left( \mathcal{A}_{i,t}^{r}\bigcup \mathcal{B}_{i,t}^{r}\right) \le \hat{H}_{1}e^{-t\hat{H}_{2}}\text{.}
\end{eqnarray*}%
Thus for $\beta <\hat{H}_{2}$ 
\begin{eqnarray*}
E_x\left[ e^{\beta \check{\tau}_{i,\xi }^{r}}\right]  
\leq 1+ \beta e^{\beta M_{\xi}}+\frac{\beta}{\hat{H}_{2}-\beta}e^{(\beta -\hat{H}_{2})M_{\xi}}
\leq H_{1}(\beta)e^{H_{1}(\beta)(\xi +w_{i})}
\end{eqnarray*}%
for suitable $H_{1}(\beta)\in (0,\infty)$, where the last inequality comes from the definition of
$M_{\xi}$ in  \eqref{eq:Mdef}.
\end{proof}

We now establish a lower bound on an exponential moment of $\check{\tau}_{i,0}^{r}$.
\begin{proposition}
 \label{thm:VlowerBnd}For all $i\in \mathbb{N}_{I}$ there exist
 $R,H_{1},H_{2},H_{3} \in (0,\infty)$ such that for all $r\geq R$, $\beta >0$ and $%
x=(q,z)\in \cls^r$
satisfying $w_{i}=(G^{r}q)_{i}\geq H_{1}$ we have 
\begin{equation*}
E_x\left[ e^{\beta \check{\tau}_{i,0}^{r}}\right] >H_{2}e^{H_{3}\beta w_{i}}
\end{equation*}
\end{proposition}

\begin{proof}
For $k\in (0,\infty )$ define the event 
\begin{equation*}
\mathcal{B}_{i,k}^{r}=\bigcup_{j\in \mathbb{N}_{J}}\left\{ \sup_{0\leq s\leq
k}\left\vert \frac{A_{j}^{r}(r^{2}s)}{r}-rs\lambda _{j}^{r}\right\vert +\sup_{0\leq
s\leq C_{i}k}\left\vert \frac{S_{j}^{r}(r^{2}s)}{r}-rs\mu _{j}^{r}\right\vert \geq 
\frac{v_{i}^{\ast }\min \{1,\bar{\mu}_{\min}\}k}{4J}\right\} 
\text{.}
\end{equation*}%
From Theorem \ref{thm:LDP} there exists $R \in (\hat{R},\infty)$ (recall \eqref{def:Rhat}) and 
$\hat{H}_{2},\hat{H}_{3} \in (0,\infty)$ such that for all $r\geq R$ and $k \in (0,\infty)$, we have%
\begin{equation*}
P\left(  \mathcal{B}_{i,k}^{r}\right) \leq \hat{H}_{2}e^{-k%
\hat{H}_{3}}\text{.}
\end{equation*}%
We assume that $R$ is big enough so that \eqref{eq:larger1155} is satisfied for all $r \ge R$.
Let $H_{1}=\max \left\{ 5c_{3},\frac{6v_{i}^{\ast }\log (2\hat{H}_{2})}{\hat{%
H}_{3}}\right\} $. \ Then for $w_{i}\geq H_{1}$ we have 
\begin{align}
P\left( \mathcal{B}_{i,w_{i}/(6v_{i}^{\ast })}^{r}\right) 
\leq \hat{H}_{2}e^{-w_{i}\hat{H}_{3}/6v_{i}^{\ast }} \leq \frac{1}{2} \label{eq:eq852}
\end{align}%
and on the event $(\mathcal{B}_{i,w_{i}/6v_{i}^{\ast }}^{r})^{c}$ we have 
from \eqref{eq:queleneqn} and \eqref{eq:eq939}, that under $P_x$,
\begin{eqnarray*}
\inf_{0\leq s\leq w_{i}/6v_{i}^{\ast }}\hat{W}_{i}^{r}(s) &=&\inf_{0\leq
s\leq w_{i}/6v_{i}^{\ast }}\left( w_{i}+\sum_{j=1}^{J}
K_{i,j}\frac{A_{j}^{r}(r^{2}s)}{ r\mu _{j}^{r}}-\sum_{j=1}^{J}K_{i,j}
\frac{S_{j}^{r}(r^{2}\bar{B}_{j}^{r}(s))}{r\mu _{j}^{r}}\right)  \\
&\geq &\inf_{0\leq s\leq w_{i}/6v_{i}^{\ast }}\left( w_{i}-r\left(
C_{i}-\sum_{j=1}^{J}K_{i,j}\varrho _{j}^{r}\right) s-\sum_{j=1}^{J}K_{i,j}\frac{%
\min \{1,\bar{\mu}_{\min}\}w_{i}}{12J\mu _{j}^{r}}\right)  \\
&\geq &\inf_{0\leq s\leq w_{i}/6v_{i}^{\ast }}\left( \frac{5w_{i}}{6}%
-2v_{i}^{\ast }s\right)  \\
&\geq & \frac{w_{i}}{2} > 2c_{3}
\end{eqnarray*}%
where the third line uses \eqref{def:Rhat} and \eqref{eq:larger1155}.
Thus $\{\check{\tau}_{i,0}^{r} \leq w_{i}/6v_{i}^{\ast }\}\cap (\mathcal{B}%
_{i,w_{i}/6v_{i}^{\ast }}^{r})^c = \emptyset$, $P_x$ a.s.. \ This gives 
\begin{eqnarray*}
E_x\left[ e^{\beta \check{\tau}_{i,0}^{r}}\right]  &=&E_x\left[ e^{\beta 
\check{\tau}_{i,0}^{r}}\mathit{I}_{\mathcal{B}_{i,w_{i}/6v_{i}^{\ast
}}^{r}}\right] +E_x\left[ e^{\beta \check{\tau}_{i,0}^{r}}\mathit{I}_{(%
\mathcal{B}_{i,w_{i}/6v_{i}^{\ast }}^{r})^{c}}\right]  \\
&\geq &e^{\beta \left( w_{i}/(6v_{i}^{\ast })\right) }P_x\left( \mathcal{B}%
_{i,w_{i}/6v_{i}^{\ast }}^{r}\right)^c
\geq \frac{1}{2} e^{(\beta /(6v_{i}^{\ast }))w_{i}}\text{,}
\end{eqnarray*}%
where the last inequality is from \eqref{eq:eq852}.
Thus completes the proof.
\end{proof}

Recall   $\delta^*$ from Proposition \ref{thm:stopTimeExpMomBnd}
and fix $\beta \in (0, \delta^*)$. For
 $i\in \mathbb{N}_{I}$ let  
\begin{equation*}
V_{i}(x)\doteq E_x\left[ e^{\beta \check{\tau}_{i,0}^{r}}\right] \text{.}
\end{equation*}
Also recall the Markov process $\hat X^r$ in \eqref{eq:eqhatxrt}. The following result proves a Lyapunov function property for $V_i$.
\begin{proposition}
\label{thm:timeDecayV} There exist   $H,R \in (0,\infty)$ such that for all $x=(q,z)\in \cls^r$, $r\geq R$, $i \in \mathbb{N}_I$, and $t\in \lbrack 0,1]$ we have  
\begin{equation*}
E_x\left[  V_{i}(\hat{X}^{r}(t))\right] \leq
e^{-\beta t} V_{i}(x)+H\text{.}
\end{equation*}
\end{proposition}

\begin{proof}
	From the Markov property we have
	\begin{eqnarray*}
	E_x\left[  V_{i}(\hat{X}^{r}(t))\right]  =E_x\left[ e^{\beta \left( \check{\tau}_{i,t}^{r}-t\right) }\right]  
	=E_x\left[ e^{\beta \left( \check{\tau}_{i,t}^{r}-t\right) }\mathit{I}%
	_{\left\{ \check{\tau}_{i,0}^{r}\geq t\right\} }\right] +E_x\left[
	e^{\beta \left( \check{\tau}_{i,t}^{r}-t\right) }\mathit{I}_{\left\{ 
	\check{\tau}_{i,0}^{r}<t\right\} }\right] \text{.}
	\end{eqnarray*}%
	Let $R$ be as in Proposition \ref{thm:stopTimeExpMomBnd}. \ Let $t\in \lbrack
0,1]$ and $r\geq R$ be arbitrary. \ Then from  Proposition \ref{thm:stopTimeExpMomBnd}, for some $\hat H_1, \hat H_2 \in (0,\infty)$
\begin{eqnarray*}
E_x\left[ e^{\beta \left( \check{\tau}_{i,t}^{r}-t\right) }\mathit{I}%
_{\left\{ \check{\tau}_{i,0}^{r}<t\right\} }\right]  \leq
\sup_{x':w_{i}\leq 2c_{3}}\sup_{0\le \xi \le 1}E_{x'}\left[ e^{\beta \check{\tau}_{i,\xi}^{r}}\right]
\leq \hat{H}_{1}e^{\hat{H}_{2}(2c_{3}+1)}
\end{eqnarray*}%
Furthermore,
\begin{eqnarray*}
E_x\left[ e^{\beta \left( \check{\tau}_{i,t}^{r}-t\right) }\mathit{I}%
_{\left\{ \check{\tau}_{i,0}^{r}\geq t\right\} }\right]  
= e^{-t\beta }E_x\left[ e^{\beta \check{\tau}_{i,0}^{r}}\mathit{I}%
_{\left\{ \check{\tau}_{i,0}^{r}\geq t\right\} }\right]  
\leq e^{-t\beta }E_x\left[ e^{\beta \check{\tau}_{i,0}^{r}}\right]  
=e^{-t\beta }V_{i}(x)\text{.}
\end{eqnarray*}%
Combining the two estimates we have the result.
\end{proof}
From the  Lyapunov function property proved in the previous result we have the following moment estimate for all time instants.
\begin{proposition}
\label{thm:timeIndVbnd} There exist  $%
H_{1},H_{2},R \in (0,\infty) $  such that for all $t\geq 0$, $i \in \mathbb{N}_I$ and $%
r\geq R$ we have 
\begin{equation*}
E_x\left[  V_{i}(\hat{X}^{r}(t))\right] \leq
H_{1}e^{-\beta t}V_{i}(x)+H_{2}
\end{equation*}
\end{proposition}

\begin{proof}
Let $R, H$ be as in Proposition \ref{thm:timeDecayV}. \ Then for all $i \in \mathbb{N}_I$, $x=(q,z) \in \cls^r$, $t \in [0,1]$ and $r\ge R$, we have
\begin{equation*}
E_x\left[  V_{i}(\hat{X}^{r}(t))\right] \leq
e^{-\beta t}V_{i}(x)+ {H}\text{.}
\end{equation*}
Then from the Markov property, for any $r\geq R$ and $t\ge 0$ we have 
\begin{eqnarray}
E_x\left[  V_{i}(\hat{X}^{r}(t))\right]  =
E_x\left[  E_x\left[ \left. V_{i}(\hat{X}^{r}(t))\right\vert \hat{X}%
^{r}(\left\lfloor t\right\rfloor )\right] \right] 
\leq  {H}+ e^{-\beta (t-\left\lfloor t\right\rfloor)} E_x\left[  V_{i}(\hat{X}^{r}(\left\lfloor
t\right\rfloor ))\right].   \label{eq:VgenTimeFromIntTime} 
\end{eqnarray}%
Using the Markov property again
\begin{eqnarray*}
E_x\left[  V_{i}(\hat{X}^{r}(\left\lfloor t\right\rfloor ))\right]  =E_x\left[  E\left[  V_{i}(\hat{X}%
^{r}(1))\mid {	\hat{X}^{r}(\left\lfloor t\right\rfloor -1)}\right] \right]  
\leq {H}+e^{-\beta}E_x\left[  V_{i}(\hat{X}^{r}(\left\lfloor
t\right\rfloor -1))\right].
\end{eqnarray*}%
Iterating the above inequality we get 
\begin{eqnarray*}
E_x\left[  V_{i}(\hat{X}^{r}(\left\lfloor t\right\rfloor ))\right]  \leq 
e^{-\beta \left\lfloor
t\right\rfloor} V_{i}(x)+ {H}\sum_{k=0}^{\left\lfloor t\right\rfloor
-1}\left( e^{-\beta}\right) ^{k} \leq 
e^{-\beta \left\lfloor t\right\rfloor }V_{i}(x)+\frac{%
{H}}{1-e^{-\beta}}\text{.}
\end{eqnarray*}%
Combining this with (\ref{eq:VgenTimeFromIntTime}) we have for all $t\ge 0$
\begin{equation*}
E_x\left[  V_{i}(\hat{X}^{r}(t))\right] \leq
e^{-\beta t} V_{i}(x)+ {H}\left(1+\frac{1}{1-e^{-\beta}}\right) .
\end{equation*}%
The result follows.
\end{proof}

{\bf Proof of Theorem \ref{thm:WmomBnd}.}
This proof is immediate from Proposition \ref{thm:VlowerBnd} and Proposition \ref{thm:timeIndVbnd} on taking $\gamma = H_3 \beta $ where $H_3$ is as in the statement of Proposition \ref{thm:VlowerBnd}
and $\beta$ is as fixed above Proposition \ref{thm:timeDecayV}.

\section{\protect Path Occupation Measure Convergence}
\label{sec:pathoccmzr}
Let for $t\ge 0$
\begin{equation}
\hat{Z}^{r}(t)=w^{r}+G^{r}(\hat{A}^{r}(t)-\hat{S}^{r}(\bar{B}^{r}(t))),\label{eq:zhrt}
\end{equation}%
where $w^r = G^rq$.
Consider the  collection of 
 random variables indexed by $T$ and $r$ taking values in  $\mathcal{P}\left(
D([0,1]:\RR_+^{I}\times \mathbb{R}^{I})\right) $, defined by 
\begin{equation*}
\theta _{T}^{r}(dx\times dy)=\frac{1}{T}\int_{0}^{T}\delta _{\hat{W}%
^{r}(t+\cdot )}(dx)\delta _{\hat{Z}^{r}(t+\cdot )-\hat{Z}^{r}(t)}(dy)dt .
\end{equation*}
In this section we will prove the tightness of the collection $\{\theta _{T}^{r}, T>0, r>0\}$ of random path occupation measures and characterize limit points along suitable subsequences.

We begin by noting the following monotonicity property of a one dimensional Skorohod map introduced in Section \ref{sec:hgi}.
\begin{theorem}
\label{thm:skorokIneq} Fix $T \in (0,\infty)$ and $f\in D([0,T]:\mathbb{R}%
)$ satisfying $f(0)=0$. Let $\varphi _{1}=\Gamma_1(f)$.
Suppose $\varphi_{2}, \varphi_{3} \in D([0,T]:\mathbb{R})$ are such that
\begin{itemize}
	\item $\varphi_2(t) = f(t) + h_2(t)$, $t\in [0,T]$, where $h_2 \in D([0,T]:\mathbb{R})$ is a nondecreasing function with $h_2(0)=0$ and $\int_{[0,T]} 1_{(0,\infty)} (\varphi_2(s)) dh_2(s) =0$.
	\item $\varphi_3(t) = f(t) + h_3(t)$, $t\in [0,T]$, where $h_3 \in D([0,T]:\mathbb{R})$ is a nondecreasing function with $h_3(0)=0$ and $\varphi_3(t)\ge 0$ for all $t \in [0,T]$.
\end{itemize}
Then for all $t\in [0,T]$,
$\varphi_2(t) \le \varphi_1(t) \le \varphi_3(t)$.
\end{theorem}

\begin{proof}
The proof of the second inequality is straightforward and is omitted. Consider now the first inequality.
 Note that $\varphi _{1}(t)=f(t)+h_{1}(t)$
where $h_{1}(t)=-\inf_{0\leq s\leq t}\{f(s)\}$ and thus it suffices to show that  for any $t\in \lbrack 0,T]$, $%
h_{2}(t)\leq -\inf_{0\leq s\leq t}\{f(s)\}$. \ Assume that there exists $%
t_{2}^{\ast }\in \lbrack 0,T]$ such that $h_{2}(t_{2}^{\ast })>-\inf_{0\leq
s\leq t_{2}^{\ast }}\{f(s)\} \doteq a$. \ Let 
\begin{equation*}
t_{1}^{\ast }=\sup \{s\in \lbrack 0,t_{2}^{\ast }]:h_{2}(s)\leq a\}
\end{equation*}%
and note that either $h_{2}(t_{1}^{\ast })>a$ or $h_{2}(t_{1}^{\ast })=a$ and $h_{2}(r)>a$ for all 
$r\in (t_{1}^{\ast },t_{2}^{\ast }]$. \ In the first case 
\begin{equation*}
\varphi _{2}(t_{1}^{\ast })=f(t_{1}^{\ast })+h_{2}(t_{1}^{\ast
})>f(t_{1}^{\ast })-\inf_{0\leq s\leq t_{2}^{\ast }}\{f(s)\}\geq 0
\end{equation*}%
so $\varphi _{2}(t_{1}^{\ast })>0$ and 
\begin{equation*}
\int_{\{t_{1}^{\ast }\}}dh_{2}(s)=h_2(t_{1}^{\ast })-\lim_{s\uparrow
t_{1}^{\ast }}h_2(s)>0
\end{equation*}%
which is a contradiction. \ In the second case for all $r\in (t_{1}^{\ast
},t_{2}^{\ast }]$ 
\begin{equation*}
\varphi _{2}(r)=f(r)+h_{2}(r)> f(r)-a \ge
f(r)-\inf_{0\leq s\leq t_2^*}\{f(s)\}\geq 0
\end{equation*}%
so $\varphi _{2}(r)>0$ for all $r\in (t_{1}^{\ast },t_{2}^{\ast }]$ and 
\begin{equation*}
\int_{(t_{1}^{\ast },t_{2}^{\ast }]}dh_{2}(s)=h_2(t_{2}^{\ast })-h_2(t_{1}^{\ast
})=h_{2}(t_{2}^{\ast })-a >0
\end{equation*}%
which is also a contradiction. \ Therefore for any $t\in \lbrack 0,T]$ we
have $h_{2}(t)\leq -\inf_{0\leq s\leq t}\{f(s)\}$ and the desired inequality follows.

\end{proof}

\begin{theorem}
\label{thm:finTimeConvToRBM}For any $\epsilon >0$ and $T \in (0, \infty)$ there
exists $R \in (0,\infty)$ such that for all $r\geq R$ and $x=(q,z) \in \cls^r$,
\begin{equation*}
\sup_{s \in [0,\infty)}P_x\left( \sup_{0\leq t\leq T}\left\vert \Gamma \left( \hat W^r(s)+\hat{Z}^{r}(s+\cdot) - \hat{Z}^{r}(s) +r(K\rho
^{r}-C)\iota\right)(t) -\hat{W}^{r}(t+s)\right\vert >\epsilon \right) <\epsilon \text{.
}
\end{equation*}

\end{theorem}

\begin{proof}
We will only prove the result without the outside supremum and in fact only when $s=0$. The general case follows on using the Markov property and the fact that the estimate in \eqref{eq:squareIdleTimeResultImplication}
is uniform over all $x \in \Gamma^r$.
 Let
\begin{equation*}
\hat{\xi}^{r}_i(t)=\frac{1}{r}\int_{0}^{r^{2}t}\mathit{I}_{\left\{ {W}%
_{i}^{r}(s)\ge c_{3}r^{\alpha }\right\} }(s)dI^{r}_i(s), \; i \in \NN_{I}.
\end{equation*}%
Note that 
\begin{equation*}
\hat{W}^{r}_i(t)-c_{3}r^{\alpha -1}=\hat{Z}^{r}_i(t)+tr(K\rho ^{r}-C)_i+\hat{\xi}%
^{r}_i(t)-c_{3}r^{\alpha -1}+ \int_{0}^{t}\mathit{I}_{\left\{ 
\hat{W}_{i}^{r}(s)-c_{3}r^{\alpha -1}< 0\right\} }(s)d\hat I^{r}_i(s)
\end{equation*}%
and consequently due to Theorem \ref{thm:skorokIneq}\ we have 
\begin{equation*}
\hat{W}^{r}(t)-c_{3}r^{\alpha -1}\leq \Gamma \left( \hat{Z}^{r} + r(K\rho
^{r}-C)\iota +\hat{\xi}^{r}-c_{3}r^{\alpha -1}\right)(t), \; t \ge 0 \text{.}
\end{equation*}%
In addition, 
\begin{equation*}
\hat{W}^{r}(t)=\hat{Z}^{r}(t)+tr(K\rho ^{r}-C)+\hat{I}^{r}(t)
\end{equation*}%
is a nonnegative function and $\hat{I}^{r}(0)$ is nondecreasing and
satisfies $\hat{I}^{r}(0)=0$. Thus once more from  Theorem \ref{thm:skorokIneq} 
\begin{equation*}
\Gamma \left( \hat{Z}^{r}(\cdot)+ r(K\rho ^{r}-C)\iota\right) \leq \hat{W}^{r}(t)%
\text{,}\; t\ge 0.
\end{equation*}%
Combining this gives for all $t\ge 0$
\begin{equation}
\Gamma \left( \hat{Z}^{r} +r(K\rho ^{r}-C)\iota\right)(t) \leq \hat{W}^{r}(t)\leq
\Gamma \left( \hat{Z}^{r} +r(K\rho ^{r}-C)\iota+\hat{\xi}^{r}(\cdot)-c_{3}r^{%
\alpha -1}\right) +c_{3}r^{\alpha -1}.  \label{eq:skorokIneq}
\end{equation}%
Lipschitz property of the Skorokhod map gives that there is a $\kappa_1 \in (0,\infty)$ such that for all $T>0$
\begin{equation*}
\sup_{0\leq t\leq T}\left\vert 
\Gamma \left( \hat{Z}^{r}(\cdot) +r(K\rho
^{r}-C)\iota\right)(t) -\hat{W}^{r}(t)
\right\vert \leq \kappa_1\left( 2c_{3}r^{\alpha -1}+\left\vert \hat{\xi}^{r}(T)\right\vert \right).
\end{equation*}%
From Theorem \ref{thm:IdleTimeExp} (see \eqref{eq:squareIdleTimeResultImplication}), for
any $\epsilon >$ $0$ and $T \in (0,\infty)$, there exists $R \in (0,\infty)$ such that for all $r\geq R$ and $x \in \cls^r$
\begin{equation*}
P_x\left( \left\vert \hat{\xi}^{r}(T)\right\vert >\epsilon
\right) <\epsilon .
\end{equation*}%
The result follows.
\end{proof}

Recall the initial condition $q^r$ introduced in \eqref{eq:queleneqn}.  
\begin{theorem}
\label{thm:occMeasTight} Suppose $\hat q^r \doteq q^r/r$ satisfies $\sup_{r>0} \hat q^r <\infty$.
Let $\{t_r\}$ be an increasing sequence
such that $t_r\uparrow\infty$ as $r\to \infty$. Suppose that $\hat w^r$ converges to some $w \in \RR_+^I$. Then, the random variables $%
\{\theta _{t_r}^{r}, r>0\}$ are tight in the space $\mathcal{P}%
\left( D([0,1]:\RR_+^{I}\times\mathbb{R}^{I})\right) $.
\end{theorem}

\begin{proof}
It suffices to show that the collection
$$\left\{\left (\hat W^r(t+\cdot), \hat Z^r(t+\cdot)- \hat Z^r(t)\right),\; r> 0,\; t>0\right\}$$
is tight in $D([0,1]: \RR_+^I \times \RR^I)$.

Let
\begin{equation*}
\mathcal{F}_{t}^{r}=\sigma \left( \hat{S}_{j}^{r}(\bar{B}^{r}(s)),\hat{A}%
_{j}^{r}(s):j\in \mathbb{N}_{J},0\leq s\leq t\right), \; t\ge 0.
\end{equation*}%
and note that for all $j\in \mathbb{N}$ both $\hat{S}_{j}^{r}(\bar{B}%
^{r}(t)) $ and $\hat{A}_{j}^{r}(t)$ are $\mathcal{F}_{t}^{r}$-martingales. \
Consequently, there are $\kappa_1, \kappa_2 \in (0,\infty)$ such that for any $r>0$, $\delta>0$ and  $\mathcal{F}_{t}^{r}$-stopping times $%
\tau _{1},\tau _{2}$ satisfying $\tau _{1}\leq \tau _{2}\leq \tau
_{1}+\delta \le 1 $, 
\begin{eqnarray*}
&&E\left[ \left( \hat{Z}_{i}^{r}(\tau _{2})-\hat{Z}_{i}^{r}(\tau
_{1})\right) ^{2}\right] \\
&\leq &\kappa_1 \sum_{j=1}^{J}G_{i,j}^{r}\left( E\left[ (\hat{A}_{j}^{r}(\tau _{2})-%
\hat{A}_{j}^{r}(\tau _{1}))^{2}\right] +E\left[ (\hat{S}_{j}^{r}(\bar{B}%
^{r}(\tau _{2}))-\hat{S}_{j}^{r}(\bar{B}^{r}(\tau _{1})))^{2}\right] \right)
\\
&\leq & \kappa_1 \sum_{j=1}^{J} E\left[ \tau _{2}-\tau _{1}%
\right] +\sum_{j=1}^{J} E\left[ \bar{B}^{r}(\tau _{2})-%
\bar{B}^{r}(\tau _{1})\right] \\
&\leq & \kappa_2\delta .
\end{eqnarray*}%
This proves the tightness of the collection $\{\hat Z^r(t+\cdot)- \hat Z^r(t),\; r> 0,\; t>0\}$.

From the convergence $r(K\varrho^r-C) \to v^*$, Theorem \ref{thm:finTimeConvToRBM}, and Lipschitz property of the Skorohod map, to prove the tightness of 
$\{\hat W^r(t+\cdot),\; r> 0,\; t>0\}$ it now suffices to prove the tightness of $\{\hat W^r(t),\; r> 0,\; t>0\}$. However that is an immediate consequence of Propositions \ref{thm:WmomBnd} and \ref{thm:stopTimeExpMomBnd}.
The result follows.

\end{proof}
Recall that the reflected Brownian motion $\{\ch W^{w_0}\}_{w_0 \in \RR_+^I}$ in \eqref{eq:eqrbm}  has a unique invariant probability distribution which we denote as $\pi$.
We will denote by $\Pi$ the unique  measure on 
$C([0,1]:\RR_+^I)$ associated with this Markov process with initial distribution $\pi$.
 The following theorem gives a characterization of the weak limit points of the sequence $\theta _{t_r}^{r}$ in Theorem \ref{thm:occMeasTight}.
We denote the canonical coordinate processes on $D([0,1]:\RR_+^I \times \RR^I)$ as $(\bw(t), \bz(t))_{0\le t\le 1}$.
Let $\clg_t \doteq \sigma\{(\bw(s), \bz(s)): 0 \le s \le t\}$ be the canonical filtration on this space.
\begin{theorem}
\label{thm:limitMeasProp} Suppose $\hat q^r \doteq q^r/r$ satisfies $\sup_{r>0} \hat q^r <\infty$.
Also suppose that $\theta _{t_r}^{r}$ converges in distribution, along some subsequence as $r\to \infty$, to a  $\clp(D([0,1]:\RR_+^I \times \RR^I))$ valued random variable  $\theta$ given on some probability space
$(\bar \Om, \bar \clf, \bar P)$. Then for $\bar P$ a.e. $\omega$, under $\theta(\om)\equiv \theta_{\om}$ the following hold.
\begin{enumerate}
	\item $\theta_{\om}(C([0,1]:\RR_+^I\times \RR^I))=1$.
	\item $\{\bz(t)\}_{0\le t\le 1}$ is a $\clg_t$-Brownian motion with covariance matrix $\Sigma = \Lambda \Lambda'$, where $\Lambda$ is as introduced above \eqref{eq:eqrbm}.
	\item $\{(\bw(t),\bz(t))\}_{0\le t\le 1}$ satisfy $\theta_{\om}$ a.s.
	$$\bw(t) = \Gamma(\bw(0)- v^*\iota +\bz)(t),\; 0\le t \le 1.$$
	\item $\theta_{\om} \circ (\bw(0))^{-1} = \pi$ and thus denoting the first marginal of $\theta_{\om}$ on 
	$C([0,1]:\RR_+^I)$ as $\theta_{\om}^1$, we have $\theta_{\om}^1=\Pi$.
\end{enumerate}
\end{theorem}

\begin{proof}
	For notational simplicity we denote the convergent subsequence of $\theta _{t_r}^{r}$  by the same symbol.
For $(x,y)\in D([0,1]:\mathbb{R}_{+}^{I}\times \mathbb{R}^{I})$
define 
$
j(x,y)=\sup_{0\leq t<1} \left\Vert (x(t),y(t))-(x(t-),y(t-))\right\Vert.
$
Then there is a $\kappa_1\in (0,\infty)$ such that for all $r$,
$E \theta^r_{t_r} ((x,y): j(x,y) > \kappa_1/r) =0.$
Thus in particular, for every $\delta \in (0,\infty)$, as $r\to \infty$,
$E \theta^r_{t_r} ((x,y): j(x,y) > \delta) \to 0.$
By weak convergence of $\theta^r_{t_r}$ to $\theta$ and Fatou's lemma we then have 
$E \theta ((x,y): j(x,y) > \delta) = 0$
which proves part (1) of the theorem.

In what follows, we will denote the expected value under
$\theta^r_{t_r}$ (resp. $\theta$) as $E_{\theta^r_{t_r}}$ (resp. $E_{\theta}$).
Let $f: D([0,1]:\mathbb{R}_{+}^{I}\times \mathbb{R}^{I}) \to \RR$ be a continuous and bounded function.
We now argue that for all $0\le s<t\le 1$, and $i \in \NN_I$
\begin{equation}
	\bar E \left(\left |E_{\theta}\left(f(\bw(\cdot \wedge s), \bz(\cdot \wedge s)) (\bz_i(t)-\bz_i(s))\right)\right|
	\wedge 1\right) =0. \label{eq:eqismart}
\end{equation}
This will prove that $\{\bz(t)\}_{0\le t\le 1}$ is a $\clg_t$-martingale under $\theta_{\om}$ for a.e. $\om$.
To see \eqref{eq:eqismart} note that
\begin{align*}
	 &E E_{\theta^r_{t_r}}\left[f(\bw(\cdot \wedge s), \bz(\cdot \wedge s)) (\bz_i(t)-\bz_i(s))\right]^2\\
	&= E\left[\frac{1}{t_r}\int_0^{t_r} f(\hat W^r(u+(\cdot \wedge s)), \hat Z^r(u+(\cdot \wedge s))- \hat Z^r(u))
	[\hat Z^r_i(u+t) - \hat Z^r_i(u+s)] du\right]^2\\
	&= \frac{2}{t_r^2}\int_0^{t_r}\int_0^u E(H_i(u)H_i(v)) dv du,
\end{align*}
where for $u\ge 0$
\begin{equation*}
H_i(u)=f (\hat{W}^{r}(u+(\cdot \wedge s)),\hat{Z}^{r}(u+(\cdot \wedge s))-\hat{Z}^{r}(u))(\hat{Z}_i^{r}(u+t)-
\hat{Z}_i^{r}(u+s))\text{.}
\end{equation*}%
Since $\hat Z^r_i$ is a martingale, we have for $v<u-1$, $E(H_i(u)H_i(v))=0$.
Also from properties of Poisson processes it follows that for every $p\ge 1$
\begin{equation}
	\sup_{r>0, u\ge 0, s,t\in [0,1]} E\left\|\hat Z^r(u+t) - \hat Z^r(u+s)\right\|^p \doteq m_p<\infty.
\label{eq:eqlpzrin}	
\end{equation}
Thus since $f$ is bounded , we have for some $\kappa_2 \in (0,\infty)$
$$\frac{2}{t_r^2}\int_0^{t_r}\int_0^u E(H_i(u)H_i(v)) dv du \le \frac{\kappa_2}{t_r}\to 0$$
as $r\to \infty$.
Thus as $r\to \infty$
\begin{equation*}
	\bar E \left(\left |E_{\theta^r_{t_r}}\left(f(\bw(\cdot \wedge s), \bz(\cdot \wedge s)) (\bz_i(t)-\bz_i(s))\right)\right|
	\wedge 1\right)\to 0. 
\end{equation*}
The equality in \eqref{eq:eqismart} now follows on noting that from \eqref{eq:eqlpzrin}, for all $t\in [0,1]$,
$\sup_{r>0} E E _{\theta^r_{t_r}}(\bz_i(t))^2 < \infty .$
In order to argue that $\{\bz(t)\}_{0\le t\le 1}$ is a $\clg_t$-Brownian motion with covariance matrix $\Sigma$ it now suffices to show that defining
$\bm(t) \doteq \bz(t)\bz'(t) - t \Sigma$, $\{\bm(t)\}_{0\le t\le 1}$ is a $I^2$ dimensional $\{\clg_t\}$-martingale.
Once more, it suffices to show that with $f$ as before, $0\le s<t\le 1$, and $i,l \in \NN_I$,
\begin{equation}
	\bar E \left(\left |E_{\theta}\left(f(\bw(\cdot \wedge s), \bz(\cdot \wedge s)) (\bm_{i,l}(t)-\bm_{i,l}(s))\right)\right|
	\wedge 1\right) =0. \label{eq:eqismartqv}
\end{equation}
For this note that
\begin{align}
&\bar E  E_{\theta^r_{t_r}}\left[f(\bw(\cdot \wedge s), \bz(\cdot \wedge s)) (\bm_{i,l}(t)-\bm_{i,l}(s))\right]^2	\nonumber\\
&=E\left[\frac{1}{t_r}\int_0^{t_r} f(\hat W^r(u+(\cdot \wedge s)), \hat Z^r(u+(\cdot \wedge s))- \hat Z^r(u))
[\hat M^{r,u}_{i,l}(t) - \hat M^{r,u}_{i,l}(s)]du\right]^2\nonumber\\
&= \frac{2}{t_r^2}\int_0^{t_r}\int_0^u E(H^r_{i,l}(u)H^r_{i,l}(v)) dv du,\label{eq:a11052}
\end{align}
where for $u\ge 0$
$$\hat M^{r,u}_{i,l}(t)=\left( \hat{Z}^{r}_{i}(u+t)-\hat{Z}^{r}_i(u)\right) \left( \hat{Z}%
^{r}_l(u+t)-\hat{Z}^{r}_l(u)\right) -t\Sigma_{il} 
$$
and 
$$
H^r_{i,l}(u) = 
f(\hat W^r(u+(\cdot \wedge s)), \hat Z^r(u+(\cdot \wedge s))- \hat Z^r(u))[\hat M^{r,u}_{i,l}(t) - \hat M^{r,u}_{i,l}(s)].
$$
Write
$$\hat M^{r,u}_{i,l}(t)- \hat M^{r,u}_{i,l}(s) = \hat{\Psi}_{i,l}^{r}(u) + \bar{\xi}_{i,l}^{r}(u),$$
where
\begin{eqnarray*}
\hat{\Psi}_{i,l}^{r}(u) &=&(\hat{Z}_{i}^{r}(u+t)-\hat{Z}_{i}^{r}(u+s))(\hat{Z%
}_{l}^{r}(u+t)-\hat{Z}_{l}^{r}(u+s)) \\
&&-\sum\limits_{j=1}^{J}G_{i,j}^{r}K_{l,j}(\bar{B}_{j}^{r}(u+t)-\bar{B}%
_{j}^{r}(u+s)+(t-s)\varrho _{j}^{r})
\end{eqnarray*}%
and%
\begin{equation}\label{eq:defxiil}
\bar{\xi}_{i,l}^{r}(u)=\sum\limits_{j=1}^{J}G_{i,j}^{r}K_{l,j}(\bar{B}%
_{j}^{r}(u+t)-\bar{B}_{j}^{r}(u+s)+(t-s)\varrho _{j}^{r})-(t-s)\Sigma_{i,l}%
\text{.}
\end{equation}%
Then for $0\le v\le u\le t_r$
\begin{align}
	|E(H^r_{i,l}(u)H^r_{i,l}(v))| &\le |E(\hat H^r_{i,l}(u)\hat H^r_{i,l}(v))|+ \|f\|_{\infty}^2\sup_{u\ge 0}E(\bar{\xi}_{i,l}^{r}(u))^2\nonumber\\
	&\quad+ 2\|f\|_{\infty}^2 \left[\sup_{u\ge 0} E(\hat{\Psi}_{i,l}^{r}(u))^2\right]^{1/2}
	\left[\sup_{u\ge 0} E(\bar{\xi}_{i,l}^{r}(u))^2\right]^{1/2}, \label{eq:eq1036}
\end{align}
where
$$
\hat H^r_{i,l}(u) = 
f(\hat W^r(u+(\cdot \wedge s)), \hat Z^r(u+(\cdot \wedge s))- \hat Z^r(u))\hat{\Psi}_{i,l}^{r}(u).
$$
From \eqref{eq:eqlpzrin}, for some $\kappa_3\in (0,\infty)$	
$$\sup_{r>0, u,v>0} E|\hat H^r_{i,l}(u)\hat H^r_{i,l}(v)| \le \kappa_3.$$
Also, from martingale properties of $\hat A_j$ and $\hat S_j$ we see that for $v<u-1$,
$$E(\hat H^r_{i,l}(u)\hat H^r_{i,l}(v))=0.$$
Combining the above two displays we now have that as $r\to \infty$
\begin{equation}\label{eq:qvdiag}
\frac{2}{t_r^2}\int_0^{t_r}\int_0^u |E(\hat H^r_{i,l}(u)\hat H^r_{i,l}(v))| dv du \le \frac{\kappa_4}{t_r} \to 0.
\end{equation}
From \eqref{eq:eqlpzrin} once more, we have for some $\kappa_5 \in (0,\infty)$
\begin{equation}
	\sup_{u\ge 0, r>0} E(\hat{\Psi}_{i,l}^{r}(u))^2 \le \kappa_5.\label{eq:bdpsiil}
\end{equation}
We now argue that
\begin{equation}
	\sup_{u\ge 0} E(\bar{\xi}_{i,l}^{r}(u))^2 \to 0 \mbox{ as } r\to \infty.\label{eq:bdpsiilxi}
\end{equation}
Note that once \eqref{eq:bdpsiilxi} is proved, it follows on combining \eqref{eq:a11052}, \eqref{eq:eq1036},
\eqref{eq:qvdiag} and \eqref{eq:bdpsiilxi} that 
$$E  E_{\theta^r_{t_r}}\left[f(\bw(\cdot \wedge s), \bz(\cdot \wedge s)) (\bm_{i,l}(t)-\bm_{i,l}(s))\right]^2\to 0$$
as $r\to \infty$. Once more using the moment bound in \eqref{eq:eqlpzrin} we then have
\eqref{eq:eqismartqv} completing the proof of (2). We now return to the proof of \eqref{eq:bdpsiilxi}.
We note that for some $\kappa_6 \in (0,\infty)$
$$\sup_{u,r>0} |\bar{\xi}_{i,l}^{r}(u)| \le \kappa_6 \mbox{ a.s. }.$$
Thus for any $\eps \in (0,\infty)$
\begin{equation}\label{eq:eps2118}
	\sup_{u>0} E|\bar{\xi}_{i,l}^{r}(u)|^2 \le \eps^2 + \kappa_6^2 \sup_{u>0}P(|\bar{\xi}_{i,l}^{r}(u)|>\eps).
\end{equation}
Next from properties of Poisson processes it follows that for any $\tilde \eps \in (0,\infty)$, as $r\to \infty$
\begin{equation*}
\sup_{u\ge 0}P\left( \left\vert \bar{A}_{j}^{r}(u+t)-\bar{A}_{j}^{r}(u+s)-(t-s)\lambda_{j}^r\right\vert >\tilde\epsilon \right) \to 0
\end{equation*}%
and 
\begin{equation*}
\sup_{u\ge 0}P\left( \left\vert \bar{S}_{j}^{r}(\bar{B}_{j}^{r}(u+t))-\bar{S}_{j}^{r}(%
\bar{B}_{j}^{r}(u+t))-(\bar{B}_{j}^{r}(u+t)-\bar{B}_{j}^{r}(u+s))\mu_{j}^r\right\vert >\tilde \epsilon \right) \to 0 \text{.}
\end{equation*}%
Also, using Theorem  \ref{thm:WmomBnd}, as $r\to \infty$
\begin{align*}
&\sup_{u\ge 0}P\left( \left\vert \bar{A}_{j}^{r}(u+t)-\bar{A}_{j}^{r}(u+s)-\left( \bar{S}%
_{j}^{r}(\bar{B}_{j}^{r}(u+t))-\bar{S}_{j}^{r}(\bar{B}_{j}^{r}(u+t))\right)
\right\vert >\tilde \epsilon \right)  \\
&=\sup_{u\ge 0}P\left( \left\vert \bar{Q}_{j}^{r}(u+t)-\bar{Q}_{j}^{r}(u+s)\right\vert
>\tilde \epsilon \right)  \to 0. 
\end{align*}%
Combining the above three convergence properties we have that as $r\to \infty$
\begin{equation}
	\sup_{u\ge 0} P\left( \left\vert (t-s)\lambda_{j}^r-(\bar{B}_{j}^{r}(u+t)-\bar{B}%
	_{j}^{r}(u+s))\mu _{j}^r\right\vert >\tilde \epsilon \right) \to 0 \text{.}
	\label{eq:eeq1113}
 \end{equation}
Recalling the definition of $\bar{\xi}_{i,l}^{r}(u)$ from \eqref{eq:defxiil} and noting that
$2\sum_{j=1}^J G_{ij}K_{l,j}\varrho_j = \Sigma_{il}$, we see from \eqref{eq:eeq1113} that for any $\eps \in (0,\infty)$
$$\sup_{u>0}P(|\bar{\xi}_{i,l}^{r}(u)|>\eps) \to 0$$
as $r\to \infty$.
Using this in \eqref{eq:eps2118} and sending $\eps\to 0$ we have \eqref{eq:bdpsiilxi}. As noted earlier this completes the proof of (1).

We now prove (3). From Theorem \ref{thm:finTimeConvToRBM} and since $r(K\varrho^r-C)\to v^*$ as $r\to \infty$, we have for every $t \in [0,1]$, as $r\to \infty$
\begin{align*}
&	E E_{\theta^r_{t_r}}\left[\left\|\bw(t) - \Gamma(\bw(0) + \bz -v^*\iota)(t)\right\|\wedge 1\right]\\
	& = \frac{1}{t_r}\int_0^{t_r} E\left[\left\|\hat W^r(u+t) - \Gamma(\hat W^r(u) + \hat Z^r(u+\cdot)-\hat Z^r(u)- v^*\iota)(t)\right\|\wedge 1\right] du \to 0.
\end{align*}
Since $\theta^r_{t_r}\to \theta$ in distribution, we have from continuous mapping theorem
$$E E_{\theta}\left[\left\|\bw(t) - \Gamma(\bw(0) + \bz +v^*\iota)(t)\right\|\wedge 1\right]=0.$$
This proves (3).

Finally in order to prove (4) it suffices to show that for every continuous and bounded $g:\RR_+^I \to \RR$
and $t\in [0,1]$
\begin{equation}
	\label{eq:geq1129}
	E\left| E_{\theta}(g(\bw(t))) - E_{\theta}(g(\bw(0)))\right| =0
\end{equation}
Note that as $r\to \infty$
\begin{align*}
	&E\left| E_{\theta^r_{t_r}}(g(\bw(t))) - E_{\theta}(g(\bw(0)))\right|\\
	&= E\left|\frac{1}{t_r}\int_0^{t_r}  g(\hat W^r(u+t)) du - \frac{1}{t_r}\int_0^{t_r}  g(\hat W^r(u)) du\right|\\
	&\le \frac{2\|g\|_{\infty}}{t_r} \to 0 .
\end{align*}
The equality in \eqref{eq:geq1129} now follows on using the convergence of $\theta^r_{t_r}\to \theta$ and applying continuous mapping theorem. This completes the proof of the theorem.
\end{proof}

\section{Proofs of Theorems \ref{thm:thm6.5} and \ref{thm:thm6.5disc}.}
\label{sec:pfsmainthms}

Recall from \eqref{eq:eq942} the cost function in the EWF, namely $\clc$. 

{\bf Proof of Theorem \ref{thm:thm6.5}.}
	From Theorem \ref{thm:discCostInefBnd} and noting that $h\cdot \hat Q^r(t) \ge \clc(\tilde W^r(t))$ a.s., we have
	$$E\frac{1}{t_r}\int_{0}^{t_r} |h\cdot \hat{Q}^{r}(t) - \clc(\tilde W^r(t))| dt \le B r^{\alpha -1/2}(1+ |\hat q^r|^2).$$
	Next, from Theorem \ref{thm:restCostJobOrd} we see that $\clc$ is a Lipschitz function. Let $L_{\clc}$ denote the corresponding Lipschitz constant.
	Since $M^r\to M$, we can find $\eta_r \in (0,\infty)$ such that $\eta_r\to 0$ as $r\to \infty$ and
	\begin{equation}
		\label{eq:eqworkldfi}
		|\tilde W^r(t) - \hat W^r(t)| \le \eta_r |\hat Q^r(t)| \mbox{ for all } t \ge 0,\, r>0.
	\end{equation}
	From Theorem \ref{thm:WmomBnd} it then follows that, as $r\to \infty$
	$$
	E\frac{1}{t_r}\int_{0}^{t_r} |\clc(\tilde W^r(t)) - \clc(\hat W^r(t))| dt \le L_{\clc} \eta_r \frac{1}{t_r}\int_{0}^{t_r} E|\hat Q^r(t)| dt \to 0.$$
	Thus in order to complete the proof it suffices to show that
	\begin{equation}\label{eq:maintpt129}
		\frac{1}{t_r}\int_{0}^{t_r} \clc(\hat W^r(t)) \to \int \clc(w)\pi(dw), \mbox{ in } L^1, \mbox{ as } r \to \infty .
	\end{equation}
From Theorems \ref{thm:occMeasTight} and \ref{thm:limitMeasProp}, for every $L\in (0,\infty)$,
$$	\frac{1}{t_r}\int_{0}^{t_r} \clc_L(\hat W^r(t)) \to \int \clc_L(w)\pi(dw), \mbox{ in } L^1, \mbox{ as } r \to \infty $$
where $\clc_L(w) \doteq \clc(w)\wedge L$ for $w \in \RR_+^I$.
Also, from linear growth of $\clc$ and Theorem \ref{thm:WmomBnd}, as $L\to \infty$,
$$\sup_{r>0} \frac{1}{t_r}\int_{0}^{t_r} E|\clc(\hat W^r(t)) - \clc_L(\hat W^r(t))| dt \le \frac{1}{L} \sup_{r>0} \frac{1}{t_r}\int_{0}^{t_r} E\clc^2(\hat W^r(t))  dt \to 0.$$
Theorem \ref{thm:limitMeasProp} and Fatou's lemma also show that $\int \clc(w) \pi(dw)<\infty$. Combining this with the above two displays we now have 
\eqref{eq:maintpt129} and the result follows. \qed \\ \ \\

We now prove the convergence of the discounted cost.  Proof is a simpler version of the argument in the proof of Theorem \ref{thm:thm6.5} and therefore we omit some details.\\ \ \\

{\bf Proof of Theorem \ref{thm:thm6.5disc}.}
Minor modifications of the proof of Theorem \ref{thm:occMeasTight} together with Theorem \ref{thm:finTimeConvToRBM} show that for any $T<\infty$
$\hat W^r$ converges in $D([0,T]: \RR_+^I)$ to $\check{W}^{w_{0}}$.  Thus using continuity of $\clc$, for every $L\in (0,\infty)$ and $\clc_L$ as in the proof of Theorem \ref{thm:thm6.5}, for every
$T<\infty$,
\begin{equation*}
\lim_{r\rightarrow \infty }E\left[ \int_{0}^{T}e^{-\theta t}\clc_{L}\left( \hat{%
W}^{r}(t)\right) dt\right] =E\left[ \int_{0}^{T}e^{-\theta t}\clc_{L}\left( 
\check{W}^{w_{0}}(t)\right) dt\right] \text{.}
\end{equation*}%
From Theorem \ref{thm:WmomBnd} we have, as $L\to \infty$,
$$
\sup_{r>0} E \int_{0}^{\infty}e^{-\theta t}|\clc( \hat{W}^{r}(t))-\clc_{L}( \hat{%
W}^{r}(t))| dt \le \frac{1}{L } \sup_{r>0} \int_{0}^{\infty}e^{-\theta t} E\clc^2( \hat{W}^{r}(t)) dt \to 0$$
From Theorem \ref{thm:WmomBnd} we also see that as $T\to \infty$
$$\sup_{r>0} \int_{T}^{\infty}e^{-\theta t} E\clc( \hat{W}^{r}(t)) dt \to 0, \; \int_{T}^{\infty}e^{-\theta t} E\clc(\check{W}^{w_{0}}(t))dt \to 0.$$
Using the fact that $E \int_{0}^{\infty}e^{-\theta t}\clc( \check{W}^{w_{0}}(t)) dt <\infty$ it then follows that for every $T\in (0,\infty)$
$$E \int_{0}^{\infty}e^{-\theta t} h \cdot\hat{Q}^{r}(t) dt \to E \int_{0}^{\infty}e^{-\theta t}\clc(\check{W}^{w_{0}}(t))dt .$$

The result follows. \hfill \qed

\setcounter{equation}{0}
\appendix
\numberwithin{equation}{section}
\section{Large Deviation Estimates for Poisson Processes}

The following result gives classical exponential tail bounds for Poisson processes. For the proof of the first estimate we refer the reader
to \cite{kurtz1978strongapp} while the second result is a consequence of \cite[Section 4.1l1
3, Theorem 5]{LipShibook}.

\begin{theorem}
\label{thm:LDP}Let $N^{r}(t)$ be a Poisson process with rates $\lambda ^{r}$
such that $\lim_{r\rightarrow \infty }\lambda ^{r}=\lambda \in (0,\infty) $. \ Then for any 
$\epsilon \in (0,\infty)$  there exist  $B_{1},B_{2}, R \in (0,\infty)$
  such that for all $0<\sigma <\infty $ and $%
r\geq R$ we have 
\begin{equation*}
P\left( \sup_{0\leq t\leq 1}\left\vert \frac{N^{r}(\sigma t)}{\sigma }%
-\lambda ^{r}t\right\vert >\epsilon \right) \leq B_{1}e^{-\sigma B_{2}}
\end{equation*}
and for all $T \in (0,\infty)$
\begin{equation*}
P\left( \sup_{0\leq t\leq T}\left\vert N^{r}(r^{2}t)-r^{2}t\lambda^{r}\right\vert \geq \epsilon rT\right) \leq B_{1}e^{-B_{2}T}
\end{equation*}%
\end{theorem}

\section*{Acknowledgement}
This research has been partially supported by the National Science Foundation (DMS-1305120), the Army Research Office (W911NF-14-1-0331) and DARPA (W911NF-15-2-0122).


\end{document}